\documentclass[a4paper]{article}

\usepackage[english]{babel}
\usepackage[utf8x]{inputenc}
\usepackage[T1]{fontenc}

\usepackage[a4paper,top=2cm,bottom=2.5cm,left=2.3cm,right=2.3cm,marginparwidth=1.75cm]{geometry}

\usepackage[colorinlistoftodos]{todonotes}
\usepackage{hyperref}
\hypersetup{colorlinks=true, allcolors=blue}
\usepackage[all]{xy}
\usepackage[affil-it]{authblk}
\usepackage{amssymb}
\usepackage{cancel}
\usepackage{epstopdf}
\usepackage{verbatim}
\usepackage{amsmath,amscd}
\usepackage{graphicx}
\usepackage{subcaption}
\usepackage{amsthm}
\usepackage{tikz}
\usetikzlibrary{cd}
\usepackage{pgf}
\usepackage{mathrsfs}
\usetikzlibrary{arrows}
\usepackage[toc,page]{appendix}
\usepackage{caption}
\usepackage[colorinlistoftodos]{todonotes}

\newtheorem{theorem}{Theorem}[section]
\newtheorem{prop}[theorem]{Proposition}

\makeatletter
\newtheorem*{rep@theorem}{\rep@title}
\newcommand{\newreptheorem}[2]{%
\newenvironment{rep#1}[1]{%
 \def\rep@title{#2 \ref{##1}}%
 \begin{rep@theorem}}%
 {\end{rep@theorem}}}
\makeatother

 \newreptheorem{theorem}{Theorem}

\newtheorem{lemma}[theorem]{Lemma}

\theoremstyle{definition}
\newtheorem{mydef}[theorem]{Definition}
\newtheorem{notation}[theorem]{Notation}
\newtheorem{rem}[theorem]{Remark}
\newtheorem{exa}[theorem]{Example}
\newtheorem{cor}[theorem]{Corollary}

\newenvironment{proofMainThm1}{\textit{Proof of Theorem \ref{MainThm1}:}}{\hfill$\square$}

\title{{\bf{Invariants for trees of non-archimedean polynomials and skeleta of superelliptic curves}}}
\author{Paul Alexander Helminck}
\affil{Swansea University, Department of Mathematics}

\begin{document}
\maketitle

\definecolor{qqqqff}{rgb}{0,0,1}

\begin{abstract}
In this paper we generalize the $j$-invariant criterion for the semistable reduction type of an elliptic curve to superelliptic curves $X$ given by $y^{n}=f(x)$. We first define a set of tropical invariants for $f(x)$ using symmetrized Pl\"{u}cker coordinates and we show that these invariants determine the tree associated to $f(x)$. 
We then prove that this tree  
completely determines the reduction type of $X$  
for $n$ that are not divisible by the residue characteristic. The conditions on the tropical invariants that distinguish between the different types are given by half-spaces as in the elliptic curve case. These half-spaces arise naturally as the moduli spaces of certain Newton polygon configurations. We give a procedure to write down their equations and we illustrate this by giving the half-spaces for polynomials $f(x)$ of degree $d\leq{5}$.






\end{abstract}

\section{Introduction}

Let $X$ be a smooth, proper, irreducible curve over a complete algebraically closed non-archimedean field $K$. The Berkovich analytification $X^{\mathrm{an}}$ of $X$ then contains a canonical subgraph often known as the minimal Berkovich skeleton of $X$. 
For elliptic curves, there are two possibilities for the minimal skeleton: it is either a cycle or a vertex of genus $1$. These two options are characterized by the $j$-invariant of the elliptic curve $E$, in the sense that $E^{\mathrm{an}}$ has a cycle if and only if $\mathrm{val}(j)<0$. Furthermore, if $E^{\mathrm{an}}$ has a cycle then the length of this cycle is $-\mathrm{val}(j)$. 
Our goal in this paper is to give similar criteria for superelliptic curves, which are given by equations of the form $y^{n}=f(x)$. We will limit ourselves to the case where $f(x)$ is separable, the other cases are similar. Also, we will assume that $n$ is coprime to the residue characteristic, since we can then express the skeleton in terms of the tree associated to $f(x)$.  

To find the skeleton, we first study the combinatorics behind the roots of 
$f(x)$. These roots determine a tree in the Berkovich analytification $\mathbf{P}^{1,\mathrm{an}}$ and it is well-known 
that this tree is completely determined by its (affine) tropical Pl\"{u}cker coordinates. That is, if we write $\alpha_{i}$ for the roots of $f(x)$ then the tree is determined by the matrix $D=(d_{i,j})$, where 
\begin{equation}
d_{i,j}=\mathrm{val}(\alpha_{i}-\alpha_{j}). 
\end{equation}
Our first goal is to express this tree more invariantly in terms of the coefficients of $f(x)$.  
To that end, we introduce a set of tropical invariants for $f(x)$, which are the valuations of certain symmetric functions in the $(\alpha_{i}-\alpha_{j})^2$. 
 %
Our first main theorem is as follows: 
\begin{theorem}
\label{MainThm1}
Let $f(x)$ and $g(x)$ be two separable polynomials in $K[x]$. 
Then the trees corresponding to $f(x)$ and $g(x)$ are isomorphic if and only if their tropical invariants coincide. 
\end{theorem}

We can thus completely recover the isomorphism class of the tree of a polynomial from its tropical invariants. Moreover, the proof shows that we only need finitely many invariants to reconstruct this tree.    
We illustrate Theorem \ref{MainThm1} in Section \ref{SectionExamplesTrees} by determining the trees of all polynomials of degree $d\leq{5}$. For $d=5$ for instance, there are $18$ different cases to consider, giving $12$ different phylogenetic types. The conditions on the tropical invariants 
are given by rational half-spaces as in the case of elliptic curves. They arise in this paper as equations that describe moduli of Newton polygons, see Section \ref{SectionExamplesTrees}. 

After the proof of Theorem \ref{MainThm1} in Section \ref{Section2}, we show in Section \ref{Section3} that the tropical invariants allow us to completely recover the semistable reduction type of superelliptic curves. The result is as follows:
\begin{theorem}\label{MainThm2}
Let $X_{n}$ be the superelliptic curve defined by $y^{n}=f(x)$, where $f(x)$ is a separable polynomial. 
Then for any 
$n$ satisfying $\mathrm{gcd}(n,\mathrm{char}(k))=1$, the minimal weighted metric graph $\Sigma(X_{n})$ of $X_{n}$ is completely determined  
by the tropical invariants of $f(x)$.  
\end{theorem} 
The main ingredients in the proof of this theorem are Theorem \ref{MainThm1} and the non-archimedean slope formula for the piecewise linear function $-\mathrm{log}|f|$ on the Berkovich analytic space $\mathbf{P}^{1,\mathrm{an}}\backslash{\mathrm{Supp}(f)}$. 
The piecewise linear function $-\mathrm{log}|f|$ can be found in practice 
using some elementary potential theory, we give explicit examples of this in Section \ref{SectionPotentialReduction}. Here we also give formulas for the edge lengths and the genera of the vertices occurring in $\Sigma(X_{n})$. 








The structure of the paper is as follows. We start by defining marked tree filtrations, which give a function-theoretic way of looking at rooted metric trees. We then define the tropical invariants in Section \ref{SectionQuasiInvariants} using the concept of edge-weighted graphs. 
In Section \ref{SectionTruncatedStructures}, we use this procedure to assign an invariant to a subtree and we show that we can predict the valuations of these invariants in certain cases. 
We then use this to give our proof of Theorem \ref{MainThm1}. After this, we determine the polyhedral structure of the tree space for polynomials of degree $d=3,4,5$.

In Section \ref{Section3}, we study superelliptic curves and their skeleta. We prove Theorem \ref{MainThm2} and we give a criterion for (potential) good reduction in Section \ref{SectionPotentialReduction}. We finish the paper by using the polyhedral structure of the tree spaces for $d=3,4,5$ to give a classification of the skeleta of superelliptic curves determined by $y^{n}=f(x)$, where $\mathrm{deg}(f(x))=d$ and $\mathrm{gcd}(n,\mathrm{char}(k))=1$. As a special case, we obtain complete descriptions of the semistable reduction types of genus $2$ curves (similar to the one in \cite{liu}) and of {\it{Picard curves}}, which are genus $3$ curves of the form $y^{3}=f(x)$ with $\mathrm{deg}(f(x))=4$. 

\subsection{Connections to the existing literature}

The criterion using $\mathrm{val}(j)$ for the semistability of elliptic curves has been known for quite some time, it goes back at least to the work of Andr\'{e} N\'{e}ron, see \cite[Page 100]{Neron2}. For curves of genus two, a criterion in terms of Igusa invariants was given in \cite{liu}, this was put into a tropical framework in \cite{Igusa}. The fact that the reduction type of a superelliptic curve defined by $y^{n}=f(x)$ is related (in residue characteristics not dividing $n$) to the roots of $f(x)$ seems to have been known for some time. For instance, in \cite[Remarque 1]{liu} we find the following statement
\begin{center}
{\flushleft{
{\it{
Pour une courbe $C$ sur $K$ d\'{e}finie par une \'{e}quation $y^{n}=P(x)$, la connaissance des racines (avec multiplicit\'{e}) de $P(x)$ d\'{e}termine la courbe $\mathcal{C}_{s}$, si $\mathrm{car}(k)$ ne divise pas $n$ (voir [Bo] pour le cas $n=2$). Mais dans la pratique, il n'est pas toujours ais\'{e} de trouver ces racines.
}}
}}
\end{center}
The main contribution of this paper is that this "connaissance des racines" is removed using tropical invariants. Explicitly, the invariants determine the tree associated to $f(x)$ up to isomorphism and this determines the structure of $\mathcal{C}_{s}$. A related result appears in \cite{tropabelian} and \cite{supertrop}. There it was shown that the structure of $\mathcal{C}_{s}$ (in the discretely valued case) can be determined from the  {\it{marked}} tree associated to $f(x)$. The latter in turn is determined by the relative valuations of the roots $d_{i,j}=\mathrm{val}(\alpha_{i}-\alpha_{j})$. In this paper we show that we can find symmetric analogues of the $d_{i,j}$, the tropical invariants of $f(x)$, that recover the {\it{unmarked}} tree associated to $f(x)$. 

In the meantime there has been progress in obtaining generalizations of the tropical elliptic curve criterion by many other authors. In \cite{LiuSmoothPlaneQuartics}, reduction types of smooth quartics were studied using Dixmier-Ohno invariants. Here they used ideas from a paper by Seshadri \cite{SESHADRI1977} on invariant theory over arbitrary commutative rings. In particular, they give a criterion for good reduction of non-hyperelliptic genus $3$ curves in terms of these invariants. In \cite{bouw_wewers_2017}, a similar criterion for good reduction of Picard curves of the form $y^{3}=f(x)$ with $f$ separable is given. 
For curves $C$ of genus $3$ that admit a Galois morphism $C\to\mathbf{P}^{1}$ with Galois group $\mathbf{Z}/2\mathbf{Z}\times\mathbf{Z}/2\mathbf{Z}$, necessary conditions for the reduction types are given in \cite{WomeninNTReductionGen3}. 
They use the tame simultaneous semistable reduction theorem (in a slightly weaker form, see Theorem \ref{SimultaneousSemSta} 
for the more general version) together with explicit local calculations. Since these fall under the category of solvable coverings, it might be interesting to investigate their geometric properties as in \cite{tropabelian}. In the latter, explicit local and global criteria for reconstructing skeleta were given. These were then used 
to study $S_{3}$-coverings of the projective line, giving for instance a new proof of the classical theorem for elliptic curves $y^2=x^3+Ax+B$ using the covering $(x,y)\mapsto{y}$ instead of $(x,y)\mapsto{x}$. 


Theorem \ref{MainThm1} fits into the tropical literature as follows. 
Let $K$ be a complete algebraically closed non-archimedean field. We can then consider the affine space $X$ associated to the ring 
\begin{equation}
C=K[p_{i,j}]=K[\alpha_{i}-\alpha_{j}]\subset{K[\alpha_{i}]}.
\end{equation}
We are interested in tropicalizing the complement of the hyperplanes $p_{i,j}=0$, so we localize by the discriminant $\Delta=\prod_{i<j}(\alpha_{i}-\alpha_{j})^2$. There is then a natural action of $S_{d}$ on $C[\Delta^{-1}]$. Explicitly, it is given by $\sigma(p_{i,j})=p_{\sigma(i)\sigma(j)}$. From this we then also obtain a group action on the tropicalization 
$\mathrm{trop}(C[\Delta^{-1}])$. 

We are now interested in the quotients of $C[\Delta^{-1}]$ and $\mathrm{trop}(C[\Delta^{-1}])$ under the action of $S_{d}$. 
Algebraically, one can find invariants using the methods in \cite[Section 2]{sturmfels2008algorithms}. 
To obtain tropical quotients, we do the following. We first consider the orbit 
of $(\alpha_{i}-\alpha_{j})^2$ under $S_{d}$ and construct a monic polynomial $F(x)\in{C[\Delta^{-1}]}[x]$ whose roots are representatives of this orbit. 
The coefficients of this polynomial are then invariant under $S_{d}$ and the corresponding field has the same transcendence degree as the original field generated by the $(\alpha_{i}-\alpha_{j})^2$, see \cite[Proposition 2.1.1]{sturmfels2008algorithms}. These coefficients do not necessarily generate the ring of invariants, but one might wonder whether these invariants are sufficient to determine the tree type. For $d\leq{4}$ this is true, but for higher $d$ it is not, see Section \ref{SectionExamplesTrees}. 
We thus see that we have to cast our net somewhat wider. In Section \ref{SectionQuasiInvariants}, we consider more subtle weighted symmetrizations, the tropical invariants, which allow us to 
separate orbits. 
We can then consider the map that takes a separable polynomial over $K$ and associates to it the corresponding tropical invariants. We denote this map by $D[\Delta^{-1}](K)\to\mathbf{T}^{m}$, where 
$\mathbf{T}=\mathbf{R}\cup\{\infty\}$ is the tropical affine line. We can then set up an equivalence relation for polynomials that define isomorphic marked tree filtrations. If we denote the set-theoretic quotient by $Z:=D[\Delta^{-1}](K)/\sim$, then Theorem \ref{MainThm1} tells us that we have an induced injective map from $Z\to\mathbf{T}^{m}$. The proof moreover gives an explicit polyhedral complex inside $\mathbf{T}^{m}$ that contains the image of this map. 
 It would interesting to see if the image under this map has some additional structure. 

A subject that is related to the above is the study of the moduli space of $d$ points on the projective line. We loosely view this as the projective marked version of the set-up we had above (although one has to make some modifications to fit it in). By a well-known result \cite{Kapranov1993} of Kapranov over $\mathbf{C}$, we have isomorphisms 
\begin{equation}\label{Isomorphisms}
\overline{\mathcal{M}}_{0,d}\simeq{}(\mathbf{P}^{1})^{d}//\mathrm{PGL}_{2}\simeq{}G(2,d)//T^{d-1}.
\end{equation} 
Here $\overline{\mathcal{M}}_{0,d}$ is the stable compactification of the moduli space $\mathcal{M}_{0,d}$ of $d$ distinct {\it{marked}} points on the projective line, $G(2,d)$ is the Grassmannian of $2$-dimensional planes in $d$-space and $T^{d-1}$ is the $(d-1)$-dimensional torus; the quotients are Chow quotients. There are tropical variants of these isomorphisms as well: if we embed the Grassmannian using the Pl\"{u}cker embedding, then the tropicalization of the open part $G^{0}(2,d)$ corresponding to nonzero Pl\"{u}cker coordinates can be identified with the space of all tree metrics on $d$ points:
\begin{equation}
\mathrm{trop}(G^{0}(2,d))=-\Delta_{\mathrm{tr}},
\end{equation} see \cite[Theorem 4.3.5]{tropicalbook}. Furthermore, if we consider the quotient by the lineality space of $\mathrm{trop}(G^{0}(2,d))$ (which corresponds to the torus action we had earlier), then this space is the tropicalization of $\mathcal{M}_{0,d}$ under a suitable embedding, see \cite[Theorem 6.4.12]{tropicalbook}. In light of the current paper, one can similarly wonder about possible $S_{d}$-quotients of the above objects. Theorem \ref{MainThm1} can be seen as a first step towards understanding these. 

Finally, we give a short account of the various ways that group actions on toric varieties have appeared in the literature. 
For instance, there is the example of a toric variety with a finite group action such that the quotient is not toric, see 
\cite{SwanToric1969}. Here a permutation action of $\mathbf{Z}/47\mathbf{Z}$ on the $47$-dimensional torus was used, together with the fact that $\mathbf{Z}[\zeta_{23}]$ is not a UFD (it is the first such cyclotomic ring). On the other hand, if the group in question is a suitable torus, then the quotient exists as a toric variety, see \cite{KapranovSturmfels1991}. This technique is unfortunately not applicable to the current paper since we are not working with a toric group. 



\subsection{Notation and terminology} 
Throughout this paper, 
we will use the same set of assumptions and notation as in \cite{TropFund1} unless mentioned otherwise. 
This paper in turn is heavily influenced by \cite{ABBR1} and \cite{BPRa1}, so the reader might benefit from a review of these as well. We give a short summary of the most important concepts used in this paper:

\begin{itemize}
\item $K$ is a complete algebraically closed non-archimedean field with valuation ring $R$, maximal ideal $\mathfrak{m}_{R}$, residue field $k$ and nontrivial valuation $\mathrm{val}:K\to\mathbf{R}\cup\{\infty\}$. The symbol $\varpi$ is used for any element in $K^{*}$ with $\mathrm{val}(\varpi)>0$. The absolute value associated to $\mathrm{val}(\cdot{})$ is defined by $|x|=e^{-\mathrm{val}(x)}$, where $e$ is Euler's constant. This constant will not be used in the rest of the paper; we will reserve the letter $e$ for edges instead. 
\item $X$ is a smooth, irreducible and proper curve over $K$. We will often omit these and simply say that $X$ is a curve. Its analytification in the sense of \cite{Berkovich1993} and \cite{berkovich2012} is denoted by $X^{\mathrm{an}}$.
\item A finite morphism of curves $\phi:X\to{Y}$ gives a finite morphism of analytifications $\phi^{\mathrm{an}}:X^{\mathrm{an}}\to{Y^{\mathrm{an}}}$. We say that $\phi$ is residually tame if for every $x\in{X^{\mathrm{an}}}$ in the \'{e}tale locus with image $y\in{Y^{\mathrm{an}}}$,  the extension of completed residue fields $\mathcal{H}(y)\to\mathcal{H}(x)$ is tame. See \cite[Section 2.1]{TropFund1}. 
\item 
%
For $r\in\mathbf{R}$, we define closed disks and annuli 
by $\mathbf{B}(x,r)=\{y\in{K}:\mathrm{val}(x-y)\geq{r}\}$ and $\mathbf{S}(x,r)=\{y\in{K}:0<{}\mathrm{val}(x-y)\leq{r}\}$. Their open counterparts are denoted by $\mathbf{B}_{+}(x,r)$ and $\mathbf{S}_{+}(x,r)$. 
If the center point is $0$, then we denote these by $\mathbf{B}(r)$, $\mathbf{S}(r)$, $\mathbf{B}_{+}(r)$ and $\mathbf{S}_{+}(r)$. 
We also use the notation $\mathbf{S}_{+}(a)$ for an element $a\in{K^{*}}$ with $\mathrm{val}(a)\geq{0}$, which by definition is $\mathbf{S}_{+}(\mathrm{val}(a))$. Furthermore, we will use the same notation to denote the corresponding Berkovich analytic subspaces, see \cite[Section 3]{ABBR1} and \cite[Section 2]{BPRa1} for more on these. 
\item Semistable vertex sets of curves are denoted by $V(\Sigma)$, where $\Sigma$ is its corresponding skeleton. 
For open annuli $\mathbf{S}_{+}(x,r)$, we denote the skeleton by $e^{0}$. We call these open edges. Similarly, for a closed annulus  $\mathbf{S}(x,r)$ we denote the skeleton by $e$. We call these closed edges. We again refer the reader to \cite[Section 2]{BPRa1} and \cite{ABBR1} for more details. 
\item Any semistable vertex set $V(\Sigma)$ with skeleton $\Sigma$ of a curve $X$ can be enhanced to a metrized complex of $k$-curves, see \cite{ABBR1} for the terminology. Loosely speaking, the construction is as follows. We start with the metric graph corresponding to $\Sigma$ and we include the data of a residue curve $C_{x}$ for every type-$2$ point in $V(\Sigma)$. We then identify every tangent direction in $\Sigma$ with the corresponding closed point on $C_{x}$. This concept gives a convenient way of expressing the fact that for residually tame coverings of curves, the local Riemann-Hurwitz formulas hold. We will use this in the proof of Theorem \ref{MainThm2}. If we only record the genus of every curve $C_{x}$, then we obtain the weighted metric graph associated to $V(\Sigma)$. 
\item For any curve $X$, there is a minimal metric graph $\Sigma_{\mathrm{min}}\subset{X^{\mathrm{an}}}$ which one obtains from any skeleton $\Sigma$ of $X$ by contracting its leaves. By adding the additional data of the genera of the type-$2$ points in $\Sigma$, we then arrive at the notion of the minimal weighted metric graph of $X$. This is the object of interest to us in Theorem \ref{MainThm2}. We also call this the {\it{reduction type}} of $X$. 



\end{itemize}

\section{Marked tree filtrations and tropical invariants}\label{Section2}

In this section we study trees associated to separable polynomials $f(x)\in{K[x]}$ over a   non-archimedean field. We start by defining the framework of marked tree filtrations, which gives a more function-theoretic way of looking at rooted marked trees, see Section \ref{SectionTreeFiltrations}. We define tropical invariants in Section \ref{SectionQuasiInvariants} and in Section \ref{SectionTruncatedStructures} we use this to assign invariants to maximal subtrees. We then use these to prove 
Theorem \ref{MainThm1}. We show in Section \ref{SectionExamplesTrees} that we can write down polyhedral equations for each filtration type using Newton polygon half-spaces. We finish this section by giving explicit equations for the 
space of trees associated to polynomials of degree $d\leq{5}$. 



\subsection{Marked tree filtrations}\label{SectionTreeFiltrations}

Let $T$ be a metric graph as in \cite[Definition 2.2]{ABBR1}. 
If the first Betti number of $T$ is zero, then we say that $T$ is a metric tree. 
 A leaf of a metric tree is a point of valency $1$. We assume throughout that the set $V_{\infty}(T)$ of infinite vertices of a metric tree $T$ is equal to the leaf set $L(T)$ of $T$.  
 A rooted or marked metric tree is a metric tree with a distinguished leaf which we denote by $\infty$. 
 Let $A$ be a finite set. We say that $T$ is a rooted metric tree on $(A,\infty)$ 
 if there is a bijection $A\cup\{\infty\}\to{L(T)}$ which takes $\infty$ to $\infty$.  
We now introduce the concept of a marked tree filtration, which will turn out to be equivalent to that of a rooted metric tree.

\begin{mydef}{\bf{[Marked tree filtrations]}}\label{TreeFiltrations}
Let $A$ be a finite set and let
\begin{equation}
\phi:A\times\mathbf{R}\to{\mathbf{R}}
\end{equation}
be any function. We say that $\phi$ defines a marked tree filtration on $A$ if the following hold:
\begin{enumerate}
\item There is a $c_{0}\in\mathbf{R}$ such that the function $\phi(x,c_{0})$ obtained by restricting $\phi$ to $A\times\{c_{0}\}$ is constant.
\item Suppose that $c_{2}>c_{1}$. If $x,y\in{A}$ satisfy $\phi(x,c_{2})=\phi(y,c_{2})$, then they also satisfy $\phi(x,c_{1})=\phi(y,c_{1})$. 
\item For any sufficiently large $c$, the restricted function $\phi(x,c):A\to\mathbf{R}$ is injective. 
\end{enumerate}

The set $A$ is the set of {\it{finite leaves}} of the marked tree filtration $\phi$. We say that $x,y\in{A}$ are in the same $c$-branch if $\phi(x,c)=\phi(y,c)$. Two functions $\phi$ and $\psi$ are said to give the {\it{same}} marked tree filtration on $A$ if for every $c\in\mathbf{R}$, we have that $\phi(x,c)=\phi(y,c)$ for $x,y\in{A}$ if and only if $\psi(x,c)=\psi(y,c)$. Given two sets $A$ and $B$ with marked tree filtrations $\phi$ and $\psi$, we say that they are isomorphic if there exists a bijection $i:A\to{B}$ such that the marked tree filtration $\psi\circ{}(i,\mathrm{id})$ on $A$ is the same as $\phi$. 

\end{mydef}

\begin{rem}
By combining the first and second condition, we see that for every $c\leq{c_{0}}$, the restriction $\phi(x,c)$ is constant. We view this as the extra leaf corresponding to $\infty$. 
\end{rem}

\begin{rem}
The choice of $\mathbf{R}$ for the target space of $\phi$ in the definition of a marked tree filtration is not strictly necessary: we can take any set whose cardinality is greater than that of $A$.
\end{rem}

\begin{exa}\label{TreeBerkovich}
Let $K$ be a non-archimedean field and let $A=\{\alpha_{1},...,\alpha_{d}\}\subset{K}$ be a set of pairwise distinct elements. 
We can assign a marked tree filtration to $A$ as follows. For every pair $(\alpha_{i},\alpha_{j})\in{A^{2}}$ not in the diagonal, we define 
\begin{equation}
d_{i,j}=\mathrm{val}(\alpha_{i}-\alpha_{j}).
\end{equation} 
We then let $\phi:A\times{\mathbf{R}}\to\mathbf{R}$ be any function such that $\phi(\alpha_{i},c)=\phi(\alpha_{j},c)$ if and only if $d_{i,j}\geq{c}$. It is not too hard to see that this defines a marked tree filtration.

\end{exa}

We now assign a marked tree filtration to a rooted metric tree $T$ on $A$. We first construct the root vertex of $T$. For any $i\in{A}$, consider the unique path $\ell_{i,\infty}$ from $i$ to $\infty$. The intersection of these paths is a line segment and we denote the endpoint not equal to $\infty$ by $v_{\infty}$. This is the root vertex of $T$.   
For every pair $(i,j)\in{A^2}$, we now consider the triple $(\ell_{i,j},\ell_{i,\infty},\ell_{j,\infty})$ consisting of the unique paths from $i$ to $j$, $i$ to $\infty$ and $j$ to $\infty$. The intersection of these three paths is a finite point $\mathrm{join}(i,j)$ on $T$. 
Let $d_{i,j}$ be the distance from $v_{\infty}$ to $\mathrm{join}(i,j)$.  Note that this is finite since $\mathrm{join}(i,j)$ and $v_{\infty}$ are not leaf vertices. 
For any positive real number $c$, 
we define $\phi(x,c):A\to\mathbf{R}$ to be any function such that $\phi(i,c)=\phi(j,c)$ if and only if $d_{i,j}\geq{c}$. 
A quick check shows that the axioms in Definition \ref{TreeFiltrations} are automatically satisfied for these functions. Conversely, we can reconstruct a rooted metric tree on $A$ from a marked tree filtration $\phi$ as follows. For every $i\in{A}$, we take an infinite line segment $L_{i}=\mathbf{R}\cup\{\infty\}\cup\{-\infty\}$. We think of the point at positive infinity as corresponding to the leaf $i\in{A}$, the point at negative infinity is the point corresponding to the leaf $\infty$ of $T$. We now define an equivalence relation on the disjoint union $L=\bigsqcup{L_{i}}$ as follows. We say that $(i,c)\sim(j,c)$ for a $c\in\mathbf{R}$ if and only if $\phi(i,c)=\phi(j,c)$. The quotient $L/\sim$ is then automatically a rooted metric tree with finite leaf set $A$ and marked tree filtration the same as $\phi$. 
For the remainder of this paper, we will use the framework of marked tree filtrations to study rooted metric trees.  

\begin{exa}\label{TreeBerkovich2}{\bf{[Finite Berkovich trees]}}
Consider the marked tree filtration from Example \ref{TreeBerkovich}. The associated metric tree has a canonical interpretation in terms of Berkovich spaces, as we now explain. We assume for simplicity that $\mathrm{val}(\alpha_{i})\geq{0}$. We first consider the $\alpha_{i}$   
as type-$1$ points of the $K$-analytic space 
$\mathbf{P}^{1,\mathrm{an}}$. These type-$1$ points can be considered as the seminorms corresponding to degenerate disks $\mathbf{B}(\alpha_{i},\infty)$. We can now connect these type-$1$ points to the type-$2$ point of the closed disk 
$\mathbf{B}(0)$ as follows. 
We take the path given by the closed disks $\mathbf{B}(\alpha_{i},r)$ for $0\leq{r}<\infty$. Two of these paths meet at $\mathbf{B}(\alpha_{i},\mathrm{val}(\alpha_{i}-\alpha_{j}))=\mathbf{B}(\alpha_{j},\mathrm{val}(\alpha_{i}-\alpha_{j}))$. 
We now similarly construct a path from the point induced by $\mathbf{B}(0)=\mathbf{B}(\alpha_{i},0)$ to $\infty$ using $\mathbf{B}(0,r)$, where $r\in(-\infty,0]$. We then easily see that these paths give a metric tree isomorphic to the one we created from the marked tree filtration. We will use this identification throughout the paper without further mention.   
 
\end{exa}

\begin{mydef}
{\bf{[Subtrees]}}
Let $\phi:A\times\mathbf{R}_{\geq{0}}\to\mathbf{R}$ be a marked tree filtration and let $i:S\to{A}$ be an injection. We say that $i$ defines a marked subtree of $\phi$.  
\end{mydef}

Note here that the induced function $\phi_{S}:S\times\mathbf{R}\to\mathbf{R}$ defined by $\phi_{S}(x,c):=\phi(i(x),c)$ 
automatically satisfies the conditions of Definition \ref{TreeFiltrations}. We thus see that a subtree comes with a canonical marked tree filtration. We will also say that the marked tree filtration $\phi_{S}$ is a subtree of $\phi$.  

\begin{mydef}
{\bf{[Truncated structures]}}
Let $\phi$ be a marked tree filtration on $A$ and let $i:S\to{A}$ be an injection giving a marked subtree of $\phi$ with induced filtration $\phi_{S}$. For any positive real number $c$, we say that $i$ is $c$-trivial if the restricted function $\phi_{S}(x,c):S\to\mathbf{R}$ is injective. 
For any $c\in\mathbf{R}$, there is a maximal $n_{0}$ such that there is a subtree on $n_{0}$ leaves with trivial $c$-structure. We refer to this $n_{0}$ as the number of branches at height $c$. A subtree of maximal cardinality for a given $c$ is called a maximal $c$-trivial subtree. This maximal $c$-trivial subtree is uniquely defined up to a permutation of leaves in the same $c$-branch. This permutation induces an isomorphism of the marked tree filtrations of the maximal $c$-trivial subtrees. 
We refer to any such maximal $c$-trivial tree as the $c$-truncated structure of $\phi$. 
We say that two marked tree filtrations are isomorphic up to height $c$ if their $c$-truncated structures are isomorphic. This is independent of the maximal $c$-trivial subtree chosen. 
\end{mydef}

\begin{mydef}
{\bf{[Branching]}}
Let $\phi$ be a marked tree filtration. We say that $\phi$ has branching at $c\in\mathbf{R}$ if for every $\epsilon>0$, the number of branches at height $c-\epsilon$ is different from the number of branches at $c+\epsilon$. For every marked tree filtration, there are only finitely many heights where branching occurs. We call these the branch heights. 
\end{mydef}

\begin{exa}\label{ExampleTreesFiveLeaves}
Consider the three marked tree filtrations on five leaves indicated in Figure \ref{TreesFiveLeaves}. Here the infinite ray to $\infty$ is omitted.  
\begin{figure}
\centering
\includegraphics[height=8cm]{{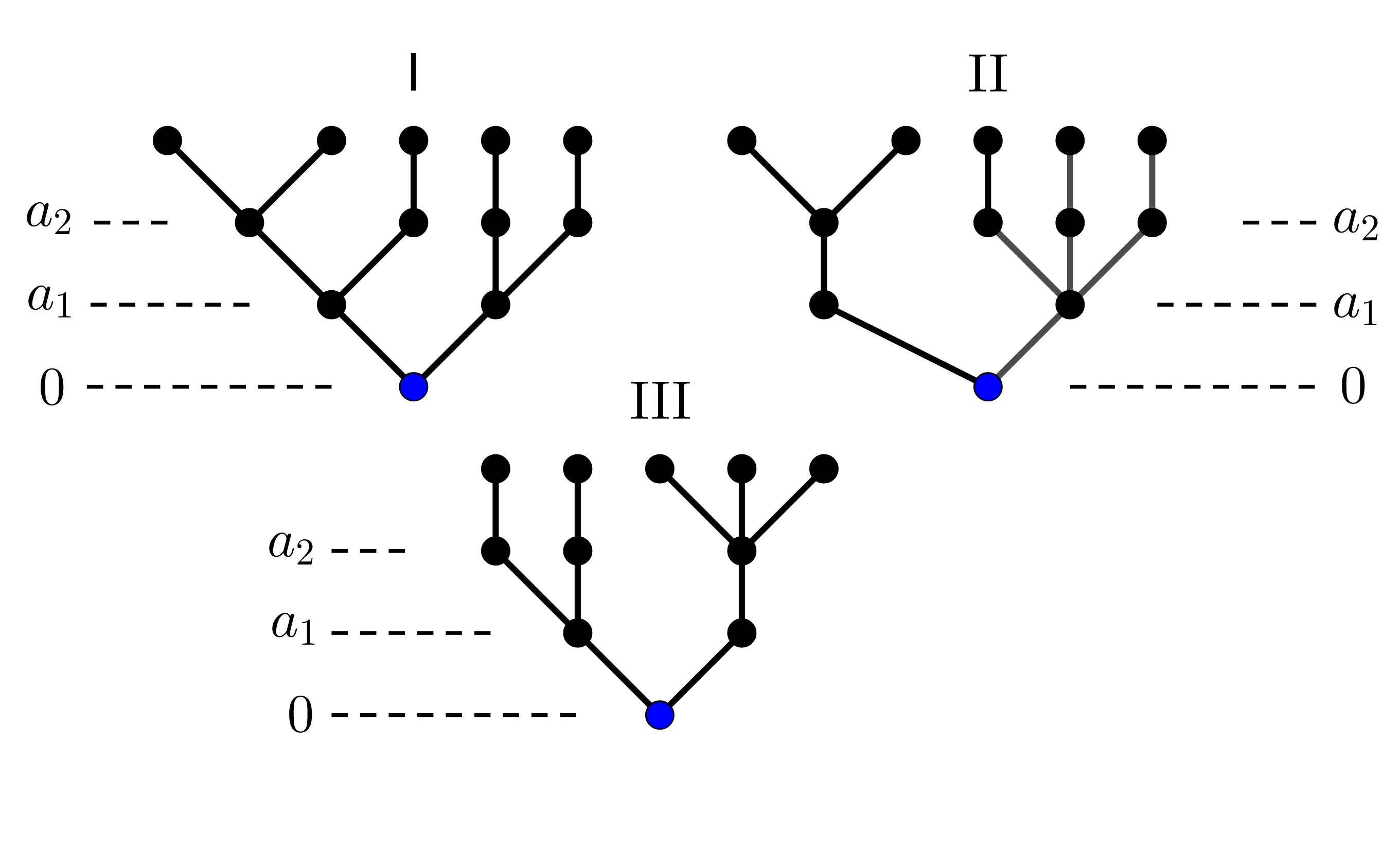}}
\caption{\label{TreesFiveLeaves}
The three marked tree filtrations on five leaves in Example \ref{ExampleTreesFiveLeaves}. The branch heights are the $a_{i}$. 
} 
\end{figure}
 The branch heights are $0$, $a_{1}$ and $a_{2}$. For $c>a_{2}$, we have that the maximal $c$-trivial subtree is just the tree itself. For $c=0$, we have that any maximal $c$-trivial subtree is given by restricting to a single leaf. For $c\in(0,a_{1}]$ in all cases the maximal $c$-trivial subtree is given by restricting to two leaves from the different branches. If $c\in(a_{1},a_{2}]$, then in the first two cases we have maximal $c$-trivial subtrees of order $4$ and in the last case we have maximal $c$-trivial subtrees of order $3$. Note however that all of these subtrees are non-isomorphic as marked tree filtrations. We thus see that these trees are isomorphic up to height $a_{1}$, but not up to any greater height. 
\end{exa}

\begin{rem}\label{ExtendingIsomorphisms}
Let $\phi$ be a marked tree filtration and suppose that $\phi$ has no branching at $c$. A maximal $c$-trivial subtree then naturally extends to a maximal $c_{0}$-trivial subtree, where $c_{0}$ is the smallest branching height greater than $c$. Similarly, let $\phi$ and $\phi'$ be two marked tree filtrations with no branching at $c$ and let $c_{0}$ and $c'_{0}$ be their first branching heights greater than $c$ (if there is no further branching, set $c_{0}=c$ or $c'_{0}=c$). If $\phi$ and $\phi'$ are isomorphic up to height $c$, then they are also isomorphic up to height $\mathrm{min}\{c_{0},c'_{0}\}$. We will use this in the proof of Theorem \ref{MainThm1}. 


\end{rem}

\begin{mydef}\label{CombinatorialStructure}
{\bf{[Filtration structure]}} Let $\phi$ and $\phi'$ be two marked tree filtrations on $A$ with branching heights $0=a_{0}<a_{1}<...<a_{m}$ and $0=a'_{0}<a'_{1}<...<a'_{m'}$. We say that $\phi$ and $\phi'$ define the same filtration structure if
\begin{enumerate}
\item $m=m'$.
\item For every $i$, and for every $c\in{(a_{i},a_{i+1})}$ and $d\in(a'_{i},a'_{i+1})$, we have that $\phi(x,c)=\phi(y,c)$ if and only if $\phi'(x,d)=\phi'(x,d)$.
\end{enumerate}
We say that the filtration structures of the marked tree filtrations $\phi$ and $\phi'$ 
are isomorphic 
if the above hold for filtrations $\phi_{0}$ and $\phi'_{0}$ with $\phi\sim{\phi_{0}}$ and $\phi'\sim{\phi'_{0}}$. Here we also say that $\phi$ and $\phi'$ have the same {\it{filtration type}}.  
\end{mydef}

To put it differently, two marked tree filtrations $\phi$ and $\phi'$ have the same filtration type if 
we can transform $\phi$ into $\phi'$ by changing the lengths of the finite edges. 
We now define one more type of structure. Suppose that we have a marked tree filtration $\phi$ 
with rooted metric tree $T$. We can associate a canonical ordinary phylogenetic tree $G_{T}$ to $T$ (without any lengths on the edges) whose leaf set is $A\cup\{\infty\}$. 




\begin{mydef}
{\bf{[Phylogenetic structure]}}
Let $\phi$ be a marked tree filtration. We say that the finite marked tree $G_{T}$ is the phylogenetic tree associated to $\phi$. 
Furthermore, we say that two marked tree filtrations $\phi$ and $\phi'$ have the same phylogenetic structure if there is an isomorphism $G_{T}\to{G_{T'}}$ sending $\infty$ to $\infty'$.     
\end{mydef}

\begin{exa}
Consider the trees in Example \ref{ExampleTreesFiveLeaves}. 
\begin{figure}
\centering
\includegraphics[width=8cm]{{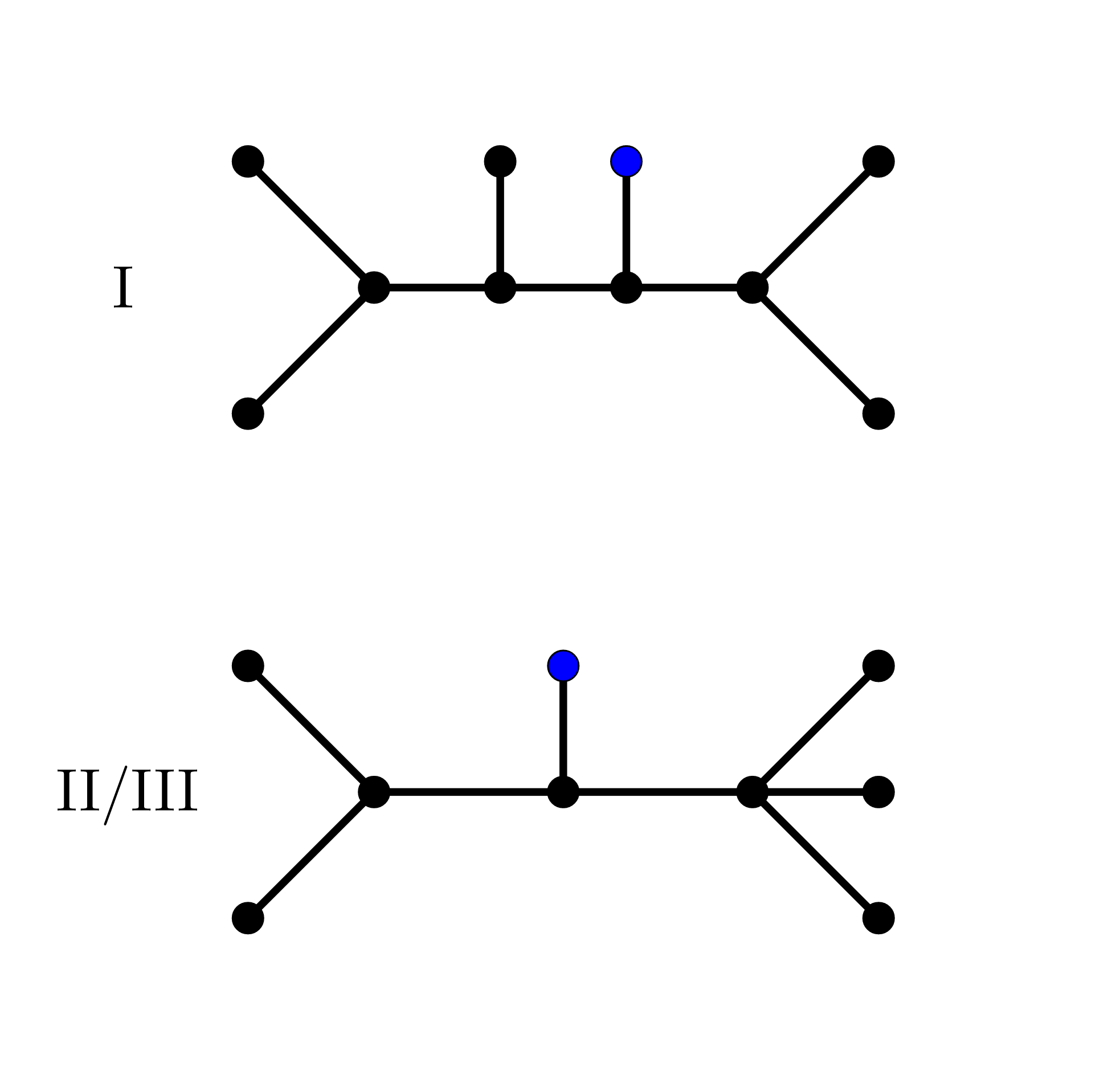}}
\caption{\label{PhylogeneticFiveLeaves}
The phylogenetic types of the trees in Example \ref{ExampleTreesFiveLeaves}. The blue leaf is the leaf corresponding to $\infty$. 
} 
\end{figure}
Their phylogenetic trees are given in Figure \ref{PhylogeneticFiveLeaves}. Note that the second and third tree define isomorphic phylogenetic types, even though their filtration types are different. 
In Section \ref{PolynomialsDegreeFive} we will see that this has consequences for our formulas of the edge lengths.  
\end{exa}

We will see in Section \ref{Section3} that the underlying graph of the skeleton of the curve $y^{n}=f(x)$ only depends on the phylogenetic structure of the tree associated to the roots of $f(x)$. To find the lengths, we will need the marked tree filtration associated to the roots of $f(x)$. 





\subsection{Algebraic invariants}\label{SectionQuasiInvariants}

Let $C=\mathbf{Z}[\alpha_{i}]$ be the polynomial ring in the variables $\alpha_{1},...,\alpha_{d}$. The group $S_{d}$ naturally acts through ring homomorphisms on this ring and this also gives an action on $C[x]$. We consider this action for the polynomial
\begin{equation}
f:=\prod_{i=1}^{d}(x-\alpha_{i}).
\end{equation}
Since every $\sigma$ acts as a ring homomorphism, we see that $\sigma(f)=f$. We now define the elementary symmetric polynomials $a_{i}$ through the equation
 $f=\sum_{i=0}^{d}(-1)^{i}a_{i}x^{i}$. Since $\sigma$ acts trivially on $f$, it also acts trivially on every $a_{i}$, so these are indeed invariant with respect to this action. We then have the following classical result on symmetric polynomials: 
\begin{lemma}\label{ElementarySymmetric}
$\mathbf{Z}[\alpha_{1},...,\alpha_{d}]^{S_{d}}=\mathbf{Z}[a_{0},...,a_{d-1}]$.
\end{lemma} 

We use this to construct a set of invariants. Let $K_{d}$ be the complete undirected graph on $d$ vertices. Here we identify the vertex set of $K_{d}$ with $\{\alpha_{1},...,\alpha_{d}\}$. Since $K_{d}$ is a simple graph, there is at most one edge between any two vertices. We write edges in $K_{d}$ as $\{\alpha_{i},\alpha_{j}\}$. Let $G$ be an subgraph of $K_{d}$ 
and let 
\begin{equation}
k:E(K_{d})\to\mathbf{Z}_{\geq{0}}
\end{equation}
be a weight function that is zero if and only if $e$ is not in $E(G)$. 
We refer to a pair $(G,k)$ as an {\it{edge-weighted graph}}.  
We now define 
\begin{equation}\label{PreInvariants}
I_{G,k}:=\prod_{e=\{\alpha_{i},\alpha_{j}\}\in{E(G)}}(\alpha_{i}-\alpha_{j})^{2k(e)}.
\end{equation}
We refer to elements of this form as pre-invariants. We will also write $[ij]=(\alpha_{i}-\alpha_{j})^2$ so that $I_{G,k}=\prod_{e\in{E(G)}}[ij]^{k(e)}$. 
Let $H_{G,k}:=\mathrm{Stab}(I_{G,k})$ be the stabilizer of $I_{G,k}$ under the action of $S_{d}$. Writing $\sigma_{1},...,\sigma_{r}$ for representatives of the cosets of $H_{G,k}$ in $S_{d}$, we then obtain the polynomial 
\begin{equation}
F_{G,k}:=\prod_{i=1}^{r}(x-\sigma_{i}(I_{G,k})).
\end{equation} 
This polynomial is invariant under the action of $S_{d}$. Indeed, $S_{d}$ acts by ring homomorphisms on $C[x]$ and every term $x-\sigma_{i}(I_{G,k})$ is sent to another $x-\sigma_{j}(I_{G,k})$ since the $\sigma_{i}$ are representatives of the cosets of $H_{G,k}$. 
Using Lemma \ref{ElementarySymmetric}, we now find that 
we can express its coefficients in terms of the $a_{i}$. 
 
 \begin{mydef}\label{QuasiInvariantsDefinition}
{\bf{[Algebraic invariants]}}
The polynomial $F_{G,k}$ is the generating polynomial for the pair $(G,k)$. Its coefficients are the algebraic invariants of $f$ with respect to the pair $(G,k)$. 
\end{mydef}

\begin{exa}
We note that the homogenized versions of these coefficients are not necessarily invariant with respect to the natural $\mathrm{SL}_{2}$-action. Indeed, if we consider $n=4$ and $G$ a graph of order $2$, then the only coefficient that is invariant is the constant coefficient, which is the discriminant of $f$. This follows from the criteria in \cite{sturmfels2008algorithms} and \cite{derksen2002computational}: 
 the non-constant coefficients 
are symmetric bracket polynomials that are not regular. The discriminant however is regular, so this does give an invariant. 
\end{exa}

\begin{rem}\label{RemarkStabilizer}
We now interpret the stabilizer $H_{G,k}$ graph-theoretically.   
We define a morphism of edge-weighted graphs $(G_{1},k_{1})\to(G_{2},k_{2})$ to be an injective morphism $\psi:G_{1}\to{G_{2}}$ of graphs such that $k_{2}\circ{\psi}=k_{1}$. An isomorphism of edge-weighted graphs is a morphism of edge-weighted graphs that is bijective. We then consider the set of edge-weighted graphs in $K_{d}$ isomorphic to $(G,k)$. The group $\mathrm{Aut}(K_{d})=S_{d}$ acts transitively on this set and the stabilizer of $(G,k)$ under this action is $H_{G,k}$. Indeed, we automatically have $\mathrm{Stab}(G,k)\subset{H_{G,k}}$, so let us prove the other inclusion. 

To do this, we use the fact that $\alpha_{i}-\alpha_{j}$ is an irreducible element of the unique factorization domain $\mathbf{Z}[\alpha_{i}]$ for every $i\neq{j}$. Suppose that $e=\{\alpha_{i},\alpha_{j}\}\in{E(G)}$ is sent to $\sigma(e)=\{\sigma(\alpha_{i}),\sigma(\alpha_{j})\}\notin{E(G)}$. 
Then one immediately finds that the factor $\sigma(\alpha_{i}-\alpha_{j})=\alpha_{\sigma(i)}-\alpha_{\sigma(j)}$ is not a divisor of $H_{G,k}$, 
a contradiction. We similarly obtain the statement in the case that the factor $\alpha_{i}-\alpha_{j}$ has a certain weight in $I_{G,k}$, we leave the details to the reader. 
Identifying orbits of $(G,k)$ with cosets of $H_{G,k}=\mathrm{Stab}(G,k)$ inside $S_{n}$, we now see that we can think of the 
$\sigma_{i}(I_{G,k})$ as isomorphism classes of $(G,k)$ inside $K_{d}$. 
 This point of view will be used throughout the upcoming sections. We will only be needing the case where $G$ is the complete graph on $n_{1}<d$ vertices and $k$ is some non-trivial weight function. In practice one might also want to use other edge-weighted graphs. 

\end{rem}




\subsubsection{Tropicalizing invariants}\label{SectionTropicalizingInvariants}

Let $C=\mathbf{Z}[\alpha_{i}]$ and $D=\mathbf{Z}[a_{i}]$ be as in the previous section. Consider the discriminant $\Delta=\prod_{i<j}(\alpha_{i}-\alpha_{j})^2$. In terms of the previous section, this is the invariant associated to the complete graph $K_{d}$. We can localize the above rings with respect to $\Delta$ to obtain
\begin{equation}
D[\Delta^{-1}]\to{C[\Delta^{-1}]}.
\end{equation} 
Now let $K$ be a field and let $\psi:D[\Delta^{-1}]\to{K}$ be a $K$-valued point. This corresponds to a separable polynomial $f\in{K[x]}$. It defines a prime ideal $\mathrm{Ker}(\psi)=\mathfrak{p}_{\psi}$ in the spectrum of $D[\Delta^{-1}]$. We now consider the fiber of $\mathfrak{p}_{\psi}$ under the map
\begin{equation}
\mathrm{Spec}(C[\Delta^{-1}])\to\mathrm{Spec}(D[\Delta^{-1}]).
\end{equation} 
The points in this fiber are all $\overline{K}$-rational and each of these corresponds to a labeling of the roots of $f$,  
see \cite[Expos\'{e} V, Proposition 1.1, Page 88]{SGA1} for instance. 

We now assume that $K$ is a complete, non-archimedean field and we let $\psi:D[\Delta^{-1}]\to{K}$ be a $K$-valued point, corresponding to a separable polynomial $f\in{K[x]}$. Note that the valuation on $K$ extends uniquely to $\overline{K}$. We fix an extension of $\psi$ to a $\overline{K}$-valued point of $C[\Delta^{-1}]$ and we denote this homomorphism by $\psi_{C}$. We then have a set of $n$ pairwise distinct elements $\{\psi_{C}(\alpha_{1}),...,\psi_{C}(\alpha_{d})\}$ and this gives a marked tree filtration $\phi$ by Example \ref{TreeBerkovich}. Note that if we choose another extension $\psi'_{C}$ of $\psi$, then this permutes the roots of $f$ in $\overline{K}$. The marked tree filtration $\phi'$ corresponding to $\{\psi'_{C}(\alpha_{1}),...,\psi'_{C}(\alpha_{d})\}$ is thus 
isomorphic to $\phi$.  

\begin{mydef}{\bf{[Marked tree filtration of a polynomial]}} \label{DefinitionTreePolynomial}
Let $\psi_{C}:C[\Delta^{-1}]\to\overline{K}$ be an extension of $\psi:D[\Delta^{-1}]\to{K}$. The marked tree filtration associated to $\{\psi_{C}(\alpha_{1}),...,\psi_{C}(\alpha_{n})\}$ in Example \ref{TreeBerkovich} is the marked tree filtration associated to $\psi_{C}$ and $f$. For any two extensions $\psi_{C}$ and $\psi'_{C}$ of $\psi$, the corresponding marked tree filtrations are isomorphic and we refer to this isomorphism class as the marked tree filtration associated to $f(x)$.    
\end{mydef} 

The above extension $\psi:C[\Delta^{-1}]\to{\overline{K}}$ gives a homomorphism of polynomial rings $C[\Delta^{-1}][x]\to{\overline{K}[x]}$ which we again denote by $\psi$. 
By applying the valuation map  \begin{equation}
\mathrm{val}:{\overline{K}}\to\mathbf{R}\cup\{\infty\}
\end{equation}  
to the coefficients of the polynomials $\psi(F_{G,k})$, we then obtain the tropical invariants.
\begin{mydef}{\bf{[Tropical invariants]}}\label{TropQuasiInvariants}
Consider the coefficients of the polynomial $\psi(F_{G,k})$. Their valuations are the tropical invariants associated to $\psi$ and $(G,k)$.
\end{mydef}

If we know the tropical invariants associated to a $\psi$ and an edge-weighted graph $(G,k)$, then we can recover the valuations of the $\psi(\sigma_{i}(I_{G,k}))$ using the classical Newton polygon theorem, see \cite[Chapter II, Proposition 6.3]{neu} for instance. 
The Newton polygon theorem moreover gives the multiplicities of these roots. 
The theorem we would like to prove is that we can recover the marked tree filtration associated to the roots, up to a permutation, from the tropical invariants. 

To phrase this more precisely, we consider the set $Z$ of all edge-weighted graphs on $K_{n}$ up to isomorphism and we define $x_{G,k}:=\mathrm{deg}(F_{G,k})+1$ to be the number of coefficients of $F_{G,k}$. Let $\mathbf{T}:=\mathbf{R}\cup{\{\infty\}}$ be the tropical affine line. Identifying $\overline{K}$-rational points of $D[\Delta^{-1}]$ with separable polynomials $f\in{\overline{K}[x]}$, we then obtain a map  
\begin{equation}
\mathrm{trop}:D[\Delta^{-1}](\overline{K})\to\prod_{(G,k)\in{Z}}\mathbf{T}^{x_{G,k}}
\end{equation}   
by mapping $f=\sum_{i=0}^{d}a_{i}x^{i}\in{\overline{K}[x]}$ to all its tropical invariants. Note that this last product is infinite. 
\begin{reptheorem}{MainThm1}
{\bf{[Main Theorem]}}
We have $\mathrm{trop}(f)=\mathrm{trop}(g)$ if and only if the marked tree filtrations corresponding to $f$ and $g$ are isomorphic. Furthermore, for any fixed degree $d$ we can find a finite number of edge-weighted graphs such that this holds. 
\end{reptheorem}

The idea of the proof is as follows. We first assign an edge-weighted graph to a maximal $c$-trivial structure of a marked tree filtration. We show that we can predict the minimum of this invariant and that the minimum is attained for another filtration if and only if a certain $c$-trivial structure is present in the marked tree filtration. We then suppose that the marked metric trees corresponding to $f$ and $g$ are not isomorphic and we consider the smallest branch point $c$ after which the $c$-structures are not isomorphic. We then use the invariants assigned above to obtain a contradiction.

\begin{rem}
Calculations in practice show that one does not need special edge-weighted graphs to distinguish between the marked tree filtrations of two polynomials: it suffices to take graphs $G$ with trivial functions $k(\cdot{}):E(K_{d})\to{\mathbf{Z}_{\geq{0}}}$. It might be that Theorem \ref{MainThm1} still holds for trivial edge-weighted graphs, but we have not been able to prove this. The same result using only the graph of order two does not hold, as we will see in Section \ref{SectionExamplesTrees}. 
\end{rem}

\subsection{Invariants for truncated structures}\label{SectionTruncatedStructures}

As before, let $K$ be a complete non-archimedean field. We start with 
a homomorphism $\psi:D[\Delta^{-1}]\to{K}$ giving a separable polynomial $\psi(f)\in{K[x]}$. We fix an extension of $\psi$ to $C[\Delta^{-1}]$, which means that we label the roots of $f$. The construction of the invariants will be independent of this choice.  

Consider the marked tree filtration $\phi$ attached to $\psi(f)$ as in Remark  
\ref{TreeBerkovich}. We let $c$ be a height for which $\phi$ has no branching and we take $L_{c}\subset{\{(\alpha_{i})\}}=:A$ to be a maximal $c$-trivial subtree with respect to $\psi$. We denote the branch heights of $L_{c}$ by 
$a_{0}<a_{1}<...<a_{m}$. Note that $a_{m}<c$ by assumption. We will also make use of the sequence $(b_{i})$ of branch differences. It is defined by $b_{i}:=a_{i}-a_{i-1}$ and $b_{0}=a_{0}$. 

Consider the complete graph $G$ on $L_{c}\subset{A}$ with weights $k_{0}(e)=1$ for every $e\in{E(G)}$. We will recursively assign new weights to the edges using the tree structure of $L_{c}$, the graph itself will not change. Let $e=\{\alpha_{i},\alpha_{j}\}\in{E(G)}$. We set $d_{e,\psi}:=\mathrm{val}(\psi(\alpha_{i}-\alpha_{j}))$. We will also denote this by $d_{e}$ if $\psi$ is clear from context. Let 
\begin{eqnarray}
R_{a_{i}}&=&\{e\in{E(G)}:d_{e}=a_{i}\},\\
S_{a_{i}}&=&\{e\in{E(G)}:d_{e}\geq{a_{i}}\},\\
\end{eqnarray}
If we set $c_{i,m}=\sum_{e\in{S_{a_{i}}}}{k(e)}=|S_{a_{i}}|$, then we can write  
\begin{equation}
\mathrm{val}(\psi(I_{G,1}))=\sum{2c_{i,m}b_{i}}.
\end{equation}
The value $c_{m,m}$ will be of special importance to us, so we assign a separate variable 
\begin{equation}
C_{m}:=c_{m,m}.
\end{equation}
We now consider the edges between leaves $\alpha_{i}$ and $\alpha_{j}$ of $L_{c}$ such that $\mathrm{val}(\alpha_{i}-\alpha_{j})=a_{m-1}$. In other words, $e=\{i,j\}\in{R_{a_{m-1}}}$. 
We set a new weight function
\begin{equation}
k_{1}(e)=\begin{cases}
C_{m}+1 & \text{ if }e\in{R_{a_{m-1}}}\\
k_{0}(e) &  \text{ if }e\notin{R_{a_{m-1}}}
\end{cases}
\end{equation}
By definition, we have that $k_{1}(e)\geq{k_{0}(e)}$ for every $e$, with a strict inequality for $e\in{R_{a_{m-1}}}$. 
This weight function defines a new pre-invariant $I_{G,k_{1}}$ by the construction in Section \ref{SectionQuasiInvariants}. If we write 
\begin{equation}
c_{i,m-1}:=\sum_{e\in{S_{a_{i}}}}k_{1}(e),
\end{equation} 
then we can express the valuation of the corresponding pre-invariant as 
\begin{equation}
\mathrm{val}(\psi(I_{G,k_{1}}))=\sum{2c_{i,m-1}b_{i}}.
\end{equation}
Note that $c_{m,m-1}=c_{m,m}=C_{m}$. As before, we define a separate variable for $c_{m-1,m-1}$:
\begin{equation}
C_{m-1}:=c_{m-1,m-1}.
\end{equation}

We now continue to the next branch layer: consider the edges defined by leaves $\alpha_{i}$ and $\alpha_{j}$ with $\mathrm{val}(\alpha_{i}-\alpha_{j})=a_{m-2}$. We set
\begin{equation}
k_{2}(e)=\begin{cases}
C_{m-1}+1 & \text{ if }e\in{R_{a_{m-2}}}\\
k_{1}(e) & \text{ if }e\notin{R_{a_{m-2}}}
\end{cases}
\end{equation}
Note that $k_{2}(e)>{k_{1}(e)}=1$ for $e\in{R_{a_{m-2}}}\subset{S_{a_{m-2}}}$ by definition of $C_{m-1}$.
As before, we obtain a new pre-invariant $I_{G,k_{2}}$ and using
\begin{equation}
c_{i,m-2}=\sum_{e\in{S_{a_{i}}}}k_{2}(e)
\end{equation}
we can write
\begin{equation}
\mathrm{val}(\psi(I_{G,k_{2}}))=\sum{2c_{i,m-2}b_{i}}.
\end{equation}
We define 
\begin{equation}
C_{m-2}:=c_{m-2,m-2} 
\end{equation}
and we have $c_{i,m-2}=c_{i,m-1}$ for $i\geq{m-1}$, and $c_{i,m-2}>c_{i,m-1}$ for $i<m-1$. Continuing in this way to the last layer, we obtain a weight function $k_{m}$ on $G$ and coefficients $C_{i}=c_{i,0}$ such that 
\begin{equation}
\mathrm{val}(\psi(I_{G,k_{m}}))=\sum{2C_{i}b_{i}}.
\end{equation}

\begin{rem}
The weights are related to the $C_{i}$ as follows. Consider two leaves $\alpha_{i}$ and $\alpha_{j}$ such that $\mathrm{val}(\alpha_{i}-\alpha_{j})=a_{m-i}$. Then 
\begin{equation}
k_{m}(e)=C_{m-(i-1)}+1.
\end{equation}
Similarly, if $\mathrm{val}(\alpha_{i}-\alpha_{j})=a_{i}$, then 
\begin{equation}
k_{m}(e)=C_{i+1}+1.
\end{equation}
\end{rem}

\begin{mydef}{\bf{[Invariants for marked tree filtrations]}} \label{InvariantsCTrivial}
Let $\psi:C[\Delta^{-1}]\to{\overline{K}}$ be a homomorphism with marked tree filtration $\phi$ and let $c$ be a height at which $\phi$ has no branching. The generating polynomials $F_{G,k_{m}}$ associated to the edge-weighted graph $(G,k_{m})$ constructed above are the invariants associated to $\psi$ and $c$. The value $M:=\mathrm{val}(\psi(I_{G,k_{m}}))$ is called the minimizing value. We will also use the notation $(G_{c},k_{c})$ for the edge-weighted graph $(G,k_{m})$. 
\end{mydef}

We will see in Proposition \ref{MainProposition4} that this value $M$ is indeed minimal among all the $\mathrm{val}(\psi(\sigma(I_{G,k_{m}})))$ for $\sigma\in{S_{d}}$. 

\begin{rem}\label{CombEquivInvariants}
The construction of the weight function is independent of the specific lengths $a_{0}<a_{1}<...<a_{m}$. That is, if we start with a different set of lengths $a'_{0}<a'_{1}<...<a'_{m}$ but the same filtration structure as in Definition \ref{CombinatorialStructure}, then the weights are the same.  The minimizing value however is dependent on the set of lengths.  
\end{rem}

\begin{exa}\label{PreInvariantsExample}
We calculate the invariants in Definition \ref{InvariantsCTrivial} for the trees in Example \ref{ExampleTreesFiveLeaves}. Here we now allow the lowest branch height $a_{0}$ to be nonzero. Let $a_{0}<c<a_{1}$. Then a maximal $c$-trivial subtree for any of the three trees consists of two leaves from two different branches. We thus consider the complete graph on two vertices with trivial weight function. In general, if a marked tree filtration has $r$ different branches at the root $v_{\infty}$, then the corresponding invariants are obtained by taking the complete graph on $r$ vertices with trivial weight function. The corresponding invariant in this case are obtained by symmetrizing $x-(\alpha_{1}-\alpha_{2})^2$. The minimizing value is $M=12a_{0}$ for all three filtration types. 


We now consider a height $a_{1}<c<a_{2}$. Consider a maximal $c$-trivial subtree for the marked tree filtration I. If we label the leaves from left to right, then we can take $\{1,3,4,5\}$. We start with the complete graph $G$ on these four vertices and with trivial weight function $k_{0}$. There are then two edges with $d_{e}=a_{1}$, namely $e_{1}=\{1,3\}$ and $e_{2}=\{4,5\}$. We thus have $C_{1}=2$. The next weight function $k_{1}$ is $1$ on $e_{1}$ and $e_{2}$, and it is $3$ on the other edges. The corresponding pre-invariant is given by
\begin{equation}
I_{G,k_{1}}=[13][45]([14][15][34][35])^3.
\end{equation}
Here $[ij]=(\alpha_{i}-\alpha_{j})^2$ as in Section \ref{SectionQuasiInvariants}. The minimizing value is $M=12a_{0}+2a_{1}$ and the stabilizer $H_{G,k_{1}}$ is the group generated by the permutation $(14)(35)$. The generating polynomial $F_{G,k_{1}}$ thus has degree $5!/2=60$. For the second marked tree filtration, we also take the $c$-trivial subtree $\{1,3,4,5\}$. We then have $C_{1}=3$, $k_{1}(\{1,3\})=k_{1}(\{1,4\})=k_{1}(\{1,5\})=4$ and $k_{1}(e)=1$ for the other edges. The corresponding pre-invariant is
\begin{equation}
I_{G,k_{1}}=[34][35][45]([13][14][15])^{4}
\end{equation}
and the minimizing value is $M=24a_{0}+6a_{1}$. The stabilizer $H_{G,k_{1}}$ is the group of permutations on $\{3,4,5\}$ and the degree of the generating polynomial is thus $5!/6=20$. We leave the calculation of the pre-invariant corresponding to the third marked tree filtration with $a_{1}<c<a_{2}$ to the reader. We finish this example by calculating the pre-invariant corresponding the first marked tree filtration with $c>a_{2}$. 
We start with the complete graph on $\{1,2,3,4,5\}$ and the trivial weight function $k_{0}$. We have $C_{2}=1$, $k_{1}(\{1,3\})=k_{1}(\{2,3\})=k_{1}(\{4,5\})=2$ and $k_{1}(e)=1$ for the other edges. To calculate $C_{1}=c_{m-1,m-1}=\sum_{e\in{S_{a_{1}}}}k_{1}(e)$, note that the edges in $S_{a_{1}}$ are $\{1,2\}$, $\{1,3\}$, $\{2,3\}$ and $\{4,5\}$, so $C_{1}=1+2+2+2=7$. We then have $k_{2}(e)=C_{1}+1=8$ for $e\in{R_{a_{0}}}$. The corresponding pre-invariant is thus given by
\begin{equation}
I_{G,k_{2}}=[12]([13][23][45])^{2}([14][15][24][25][34][35])^{8}.
\end{equation}
The minimizing value is given by $M=96a_{0}+6a_{1}+2a_{2}$. %
The stabilizer $H_{G,k_{2}}$ is the permutation group on $\{4,5\}$, so the generating polynomial $F_{G,k_{2}}$ has degree $5!/2=60$.


\end{exa}

We now prove some simple properties of the $C_{i}$. 

\begin{lemma}
$C_{m-(i-1)}<C_{m-i}$.
\end{lemma}
\begin{proof}
We have $C_{m-(i-1)}=\sum_{e\in{S_{a_{m-(i-1)}}}}k_{i-1}(e)$ and $C_{m-i}=\sum_{e\in{S_{a_{m-i}}}}k_{i}(e)$ by definition. Furthermore, $k_{i}(e)$ is defined by 
\begin{equation}
k_{i}(e)=\begin{cases}
C_{m-(i-1)}+1 & \text{ for } e\in{R_{a_{m-i}}}\\
k_{i-1}(e) & \text{ for }  e\notin{R_{a_{m-i}}}
\end{cases}.
\end{equation}
Since $k_{i-1}(e)=1$ for $e\in{R_{a_{m-i}}}$, we find that 
 $k_{i}(e)>k_{i-1}(e)$ for these edges and $k_{i}(e)=k_{i-1}(e)$ for $e\notin{R_{a_{m-i}}}$. Using $S_{a_{m-(i-1)}}\subset{S_{a_{m-i}}}$, we then directly see that $C_{m-i}>C_{m-(i-1)}$, as desired.   
 
\end{proof}

\begin{cor}\label{Decreasing}
The sequence $C_{i}$ is a strictly decreasing sequence. That is, if $j>i$, then $C_{j}<C_{i}$. 
\end{cor}


\begin{lemma}\label{WeightInequality}
If $e\notin{S_{a_{i}}}$, then $k_{m}(e)>C_{i}$. 
\end{lemma}
  \begin{proof}
  Suppose that $e\in{R_{a_{j}}}$ for $j<i$. If $i=j+1$, then $k_{m}(e)=C_{j+1}+1=C_{i}+1>C_{i}$. Suppose now that $i>j+1$. Then $k_{m}(e)=C_{j+1}+1>C_{i}+1>C_{i}$ by Corollary \ref{Decreasing}. This concludes the proof. 
  \end{proof}

\subsubsection{Comparing invariants}

We now consider two homomorphisms $\psi$ and $\psi'$ with marked tree filtrations $\phi$ and $\phi'$. We assume that $\phi$ and $\phi'$ are isomorphic up to height $a_{m}$. This in particular implies that they have the same number of branches at every $a_{i}\leq{a_{m}}$. Let $c$ be slightly larger than $a_{m}$, so that neither $\psi$ nor $\psi$ has branching between $a_{m}$ and $c$. For this $c$, we consider the polynomial $F_{G,k_{m}}$ associated to $\phi$ in Definition \ref{InvariantsCTrivial}. By Remark \ref{RemarkStabilizer}, a root of $\psi'(F_{G,k_{m}})$ corresponds to an isomorphic edge-weighted graph $(G',k')$. We write 
\begin{equation}
\sigma:G\to{G'}
\end{equation} 
for the isomorphism. Note that every edge $e$ in $K_{n}$ is a pair $\{\alpha_{i},\alpha_{j}\}$. We then set 
\begin{eqnarray*}
d_{e,\psi}&:=&\mathrm{val}(\psi(\alpha_{i}-\alpha_{j})),\\
d_{e,\psi'}&:=&\mathrm{val}(\psi'(\alpha_{i}-\alpha_{j})).
\end{eqnarray*}
Furthermore, we define 
the analogue of $S_{a_{i}}$ for $G'$ and $\psi'$ as follows:
\begin{equation}
T_{s}=\{e\in{E(G')}:d_{e,\psi'}\geq{s}\}.
\end{equation}
Here we allow $s$ to be greater than $a_{m}$. If we adopt the same notation for $G$ and $\psi$, then $S_{c}=\emptyset$ 
since $L_{c}$ is a maximal $c$-trivial tree. Setting
\begin{equation}
r_{s}=\sum_{e'\in{T_{s}}}k'(e'),
\end{equation}
we then have the formula
\begin{equation}
\mathrm{val}(\psi'(I_{G',k'}))=\sum{2r_{s}b_{s}},
\end{equation}
where the sum is over all branch heights $s$ of $\phi'$.

  \begin{lemma}\label{LevelEquality0b}
Suppose that there exists an $e'\in{T_{a_{i}}}$ such that $\sigma^{-1}(e')\notin{S_{a_{i}}}$. Then $r_{a_{i}}>C_{i}$. Suppose that for some $s>a_{m}$, there exists an $e'\in{T_{s}}$ such that $\sigma^{-1}(e')\notin{S_{s}}$. Then $C_{s}=0<r_{s}$.   
\end{lemma}
\begin{proof}
It suffices to show that $k'(e')>C_{i}$ for this particular $e'$. Let $a_{j}=d_{e,\psi}$ for $e=\sigma^{-1}(e')$, so that $a_{j}<a_{i}$ by assumption. 
By Lemma \ref{WeightInequality}, we have $k_{m}(e)>C_{i}$. Since $\sigma$ is weight-preserving, we have $k_{m}(e)=k'(e')>C_{i}$, as desired. Note that the statement in the second part is equivalent to the existence of an $e'\in{T_{s}}$, since we automatically have $\sigma^{-1}(e')\notin{S_{s}}=\emptyset$. If this happens for some $s>a_{m}$, then $r_{s}$ is nonzero by definition and $C_{s}=0<r_{s}$, as desired. 

\end{proof}

\begin{lemma} \label{LevelEqualityb}
Let $i\leq{m}$. Suppose that for every $e'\in{T_{a_{i}}}$, we have $\sigma^{-1}(e')\in{S_{a_{i}}}$. Then $\phi'_{a_{i}}(v')=\phi'_{a_{i}}(w')$ if and only if $\phi_{a_{i}}(\sigma^{-1}(v'))=\phi_{a_{i}}(\sigma^{-1}(w'))$. Suppose that for every $e'\in{T_{c}}$, we have $e=\sigma^{-1}(e')\in{S_{c}}$. Then $T_{c}=\emptyset$.

\end{lemma}

\begin{proof}
Recall that the number of branches of a marked tree filtration $\phi$ at a certain height $s$ is the number of values attained by the restricted function $\phi_{s}$. Since $V(G)$ is a maximal $c$-trivial tree for $c>a_{i}$, we find that the restriction of $\phi:A\times{\mathbf{R}}\to\mathbf{R}$ to $V(G)$ also attains the same maximum number of values for $a_{i}$. Since $\phi$ and $\phi'$ are isomorphic up to height $a_{m}$, we find that they have the same number of branches at every height $s\leq{}a_{m}$. It follows that the number of branches at $a_{i}$ of the restriction of $\phi'$ to $V(G')$ is less than or equal to the number of branches of $\phi$ restricted to $V(G)$. Note now that the assumptions in the Lemma give a well-defined map from the set of $a_{i}$-branches of $\psi'$ to the set of $a_{i}$-branches of $\psi$. We claim that this map is a bijection.   
Indeed, let $i_{1},...,i_{\ell}\in{V(G)}$ be the representatives of the $a_{i}$-branches with respect to $\psi$ so that no pair of these is in $S_{a_{i}}$. Using the assumption in the lemma, we then see that $\sigma(i_{1}),...,\sigma(i_{\ell})$ are in different $a_{i}$-branches with respect to $\psi'$. 
 The number of $a_{i}$-branches in $V(G')$ with respect to $\psi'$ is thus greater than or equal to $\ell$. Since we already had the other inequality, it follows that they are equal and that the induced map of $a_{i}$-branches is a bijection. Translating this statement then directly gives the first part of the lemma. The remaining part about $T_{c}$ and $S_{c}$ also follows, since $S_{c}=\emptyset$.




\end{proof}




\begin{cor}\label{LevelEquality1b}
Suppose that we are in the situation described in Lemma \ref{LevelEqualityb} with height $s=a_{i}$. Then $r_{s}=C_{i}$. 
\end{cor}
\begin{proof}
Indeed, by definition we have $C_{i}=\sum_{e\in{S_{a_{i}}}}k_{m}(e)$. 
Since the edges in $S_{a_{i}}$ are exactly the $vw$ with $\phi_{a_{i}}(v)=\phi_{a_{i}}(w)$, we obtain $C_{i}=r_{s}$ from Lemma \ref{LevelEqualityb} and the fact that $\sigma$ preserves weights. 
\end{proof}

\begin{lemma}\label{LevelEquality2a}
Suppose that $M'=M$. Then $\sigma$ induces an isomorphism of marked tree filtrations.
\end{lemma}
\begin{proof}
By Lemmas \ref{LevelEquality0b} and \ref{LevelEquality1b}, we see that $M$ is the minimal value that can be attained. It occurs exactly when the conditions in Lemma \ref{LevelEqualityb} are satisfied for every height $s\in\{a_{i}\}$ and for $s=c$. 
We claim that this gives an isomorphism of marked tree filtrations. We have to check that $\phi_{s}(v)=\phi_{s}(w)$ if and only if $\phi'_{s}(\sigma(v))=\phi_{s}(\sigma(w))$ for every pair of vertices $v,w\in{V(G)}$ and every height $s$. These equalities do not occur for heights $s$ strictly greater than $a_{m}$ (since otherwise we would overshoot the minimum by Lemma \ref{LevelEquality0b}), so we only have to check them for heights $s\leq{a_{m}}$. For these the desired equalities follow from Lemma \ref{LevelEqualityb}, as desired. 



\end{proof}

\begin{prop}\label{MainProposition4}
Let $\psi,\psi':D[\Delta^{-1}]\to{K}$ be two homomorphisms with marked tree filtrations $\phi$ and $\phi'$ and branching heights $(a_{i})$ and $(a'_{i})$. Suppose that $\phi$ and $\phi'$
 are isomorphic up to height $a_{m}$ and let $c=a_{m}+\epsilon$ for $\epsilon>0$ small enough so that there is no branching between $a_{m}$ and $c$ for $\phi$ and $\phi'$. 
Let $L_{c}$ be a maximal $c$-trivial subtree of $\phi$ with edge-weighted graph $(G_{c},k_{c})$ and minimal value $M$. 
Then the slopes of the Newton polygon of $\psi'(F_{G_{c},k_{c}})$ are $\geq{M}$. The marked tree filtration $\phi'$ contains a $c$-trivial subtree isomorphic to the maximal $c$-trivial subtree of $\phi$ if and only if $M$ is attained.  

\end{prop}

\begin{proof}
By Lemmas \ref{LevelEquality0b} and \ref{LevelEquality1b} we obtain that $M$ is a lower bound for the slopes. If $M$ is attained, then we obtain the desired isomorphic subtree from Lemma \ref{LevelEquality2a}. If $\phi'$ contains an isomorphic copy of the maximal $c$-trivial subtree of $\phi$, then the calculations in the lemmas above similarly show that $M$ is attained. 
\end{proof}

\begin{rem}
We can in particular apply Proposition \ref{MainProposition4} to case where $\phi=\phi'$. It then says that $M$ is the minimal slope of the Newton polygon of $\psi(F_{G_{c},k_{c}})$, thus justifying the nomenclature in Definition \ref{InvariantsCTrivial}.   
\end{rem}

We can now prove our main theorem: 

\vspace{0.1cm}

\begin{proofMainThm1}
Let $a_{i}$ and $a'_{i}$ be the branching heights of $\phi$ and $\phi'$. We write $\psi$ and $\psi'$ for the corresponding homomorphisms from $D[\Delta^{-1}]$ to $\overline{K}$. Suppose without loss of generality that $a_{0}<a'_{0}$. Then the smallest slope in the Newton polygon of $\psi(F_{G_{2}})$ is $a_{0}$, which is different from $a'_{0}$, a contradiction. We conclude that $a_{0}=a'_{0}$. We now let $a_{m}$ be the largest height such that $a_{m}$-structures of $\phi$ and $\phi'$ are isomorphic. This exists since the $a_{0}$-structures are isomorphic.   

Let $n_{1}$ be the number of branches of $\phi$ at a height $c$ slightly larger than $a_{m}$ (so that there is no branching between $c$ and $a_{m}$). We   
 similarly define $n_{2}$ for $\phi'$. We suppose without loss of generality that $n_{1}\geq{n_{2}}$. If this inequality is strict, then we consider a maximal $c$-trivial tree of $\phi$ for $c>a_{m}$ with edge-weighted graph $(G_{c},k_{c})$. The slopes of $\psi'(F_{G_{c},k_{c}})$ are then larger than $M$, because otherwise $\phi'$ would contain a $c$-trivial subtree of order $n_{1}$ by Proposition \ref{MainProposition4}. Suppose now that $n_{1}=n_{2}$ and consider the same $(G_{c},k_{c})$. If $\psi'(F_{G_{c},k_{c}})$ would contain $M$ as a minimal slope, then $\phi'$ would be isomorphic to $\phi$ up to a larger height by Proposition \ref{MainProposition4} and Remark \ref{ExtendingIsomorphisms}, a contradiction. We conclude that $\mathrm{trop}(f)$ is not equal to $\mathrm{trop}(g)$. By construction, only finitely many invariants are needed here since the invariants are independent of the branching heights $a_{i}$. This finishes the proof.


\end{proofMainThm1}

\subsection{Newton half-spaces}
\label{SectionExamplesTrees}

In this section we study moduli of Newton polygons to write down explicit equations for the filtration types of trees with a fixed number of leaves. These moduli are described by intersections of polyhedral half-spaces. 
We explicitly work these half-spaces out for all polynomials of degree $d\leq{5}$. 

\begin{rem}
The variables $b_{i}$ were used in Section \ref{SectionTruncatedStructures} to denote the branch differences. In this section, they will be used to denote the tropical invariants introduced in Section \ref{SectionQuasiInvariants}. 
\end{rem}


We start by introducing notation for lines between two points.
Let $P=(z_{0},w_{0})$ and $Q=(z_{1},w_{1})$ be two points with coefficients in a field 
and assume that $z_{0}\neq{z_{1}}$. Define 
\begin{eqnarray*}
h_{0}(x)=\dfrac{x-z_{0}}{z_{1}-z_{0}},\\
h_{1}(x)=\dfrac{x-z_{1}}{z_{0}-z_{1}}.
\end{eqnarray*} 
Then the polynomial describing the line between $P$ and $Q$ is given by
\begin{equation}
h_{P,Q}(x):=w_{1}h_{0}(x)+w_{0}h_{1}(x).
\end{equation}
That is, the graph of $h_{P,Q}$ passes through $P$ and $Q$. We will use these polynomials to define half-space conditions for families of Newton polygons. 

For a fixed $r\in\mathbf{N}$, let $B=\mathbf{R}[b_{i}]$ be the polynomial ring in $r+1$ variables and let $S$ be the set of pairs $\{P_{i}\}=\{(i,b_{r-i})\}$. The reverse labeling will be explained in Example \ref{NewtonPolygonExample}. We view the $b_{r-i}$ as the valuations of the coefficients of some polynomial over a valued field. We now fix a pair ($P_{i},P_{j}$) in $S$ with $i\neq{j}$ and consider the polynomial $h_{P_{i},P_{j}}(x)\in{B[x]}$. Given any $k\in\{0,...,r\}$, we then set
\begin{equation}
\Delta_{k,P_{i},P_{j}}=b_{r-k}-h_{P_{i},P_{j}}(k).
\end{equation} 
We will write $\Delta_{k}$ for this linear polynomial if $P_{i}$ and $P_{j}$ are clear from context. It consists of at most $3$ monomials: $b_{r-i}$, $b_{r-j}$ and $b_{r-k}$. Suppose now that we are given $\psi({b})=(\psi(b_{i}))\in\mathbf{T}^{r+1}$, where $\mathbf{T}=\mathbf{R}\cup\{\infty\}$ is the tropical affine line. If $\psi(b_{r-i})$, $\psi(b_{r-j})$ and $\psi(b_{r-k})$ are in $\mathbf{R}$, then we can safely evaluate $\Delta_{k}$ at these values and obtain a value $\psi(\Delta_{k})\in\mathbf{R}$.  If one of these values is infinite, then we define $\psi(\Delta_{k})$ as follows: 
\begin{enumerate}
\item If $\psi(b_{r-i}),\psi(b_{r-j})\in\mathbf{R}$ and $\psi(b_{r-k})=\infty$, then we set $\psi(\Delta_{k})=\infty$.
\item If either $\psi(b_{r-i})$ or $\psi(b_{r-j})$ is $\infty$, then we set $\psi(\Delta_{k})=-1$.
\end{enumerate}
Our motivation for these will be given after the following definition.

\newpage
\begin{mydef}{\bf{[Newton half-spaces]}} \label{NPHalfSpace}
 Let $P_{i}$, $P_{j}$ and $\Delta_{k}$ be as above. 
Let $\mathbf{T}^{r+1}$ be the $r+1$-dimensional tropical affine line. 
We define $I(P_{i},P_{j})\subset\mathbf{T}^{r+1}$ to be the set of all $\psi({b})=(\psi(b_{i}))\in\mathbf{T}^{r+1}$ such that the following hold: 
\begin{enumerate}
\item For every integer $k$ with $i\leq{k}\leq{j}$, we have $\psi(\Delta_{k})\geq{0}$.
\item For every integer $k$ with $k<i$, we have $\psi(\Delta_{k})>0$. 
\end{enumerate}  
We call $I(P_{i},P_{j})$ the Newton half-space associated to $P_{i}$ and $P_{j}$. 

\end{mydef}


\begin{rem}\label{NotationRemark}
We will often write $b_{i}$ from now on instead of $\psi(b_{i})$ to ease notation. If a certain specific choice of $\psi(b_{i})$ is important, then we will use $\psi(b_{i})$. 

\end{rem}



\begin{rem}
The half-spaces conditions in Definition \ref{NPHalfSpace} are the local convexity conditions for the Newton polygons (lower convex hulls) associated to the points $P_{i}$. 
That is, the first condition says that  
if the line between $P_{i}$ and $P_{j}$ is a line segment of the Newton polygon, then every value in between lies above or on this line. The second condition says that the line segment stops at $P_{i}$. Note that these conditions by themselves are not enough to deduce the existence of line segments in Newton polygons. If we have a disjoint sequence of these conditions that completely cover (in the obvious sense) the options in $S$, then this does determine the Newton polygon. See Example \ref{NewtonPolygonExample} for instance. 
\end{rem}

\begin{rem}
Our definition of $\psi(\Delta_{k})$ in the infinite cases comes from the following observation.  
Suppose that we are given a polynomial $F=\sum_{i=0}^{r}c_{i}x^{i}\in{K[x]}$ over a non-archimedean field $K$ with Newton polygon $\mathcal{N}(F)$. 
We use the notation $\psi(b_{i})=\mathrm{val}(c_{i})$, $\psi({b})=(\psi(b_{i}))$, $\psi(P_{i})=(i,\psi(b_{r-i}))$. If $\mathcal{N}(F)$ has a segment $\ell=\psi(P_{i})\psi(P_{j})$ of finite slope, then we have $\psi({b})\in{I(P_{i},P_{j})}$.   
Let $F'=\sum_{i=0}^{r}c'_{i}x^{i}$ be another polynomial with $\psi({b}')=(\psi(b'_{i}))=(\mathrm{val}(c'_{i}))$. We now want to check using our definition of $I(P_{i},P_{j})$ if the line segment $\psi(P'_{i})\psi(P'_{j})$ is also present in the Newton polygon of $F'$. 
If one of the coefficients $\psi(b'_{r-i})$ or $\psi(b'_{r-j})$ is infinite, then these do not give rise to a line segment, so we should impose the condition $\psi({b}')\notin{I(P_{i},P_{j})}$. If $\psi(b'_{r-i})$ and $\psi(b'_{r-j})$ are finite and $\psi(b'_{r-k})$ is infinite, then $b'_{r-k}$ does not give an obstruction to $\psi(P'_{i})\psi(P'_{j})$ giving a line segment, so we should impose $\psi({b}')\in{I(P_{i},P_{j})}$. This is our motivation for defining $\psi(\Delta_{k})$ as above. 

\end{rem}

\begin{exa}\label{NewtonPolygonExample}
Suppose that we want to describe tuples 
$${b}=(b_{0},b_{1},b_{2},b_{3},b_{4})\in\mathbf{T}^{5}$$ that give a Newton polygon as in Figure \ref{NewtonPolygonExample2} (see Remark \ref{NotationRemark} for the notation).
\begin{figure}
\centering
\includegraphics[height=8cm]{{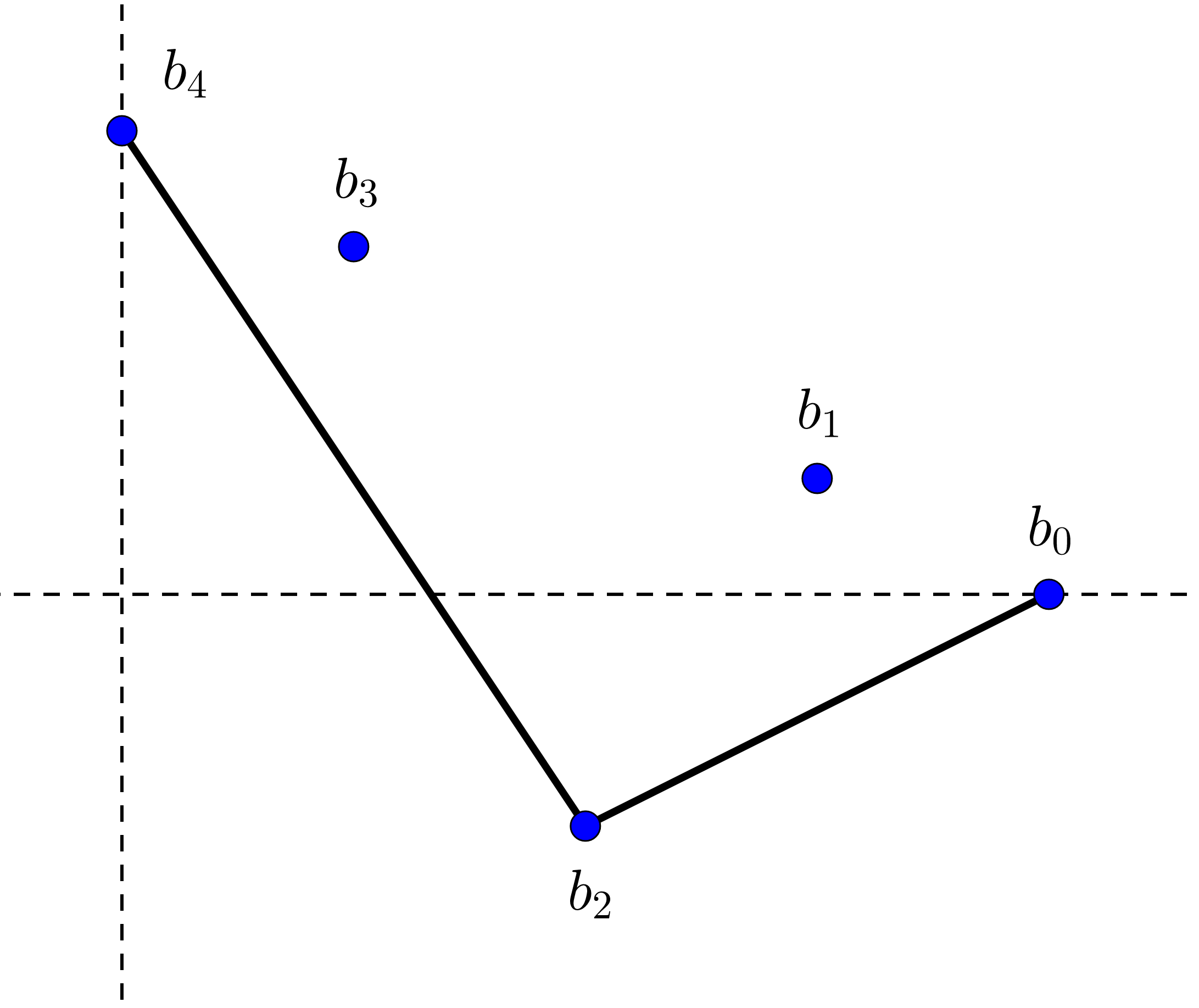}}
\caption{\label{NewtonPolygonExample2}
One of the Newton polygons in Example \ref{NewtonPolygonExample}. The dotted lines are the $x$ and $y$-axes. 
} 
\end{figure}
Then $b_{0},b_{2}$ and $b_{4}$ are finite, so 
$b_{0},b_{2},b_{4}\in\mathbf{R}$ (this would also come out of the equations we defined in Definition \ref{NPHalfSpace}). 
We first write down the polynomials $\Delta_{k,P_{2},P_{4}}$ for the pair $(P_{2},P_{4})$. The non-trivial ones are given by 
\begin{align*}
\Delta_{3,P_{2},P_{4}}=&b_{1}-1/2(b_{0}+b_{2}),\\
\Delta_{1,P_{2},P_{4}}=&b_{3}+1/2b_{0}-3/2b_{2},\\
\Delta_{0,P_{2},P_{4}}=&b_{4}-2b_{2}+b_{0}.
\end{align*} 
The nontrivial polynomial $\Delta_{1,P_{0},P_{2}}$ for the pair $(P_{0},P_{2})$ is then 
\begin{align*}
\Delta_{1,P_{0},P_{2}}=&b_{3}-1/2(b_{2}+b_{4}).
\end{align*}
We directly find that the set of tuples that describe a Newton polygon as in Figure \ref{NewtonPolygonExample2} is given by $I(P_{2},P_{4})\cap{I(P_{0},P_{2})}$. Explicitly, it is given by the equations 
\begin{align*}
b_{1}&\geq{1/2(b_{0}+b_{2})},\\
 b_{3}&>3/2b_{2}-1/2b_{0},\\
 b_{4}&>2b_{2}-b_{0},\\
 b_{3}&\geq{1/2(b_{2}+b_{4})}.
\end{align*} 
Here $(b_{0},b_{1},b_{2},b_{3},b_{4})\in\mathbf{R}\times\mathbf{T}\times\mathbf{R}\times\mathbf{T}\times\mathbf{R}$. 
Notice that if we assign $\mathrm{wt}(b_{i})=i$ and $\mathrm{wt}(m/n)=m/n$ for a fraction $m/n$, then the above hypersurfaces are homogeneous. This is the reason we defined $S$ as the set of $(i,b_{r-i})$ instead of the set of $(i,b_{i})$.
\end{exa}

\begin{mydef}\label{NewtonPolygonModuli}{\bf{[Moduli of Newton polygons]}}
Consider a vector $(\psi({b}_{i}))\in\mathbf{T}^{r+1}$, where at least $\psi(b_{0})$ and $\psi(b_{r})$ are finite. Write $\mathcal{N}(\psi(P_{i}))$ for the Newton polygon of the points $\psi(P_{i})$ and consider its essential vertices $\psi(P_{i_{1}})$, $\psi(P_{i_{2}})$,..., $\psi(P_{i_{t}})$. 
The moduli space of Newton polygons of type $\mathcal{N}(\psi(P_{i}))$ is then the intersection of the half-spaces 
$I(P_{i_{j}},P_{i_{j+1}})$. 
\end{mydef}

\begin{rem}
If we take another point $Q\in{I(P_{i_{j}},P_{i_{j+1}})}$ and repeat the above procedure, then one easily sees that we obtain the same intersection of half-spaces. The definition thus does not depend on the particular vector $(\psi(b_{i}))$ used in the beginning. 
\end{rem}

\begin{rem}
In this paper, the polygons we are interested are the Newton polygons of the generating polynomials $F_{G,k}$ defined in Section \ref{SectionQuasiInvariants}. These satisfy two additional conditions:
\begin{enumerate}
\item $b_{r}=0$,
\item $b_{0}\in\mathbf{R}$.
\end{enumerate} 
The first holds because the leading coefficient of $F_{G,k}$ is $1$. The second holds since the roots of $F_{G,k}$ are all nonzero. Indeed, the roots $\alpha_{i}$ of $f(x)$ are distinct, so $(\alpha_{i}-\alpha_{j})^{2}=[ij]\neq{0}$. This means that any pre-invariant, being a product of these $[ij]$, is also nonzero and thus $\psi(b_{0})\neq\infty$. 
We thus see that our tropicalization map from $D[\Delta](K)$ lands in the subspace $\mathbf{R}\times\mathbf{T}^{r-1}\times\{0\}\subset{}\mathbf{T}^{r+1}$.
\end{rem}

\subsubsection{Equations for filtration types}

In this section we give a set of equations that completely determine the filtration type of a separable polynomial.  To do this, we iteratively apply 
Proposition \ref{MainProposition4}, which tells us that we can determine tree types using line segments of minimal slope in the generating polynomials of the tropical invariants. 
We use a modification of the Newton half-spaces from the previous section to determine whether these minimal slopes are attained. 

We start with a preliminary observation on the number of filtration types (see Definition \ref{CombinatorialStructure}). 
\begin{lemma}\label{FiniteFiltration}
For any fixed number of leaves, the number of filtration types is finite.
\end{lemma}  
\begin{proof}
Indeed, at every branching height there are only finitely many options for every branch to split. Since there are only finitely many branching heights, we see that the number of filtration types is finite.
\end{proof}

Suppose now that we have a filtration type, represented by a marked tree filtration $\phi$. If its branch heights are $\Lambda$-rational (where $\Lambda$ is the value group of $\overline{K}$), then we can find a separable polynomial $f(x)\in{\overline{K}[x]}$ with tree equal to $\phi$. Since we can freely change the branch heights in a filtration type, we see that we can represent any filtration type by a $\overline{K}$-rational polynomial. By Lemma \ref{FiniteFiltration}, there are only finitely many of these, so we can represent the filtration types by a finite set of polynomials $\mathcal{G}=\{f_{i}\}$. We write $\psi_{i}$ for the homomorphisms $D[\Delta^{-1}]\to\overline{K}$ corresponding to these. 

\begin{notation}{\bf{[Invariant generators of filtration types]}}
Let $f_{i}\in\mathcal{G}$ be a filtration type. For every height $c\in\mathbf{R}$, we obtain 
an edge-weighted graph and thus a generating polynomial  
using the construction in Section \ref{SectionTruncatedStructures}. 
By varying $c$, this gives finitely many edge-weighted graphs (since they only change at branch heights) and thus finitely many generating polynomials. 
We write $c_{i,j}$ for the $j$-th height corresponding to the polynomial $f_{i}$. 
We denote the corresponding edge-weighted graphs by $G_{i,j}$ and the generating polynomials by $F_{i,j}$. The degree of $F_{i,j}$ is written as $d(i,j)$. We furthermore write $G_{2}$ for the trivially weighted graph of order two with generating polynomial $F_{G_{2}}$ of degree $d({2})$.
\end{notation}

By mapping a polynomial $f(x)$ with corresponding homomorphism $\psi_{f}:D[\Delta^{-1}]\to{\overline{K}}$ to the valuations of the generating polynomials $\psi_{f}(F_{G_{2}})$ and $\psi_{f}(F_{i,j})$ for every $i,j$, 
we obtain a map 
\begin{equation}
\mathrm{trop}:D[\Delta^{-1}](\overline{K})\to{\mathbf{T}^{d({2})}}\times\prod_{i,j}\mathbf{T}^{d(i,j)}. 
\end{equation}  
A quick inspection of the proof of Theorem \ref{MainThm1} shows that the invariants used there are a subset of the ones we use here, so we find that this map is injective if we consider polynomials up to isomorphisms of marked tree filtrations. We would now like to distinguish between the different filtration types of the polynomials on the left-hand side using half-spaces on the right-hand side. To do this, it will be convenient to have some notation for the various projection maps involved. We write  
\begin{equation}
\pi_{i,j}:{\mathbf{T}^{d({2})}}\times{}\prod_{i,j}\mathbf{T}^{d(i,j)} 
\to{\mathbf{T}^{d(i,j)}}
\end{equation}
and 
\begin{equation}
\pi_{2}:{\mathbf{T}^{d({2})}}\times{}\prod_{i,j}\mathbf{T}^{d(i,j)}
\to\mathbf{T}^{d({2})}.
\end{equation}

We first a set of preliminary Newton half-spaces using $F_{G_{2}}$.
\begin{mydef}\label{HalfSpaceOrderTwo}{\bf{[Half-spaces of order two]}}
Let $f_{i}\in\mathcal{G}$ be a filtration type and let $G_{2}$ be the trivially weighted graph of order two. The Newton half-space $NH(i,2)$ for $f_{i}$ is the set of $Q\in{\mathbf{T}^{d({2})}}\times{}\prod_{i,j}\mathbf{T}^{d(i,j)}$ such that $\pi_{2}(Q)$ is in the Newton polygon moduli space 
associated to the valuations of $\psi_{i}(F_{G_{2}})$ in Definition \ref{NewtonPolygonModuli}. We also refer to these as the half-spaces of order two.  
\end{mydef}


We now consider a fixed filtration type represented by a polynomial $f_{k}\in\mathcal{G}$ with homomorphism $\psi_{k}$ and half-space $NH(k,2)$.  
 Let $\mathcal{G}_{k,0}\subset{\mathcal{G}}$ be the set of all polynomials $f_{i}$ such that $NH(i,2)=NH(k,2)$. 
 For each of these polynomials $f_{i}\in\mathcal{G}_{k,0}$, we find that they share the same number of branch heights. 
 Indeed, the number of branch heights is equal to the number of distinct line segments in the Newton polygon of $\psi_{i}(F_{G_{2}})$ and by assumption these are the same for $f_{i}\in\mathcal{G}_{k,0}$. Since filtration types are defined up to a changes in lengths, we can and do assume that the branch heights for all the $f_{i}\in\mathcal{G}_{k,0}$ are equal. This means that the Newton polygons of the polynomials $\psi_{i}(F_{G_{2}})$ are in fact equal. 

We now consider the height $c_{k,0}$ satisfying $a_{0}<c_{k,0}<a_{1}$. Here the $a_{i}$ are the branch heights of $f_{k}$. For $f_{i}\in{\mathcal{G}_{k,0}}$, we then consider the Newton polygon of $\psi_{i}(F_{i,0})$ together with the line segment of smallest slope. 
We write $m(i,0)$ for the horizontal length (or multiplicity) of this segment. 
The slope of the minimal line segment is the minimizing value by Proposition \ref{MainProposition4}. 
Note that we can express this minimizing value as a linear polynomial in the valuations of the coefficients of $\psi_{i}(F_{2})$, see Example \ref{PreInvariantsExample} for instance. We write $M_{i,0}$ for this linear function. 
For every $f_{i}$, we now have a half-space $I(P_{d(i,0)-m(i,0)}P_{d(i,0)})$ 
and a linear function $M_{i,0}$. We want to compare the value of this function to the coefficient $b_{m(i,0)}$ at $P_{d(i,0)-m(i,0)}$. We thus consider the modified function
\begin{equation}
\Delta(M_{i,0})=b_{m(i,0)}-M_{i,0}.
\end{equation} 
This function will only need to be evaluated at elements in $\mathbf{R}$. 

\begin{mydef}\label{PolyhedraHeightZero}
{\bf{[Generalized polyhedra up to finite heights I]}}
Write $\mathcal{G}_{k,0,>}\subset{\mathcal{G}_{k,0}}$ for the filtration types in $\mathcal{G}_{k,0}$ 
that have strictly more branches than $f_{k}$ at $c_{k,0}$. 
The generalized polyhedron for $f_{k}$ up to height $c_{k,0}$ consists of all points $Q$ in ${\mathbf{T}^{d({2})}\times\prod_{i,j}\mathbf{T}^{d(i,j)}}$ 
such that:
\begin{enumerate}
\item $Q\in{NH(k,2)}$ (see Definition \ref{HalfSpaceOrderTwo}). 
\item $\pi_{k,0}(Q)\in
I(P_{d(k,0)-m(0,0)}P_{d(k,0)})$ 
and $\Delta(M_{k,0})(\pi_{k,0}(Q))=0$. 
\item For $f_{i}\in\mathcal{G}_{k,0,>}$, we have $\pi_{i,0}(Q)\notin{}I(P_{d(i,0)-m(i,0)}P_{d(i,0)})$ 
or $\Delta(M_{i,0})(\pi_{i,0}(Q))>0$. 
\end{enumerate}
We denote this space by $I(k,0)$. 
\end{mydef}

\begin{rem}
The first condition says that if $Q=\mathrm{trop}(f)$ for some polynomial $f$ with homomorphism $\psi_{f}$, 
then the corresponding Newton polygon of $\psi_{f}(F_{G_{2}})$ is of the same type as $\psi_{f_{k}}(F_{G_{2}})$. The second condition then says the line segment of smallest slope in $\psi_{f}(F_{k,0})$ is of the same type as that of $f_{k}$. That is, it has the same multiplicity and it is related to the slopes of $\psi_{f}(F_{G_{2}})$ in the same way.  
Finally, the third condition says that the line segments induced by filtration types with more branches are not attained by $\mathrm{trop}(f)$: either the line segment of smallest slope has the wrong multiplicity or the slope is wrong.  
\end{rem}

\begin{lemma}\label{PolyhedraIsomorphic}
We have $\mathrm{trop}(f_{k})\in{I(k,0)}$. If $\mathrm{trop}(f_{k})\in{I(k',0)}$ for another $f_{k'}\in\mathcal{G}_{k,0}$, then we have that the marked tree filtrations of $f_{k}$ and $f_{k'}$ are isomorphic up to height $a_{1}$. 
\end{lemma}
\begin{proof}
The first directly follows from the definitions above and the Newton polygon theorem. The second follows from Proposition \ref{MainProposition4}.   
\end{proof}

We now consider the set of all $f_{i}\in{\mathcal{G}_{k,0}}$ such that $\mathrm{trop}(f_{i})\in{I(k,0)}$. We denote this set by $\mathcal{G}_{k,1}$. By the Lemma above, the polynomials in this set have marked tree filtrations that are isomorphic up to height $a_{1}$. 
We now repeat the procedures that led us to Definition \ref{PolyhedraHeightZero} and similarly obtain a subspace $I(k,1)$ using the $f_{i}\in\mathcal{G}_{0,1}$. By considering the $f_{i}\in\mathcal{G}_{k,1}$ whose tropicalizations land in $I(k,1)$, we then obtain $\mathcal{G}_{k,2}$, which gives a subspace $I(k,2)$ and so on. This gives us generalized polyhedra up to any height.  
\begin{mydef}
{\bf{[Generalized polyhedra up to finite heights II]}}
The generalized polyhedron for $f_{k}$ up to height $c_{k,r}$ is the polyhedron $I(k,r)$. 

\end{mydef}


\begin{prop}\label{PropositionStratification}
Let $f_{k}\in\mathcal{G}$ be any polynomial and write $a_{r_{k}}$ for its largest branch height. We then have $\mathrm{trop}(f_{k})\in{I(k,r_{k})}$ and $\mathrm{trop}(f_{k})\notin{I(k',r_{k'})}$ for any other $f_{k'}\in\mathcal{G}$.
\end{prop}
\begin{proof}
Indeed, as in Lemma \ref{PolyhedraIsomorphic} we can show that $\mathrm{trop}(f_{k})\notin{I(k',r_{k'}-1)}$, since otherwise the tree associated to $f_{k}$ would be isomorphic to that of $f_{k'}$ up to height $r_{k'}$ by Proposition \ref{MainProposition4}, a contradiction. Since $I(k',r_{k'})\subset{I(k',r_{k'}-1)}$, we obtain the statement of the proposition.  
\end{proof}

By Proposition \ref{PropositionStratification}, we find that the associated polyhedron for every filtration type is unique. This gives a theoretical decomposition of the space of trees of any given degree $d$ in terms of half-spaces. In practice one can stop the process above much sooner, see for instance the examples later on in this section. 

\begin{rem}
To distinguish between two polynomials that have the same filtration type, we only need to look at the corresponding Newton polygon of $F_{G_{2}}$. The slopes of these polygons give us the branch heights, which in turn gives us the desired marked tree filtration. In other words, to navigate the moduli space of trees with the same filtration type, we accordingly change the branch heights. To go from one filtration type to the other, we can let the difference between two of these branch heights go to zero. It would be interesting to investigate these limits from a moduli space point of view as in \cite{ACP2015}. 



\end{rem}

\subsubsection{Trivial trees}

We first show that the half-spaces introduced in Definition \ref{NPHalfSpace} can detect {\it{trees of good reduction}}. The nomenclature will be explained in Section \ref{SectionPotentialReduction}. 

\begin{prop}
Let $f(x)$ 
be a separable polynomial over $K$. Let $G_{2}$ be a graph of order two with generating polynomial $F_{G_{2}}$ of degree $d(2)$. Then the phylogenetic type of the tree associated to $f(x)$ is trivial if and only if $\pi_{2}(\mathrm{trop}(f))\in{I(P_{0},P_{d(2)})}$. 
\end{prop}
\begin{proof}
If the phylogenetic type of the tree is trivial, then $v(\alpha_{i}-\alpha_{j})$ is the same for all $i,j$. This means that the Newton polygon of $F_{G_{2}}$ is a single line segment. Conversely, suppose that the Newton polygon of $F_{G_{2}}$ contains a single line segment. Then all of the $v(\alpha_{i}-\alpha_{j})$ are the same, which means that the tree associated to $f(x)$ is trivial. 
\end{proof}







\subsubsection{Polynomials of degree three }\label{SectionDegreeThree} 

Consider a separable polynomial of degree $d=3$: 
\begin{equation}
f(x)=x^{3}-a_{1}x^2+a_{2}x-a_{3}.
\end{equation}
%
 The marked tree filtration corresponding to the roots of $f$ in the sense of Definition \ref{DefinitionTreePolynomial} has $3$ unmarked points and one marked point. There are exactly two types, see Figure \ref{TreeN3}. 
 \begin{figure}
\centering
\includegraphics[width=9cm]{{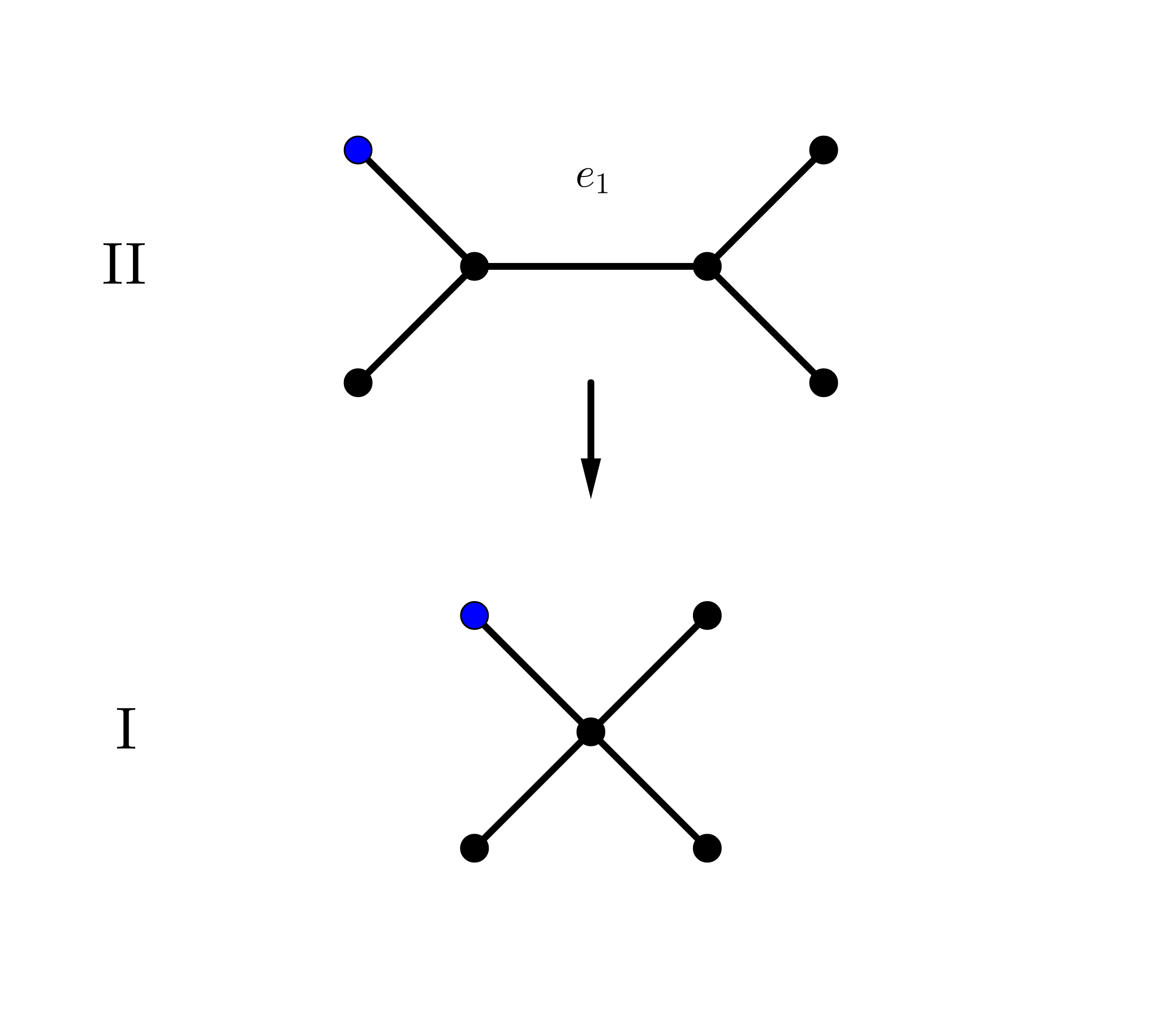}}
\caption{\label{TreeN3}
The two phylogenetic types for a polynomial of degree $3$ considered in Section \ref{SectionDegreeThree}. The blue point corresponds to $\infty$ and the arrow corresponds to the contraction of one of the edges. } 
\end{figure}
 In this case, we can detect the two types using the trivially weighted graph $G_{2}$ 
 of order two. 
 The generating polynomial $F_{G_{2}}$ has degree $3$ and there are exactly two types of Newton polygons: either there is a point where the slope changes (corresponding to a tree of type $\mathrm{II}$), or there is no such point (corresponding to a tree of type $\mathrm{I}$). 
 
 We write $F_{G_{2}}=\sum_{i=0}^{3}{c_{3-i}x^{i}}$ for the generating polynomial and $b_{i}$ for the valuations of its coefficients.  
 Explicitly, we have 
 \begin{eqnarray*}
 c_{3}&=&-a_{2}^2a_{1}^2+4a_{2}^3+4a_{3}a_{1}^3-18a_{1}a_{2}a_{3}+27a_{3}^2,\\
 c_{2}&=&a_{1}^4-6a_{2}a_{1}^2+9a_{2}^2.
 \end{eqnarray*}
 Note that $c_{3}$ is the discriminant of $f(x)$. We write $P_{i}=(i,b_{3-i})$ for the points in the Newton polygon of $F_{G_{2}}$.  
If the Newton polygon of $F_{G_{2}}$ has a breaking point, then it occurs at $P_{1}=(1,b_{2})$. In terms of the notation introduced in the previous section, we thus have the following description of the tree-types: 
 \begin{center}
 \begin{tabular}{|c| c|c| }
 \hline
 Tree type & Polyhedron & $2\ell(e_{1})$ \\
 \hline
 I & $I(P_{0},P_{3})$ & -\\
 \hline
 II & $I(P_{1},P_{3})$ & $\dfrac{2b_{3}-3b_{2}}{2}$\\
 \hline
 \end{tabular}
 \end{center}
More explicitly, we can describe these Newton polygon half-spaces as follows. The line through $(1,b_{2})$ and $(3,0)$ attains the value $3/2\cdot{}b_{3}$ at $0$, so we have 
  \begin{equation}
  3b_{2}<2b_{3}
  \end{equation} 
  if and only if the tree is of type $\mathrm{II}$. We have $3b_{2}\geq{}2b_{3}$ if and only if the tree is of type $\mathrm{I}$. 

 We now transform this criterion to one that might be more familiar to the reader.  
  Let 
  \begin{equation}
  j_{\mathrm{trop}}=\dfrac{c_{2}^3}{c_{3}^2}.
  \end{equation} 
  Then the tree-type of $f(x)$ is $\mathrm{II}$ if and only if $v(j_{\mathrm{trop}})<0$. 
We can re-interpret this using the theory of elliptic curves as follows. 
Let $E$ be the elliptic curve defined by
  \begin{equation}
  y^2=f(x)
  \end{equation} 
  and let $j(E)$ be its $j$-invariant.  We then have
  \begin{equation}\label{JInvariantModified}
  2^{16}j_{\mathrm{trop}}=j(E)^2.
  \end{equation}
If the residue characteristic of $K$ is not $2$, 
then the tree associated to $f(x)$ has type II if and only if $E$ has multiplicative reduction. This will follow from Proposition \ref{CriterionPotGoodRed}, but see also \cite[Proposition 3.10]{FaithTropEll} for a proof that works when the residue characteristic is furthermore not equal to $3$.    
We will see in Section \ref{Section3} that we can in general characterize the reduction types of superelliptic curves using the tree type of $f(x)$. 

 \subsubsection{Polynomials of degree four}\label{SectionDegreeFour}

 We write 
 \begin{equation}
 f(x)=x^4-a_{1}x^3+a_{2}x^2-a_{3}x+a_{4}.
 \end{equation}
If we consider unmarked phylogenetic trees with $5$ points, then there are three types. If we add a marked point, then there are five types, see 
Figure \ref{TreeN4}. Here we used an Arabic numeral to denote the phylogenetic type and a Roman numeral to denote the non-leaf vertex that the marked point is attached to.  
 \begin{figure}
\centering
\includegraphics[width=11cm]{{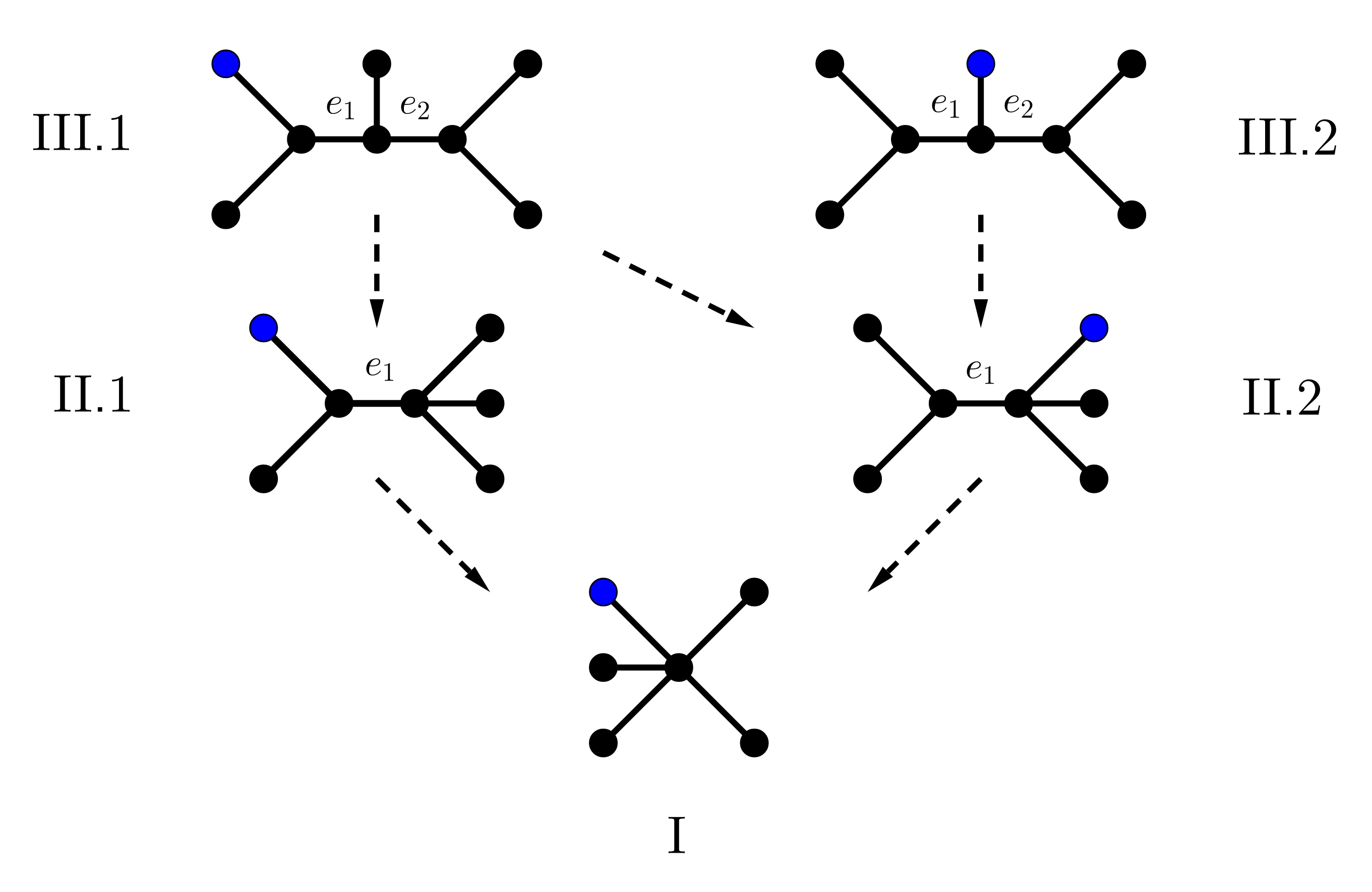}}
\caption{\label{TreeN4}
The five marked phylogenetic types for a polynomial of degree $4$ considered in Section \ref{SectionDegreeFour}. The blue point corresponds to $\infty$ and an arrow corresponds to a contraction of one of the edges. } 
\end{figure}
 We consider the invariants corresponding to a graph $G$ of order two 
 and trivial weight function. The generating polynomial $F_{G,1}=\sum_{i=0}^{6}c_{n-i}x^{i}$ has degree $|H_{G,1}|={{4}\choose{2}}=6$, we will only need $c_{3}$, $c_{4}$, $c_{5}$ and $c_{6}$. Explicitly, we have the following:
 \begin{align*}
c_{3}:= &-a_{1}^6 + 8a_{1}^4a_{2} + 8a_{1}^3a_{3} - 24a_{1}^2a_{2}^2 - \\
 {}&6a_{1}^2a_{4} -30a_{1}a_{2}a_{3} + 28a_{2}^3 + 16a_{2}a_{4} + 26a_{3}^2,\\
c_{4}:=&-6a_{1}^5a_{3} + 2a_{1}^4a_{2}^2 + 6a_{1}^4a_{4} + \\
&38a_{1}^3a_{2}a_{3} -
    12a_{1}^2a_{2}^3 - 32a_{1}^2a_{2}a_{4} - 25a_{1}^2a_{3}^2 - \\
    &54a_{1}a_{2}^2a_{3} +
    56a_{1}a_{3}a_{4} + 17a_{2}^4 + 24a_{2}^2a_{4} + 48a_{2}a_{3}^2 - 112a_{4}^2,\\
    c_{5}:=&-9a_{1}^4a_{3}^2 + 6a_{1}^3a_{2}^2a_{3} + 18a_{1}^3a_{3}a_{4} - a_{1}^2a_{2}^4 -
    6a_{1}^2a_{2}^2a_{4} +\\
    & 42a_{1}^2a_{2}a_{3}^2 + 72a_{1}^2a_{4}^2 - 26a_{1}a_{2}^3a_{3} -
    120a_{1}a_{2}a_{3}a_{4} - 54a_{1}a_{3}^3 +\\
    & 4a_{2}^5 + 32a_{2}^3a_{4} + 18a_{2}^2a_{3}^2
    - 192a_{2}a_{4}^2 + 216a_{3}^2a_{4},\\
    c_{6}:=&-27a_{1}^4a_{4}^2 + 18a_{1}^3a_{2}a_{3}a_{4} - 4a_{1}^3a_{3}^3 -\\
    & 4a_{1}^2a_{2}^3a_{4}
    + a_{1}^2a_{2}^2a_{3}^2 + 144a_{1}^2a_{2}a_{4}^2 - 6a_{1}^2a_{3}^2a_{4} -
    80a_{1}a_{2}^2a_{3}a_{4} + \\
    &18a_{1}a_{2}a_{3}^3 - 192a_{1}a_{3}a_{4}^2 + 16a_{2}^4a_{4} -
    4a_{2}^3a_{3}^2 - 128a_{2}^2a_{4}^2 + \\
    &144a_{2}a_{3}^2a_{4} - 27a_{3}^4 + 256a_{4}^3.
\end{align*}
Here the polynomial $c_{6}$ is again the discriminant of $f(x)$. We write $b_{i}=v(c_{i})$ for the corresponding valuations.  
 The following half-spaces now describe the different tree types:  
 
   \begin{center}
 \begin{tabular}{|c| c|c|c| }
 \hline
 \textnormal{Tree type} & \textnormal{Polyhedra} & $2\ell(e_{1})$ & $2\ell(e_{2})$ \\
 \hline
 I & $I(P_{0},P_{6})$ & -- & -- \\
 \hline
 II.1 & $I(P_{3},P_{6})$, $I(P_{0},P_{3})$ & $\dfrac{b_{6}-2b_{3}}{3}$ & -- \\
 \hline
 II.2 & $I(P_{1},P_{6})$ & $\dfrac{5b_{6}-6b_{5}}{5}$ & --\\
 \hline
 III.1 & $I(P_{3},P_{6})$, $I(P_{1},P_{3})$ & $\dfrac{3b_{5}-5b_{3}}{6}$ & $\dfrac{2b_{6}-3b_{5}+b_{3}}{2}$\\
 \hline
 III.2 & $I(P_{2},P_{6})$ & $\dfrac{4b_{5}-5b_{4}}{4}$ & $\dfrac{4b_{6}-4b_{5}-b_{4}}{4}$\\
      \hline
 \end{tabular}
\end{center}

Here the edge lengths were calculated by first determining the slope of the trivial segment (which is $b_{5}/5$ for instance for trees of type II.2). One then subtracts this from the other slopes to obtain the edge lengths. We can write out the above Newton polygon half-spaces explicitly as follows. Let 
 \begin{equation}
h(x):=h_{P_{1},P_{3}}(x)=-\dfrac{(b_{5}-b_{3})}{2}x+\dfrac{3b_{5}-b_{3}}{2}
 \end{equation}
 By eliminating some of the redundant conditions in the above table, we arrive at the following equivalent description. 
 
 
  \begin{center}
 \begin{tabular}{|c| c| }
 \hline
 \textnormal{Tree type} & \textnormal{Polyhedra} \\
 \hline
 I & $\dfrac{i}{6}b_{6}\leq{b_{i}}$ for $i\in\{3,4,5\}$ \\
 \hline
   II.1 & $\dfrac{i}{3}b_{3}<{b_{i}}$ for $i\in\{4,5,6\}$ and $
    h(0)\geq{}b_{6}$ \\ 
    \hline
   II.2 & $\dfrac{i}{5}b_{5}\leq{b_{i}}$ for $i\in\{2,3\}$ and $\dfrac{6}{5}b_{5}<b_{6}$\\
    \hline
    III.1 & $\dfrac{i}{3}b_{3}<{b_{i}}$ for all $i\in\{4,5,6\}$ and $h(0)<b_{6}$. \\
     \hline
     III.2 &  $\dfrac{3}{4}b_{4}\leq{b_{3}}$ and $\dfrac{i}{4}b_{4}<b_{i}$ for $i\in\{5,6\}$ \\
     \hline
 \end{tabular}
\end{center}

 \subsubsection{Polynomials of degree five}\label{PolynomialsDegreeFive} \label{SectionDegreeFive}
 
 We write 
 \begin{equation}
 f(x)=x^5-a_{1}x^4+a_{2}x^3-a_{3}x^2+a_{4}x-a_{5}.
 \end{equation}
 
 There are $7$ types of unmarked phylogenetic trees with $6$ leaves, see Figure \ref{TreeN5}.
  \begin{figure}
\centering
\includegraphics[width=11cm]{{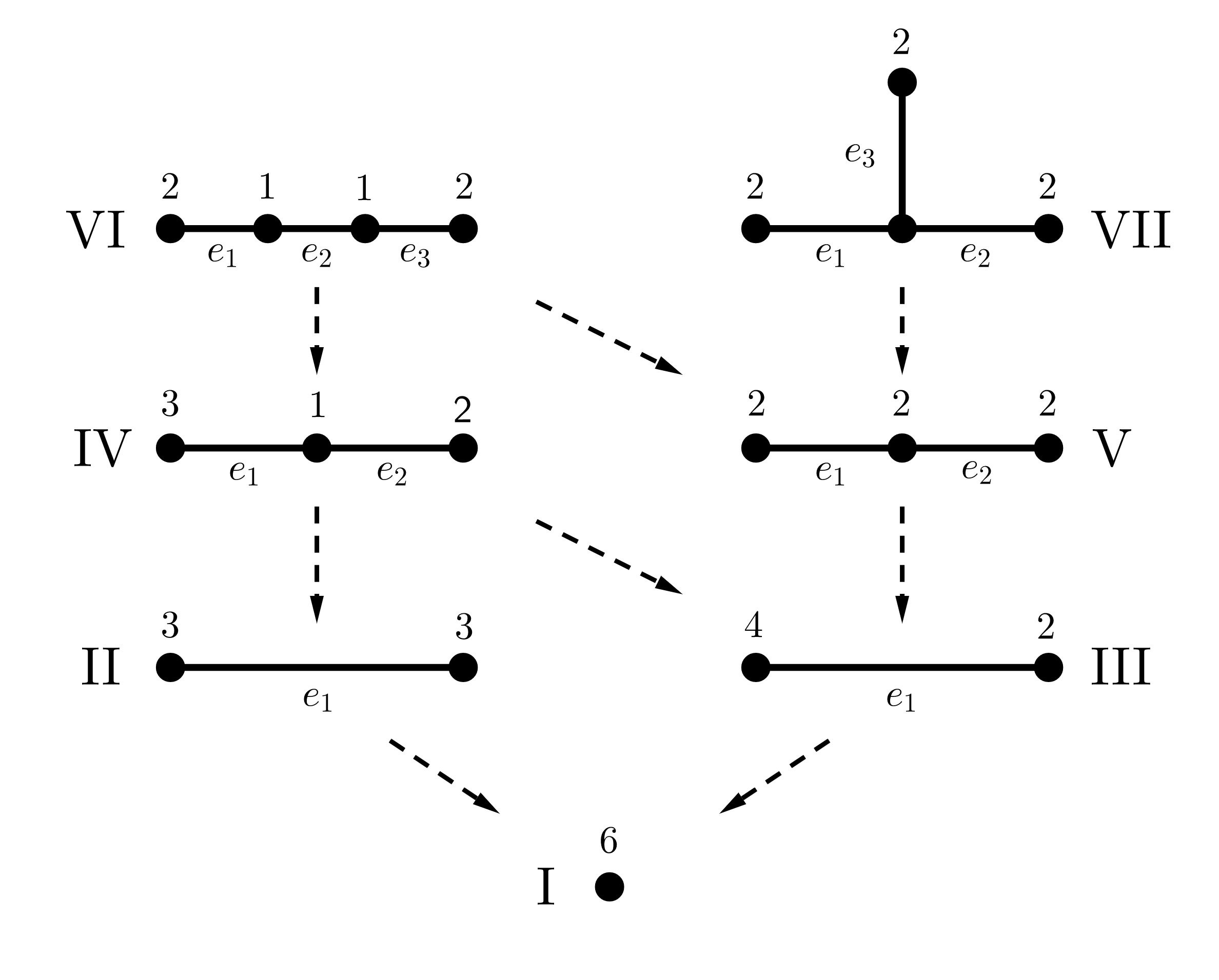}}
\caption{\label{TreeN5}
The seven unmarked phylogenetic types for a polynomial of degree $5$ considered in Section \ref{PolynomialsDegreeFive}. The arrows correspond to the contraction of one of the edges and the numbers next to the vertices correspond to the number of leaves attached to the vertex. } 
\end{figure}
By adding a marked point, we obtain $12$ types. We use a Roman numeral to denote the phylogenetic type, we add an Arabic numeral to denote the marked vertex. Here we read the  vertices from left to right. We now encounter a new phenomenon: the invariants associated to a graph of order two do not give enough information to distinguish between the tree types $\mathrm{IV}.2$ and $\mathrm{VI}.2$. For a concrete example, suppose that $\mathrm{char}(k)\neq{2,3}$ and consider the polynomials
\begin{eqnarray*}
f_{1}&=&(x-(1+\varpi))(x-(1+2\varpi))(x-\varpi)(x-\varpi^{2})(x-2\varpi^{2})\\
f_{2}&=&(x-\varpi)(x-2\varpi)(x-3\varpi)(x-(1+\varpi^2))(x-(1+2\varpi^2)).
\end{eqnarray*}
Here $\varpi$ is some element in $K^{*}$ with $\mathrm{val}(\varpi)>0$. The trees associated to $f_{1}$ and $f_{2}$ are the ones labeled I and II in Figure \ref{TreesFiveLeaves}. Their branch heights are  $0$, $a_{1}=\mathrm{val}(\varpi)$ and $a_{2}=2\cdot{\mathrm{val}(\varpi)}$. The Newton polygons of the polynomials $F_{G_{2},1}$ for these marked tree filtrations are identical: they both have a line segment of slope $0$ and multiplicity $6$, a line segment of slope $\mathrm{val}(\varpi)$ and multiplicity $3$, and a line segment of slope $2\cdot{}\mathrm{val}(\varpi)$ and multiplicity $1$. To distinguish between these trees, we will use the invariants associated to a graph of order three.  

 
Before we study these higher-order invariants, we first describe the trees that can be distinguished using $F_{G_{2},1}=\prod_{i<j}(x-[ij])=\sum_{i=0}^{10}c_{10-i,2}x^{i}$. The explicit forms of the $c_{i}$ can be found in Section \ref{Appendix}. We write $b_{10-i}=b_{10-i,2}$ for the valuations of these coefficients and $S=\{P_{i}\}$ with $P_{i}=(i,b_{10-i,2})$. We have the following preliminary subdivision of the tree-types:
   \begin{center}
 \begin{tabular}{|c| c| }
 \hline
 \textnormal{Polyhedra} & \textnormal{Tree types} \\
 \hline
$I(P_{6},P_{10})$ &  III.2, IV.3, V.1, VI.1, VII. \\
 \hline
 $I(P_{4},P_{10})$ & IV.2, VI.2 \\
 \hline
$I(P_{3},P_{10})$ & II, IV.1 \\
 \hline
$I(P_{2},P_{10})$ & V.2 \\
 \hline
$I(P_{1},P_{10})$ & III.1\\
 \hline
$I(P_{0},P_{10})$ & I \\
 \hline 
 \end{tabular}
\end{center} 

We now subdivide the tree types satisfying $I(P_{6},P_{10})$ by adding conditions for the extra line segments in the Newton polygon of $F_{G_{2}}$. The result is as follows: 

   \begin{center}
 \begin{tabular}{|c| c|c|c|c| }
 \hline
 \textnormal{Tree type} & \textnormal{Extra polyhedra} & $2\ell(e_{1})$ & $2\ell(e_{2})$ & $2\ell(e_{3})$ \\ 
 \hline
III.2 & $I(P_{0},P_{6})$ & $\dfrac{2b_{10}-5b_{4}}{12}$ & -- & -- \\ 
\hline
IV.3 & $I(P_{0},P_{3})$, $I(P_{3},P_{6})$ & $\dfrac{b_{10}-2b_{7}+b_{4}}{3}$  & $\dfrac{4b_{7}-7b_{4}}{12}$ & -- \\
\hline
V.1 & $I(P_{1},P_{6})$ & $\dfrac{4b_{9}-9b_{4}}{20}$ & $\dfrac{5b_{10}-6b_{9}+b_{4}}{5}$ & --  \\ 
\hline
VI.1 & $I(P_{3},P_{6})$, $I(P_{1},P_{3})$ & $\dfrac{4b_{7}-7b_{4}}{12}$ & $\dfrac{3b_{9}-5b_{7}+2b_{4}}{6}$ & $\dfrac{2b_{10}-3b_{9}+b_{7}}{2}$ \\ 
\hline
VII & $I(P_{2},P_{6})$ & $\dfrac{b_{8}-2b_{4}}{4}$ & $\dfrac{4b_{9}-4b_{8}-b_{4}}{4}$ & $\dfrac{4b_{10}-4b_{9}-b_{4}}{4}$ \\
\hline
 \end{tabular}
\end{center}

We can similarly treat trees of type II and IV.1:

   \begin{center}
 \begin{tabular}{|c| c|c|c| }
 \hline
 \textnormal{Tree type} & \textnormal{Extra polyhedra} & $2\ell(e_{1})$ & $2\ell(e_{2})$ \\ 
 \hline
 II & $I(P_{0},P_{3})$ & $\dfrac{7b_{10}-10b_{7}}{21}$ & -- \\
 \hline
 IV.1 & $I(P_{1},P_{3})$ & $\dfrac{7b_{9}-9b_{7}}{14}$ & $\dfrac{2b_{10}-3b_{9}+b_{7}}{2}$ \\
 \hline
 \end{tabular}
\end{center}

To distinguish between the types in $I(P_{4},P_{10})$, we use higher-order invariants. 
Consider the complete graph $G_{3}$ on three vertices. 
The corresponding polynomial is $F_{G_{3},1}=\prod_{}(x-\sigma_{i}([12][13][23]))$, where the product is over a set of representatives $\sigma_{i}$ of $S_{5}/H$ and $H$ is the stabilizer of $[12][13][23]$. 
This stabilizer has order ${5\choose{3}}=12$, so $F_{G_{3},1}$ has degree $5!/12=10$. 
We denote the coefficients by $c_{i,3}$, where the labeling is so that $F_{G_{3},1}=\sum_{i=0}^{10}c_{10-i,3}x^{i}$. Their explicit forms can be found in Section \ref{Appendix}. As before, the valuations of the $c_{i,3}$ are denoted by $b_{i,3}$. For the Newton polygon conditions as in Definition \ref{NPHalfSpace}, we will use the set $\{P_{i,3}\}$ where $P_{i,3}=(i,b_{10-i,3})$. 
We consider three cases: 
\begin{enumerate}
\item $\ell(e_{1})=\ell(e_{2})$,
\item $\ell(e_{1})>\ell(e_{2})$,
\item $\ell(e_{1})<\ell(e_{2})$.
\end{enumerate}
We write $\mathrm{IV.2.i}$ for these. The Newton polygons of $F_{G_{2}}$ for the first two are unique, so we don't need higher invariants here. The third however has overlap with $\mathrm{VI.2}$ (as we saw earlier), so we add a half-space for $F_{G_{3}}$. The result is:  


  \begin{center}
 \begin{tabular}{|c| c|c|c| }
 \hline
 \textnormal{Tree type} & \textnormal{Extra polyhedra} & $2\ell(e_{1})$ & $2\ell(e_{2})$\\ 
 \hline
 IV.2.1 & $I(P_{0},P_{4})$ & $\dfrac{3b_{10}-5b_{6}}{12}$ & $\dfrac{3b_{10}-5b_{6}}{12}$ \\ 
 \hline
IV.2.2 & $I(P_{0},P_{3})$, $I(P_{3},P_{4})$ & $\dfrac{b_{10}-4b_{7}+3b_{6}}{3}$ & $\dfrac{6b_{7}-7b_{6}}{6}$  \\ 
 \hline
IV.2.3 & $I(P_{1},P_{4})$, $I(P_{4,3},P_{10,3})$ & $\dfrac{2b_{9}-3b_{6}}{6}$ & $\dfrac{3b_{10}-4b_{9}+b_{6}}{3}$\\ 
 \hline
 \end{tabular}
\end{center}

For trees of type VI.2, we have a similar subdivision: 


  \begin{center}
 \begin{tabular}{|c| c|c|c|c| }
 \hline
 \textnormal{Tree type} & \textnormal{Extra polyhedra} & $2\ell(e_{1})$ & $2\ell(e_{2})$ & $2\ell(e_{3})$ \\ 
 \hline
 VI.2.1 & $I(P_{3},P_{4})$, $I(P_{1},P_{3})$ & $\dfrac{6b_{7}-7b_{6}}{6}$ & $\dfrac{b_{9}-3b_{7}+2b_{6}}{2}$  & $\dfrac{2b_{10}-3b_{9}+b_{7}}{2}$\\
 \hline
VI.2.2 & $I(P_{1},P_{4})$, $I(P_{3,3},P_{10,3})$ & $\dfrac{2b_{9}-3b_{6}}{6}$ & $\dfrac{2b_{9}-3b_{6}}{6}$ & $\dfrac{3b_{10}-4b_{9}+b_{6}}{3}$ \\
 \hline
VI.2.3 & $I(P_{2},P_{4})$, $I(P_{1},P_{2})$, & $\dfrac{2b_{9}-3b_{8}+b_{6}}{2}$ & $\dfrac{3b_{8}-4b_{6}}{6}$ & $\dfrac{b_{10}-2b_{9}+b_{8}}{1}$ \\
& $I(P_{6,3},P_{10,3})$, $I(P_{3,3},P_{6,3})$ & {} & {} & {}\\
 \hline
 VI.2.4 & $I(P_{2},P_{4})$, $I(P_{0},P_{2})$ & $\dfrac{b_{10}-2b_{8}+b_{6}}{2}$ & $\dfrac{3b_{8}-4b_{6}}{6}$ & $\dfrac{b_{10}-2b_{8}+b_{6}}{2}$ \\
 \hline
 VI.2.5 & $I(P_{2},P_{4})$, $I(P_{1},P_{2})$ & $\dfrac{b_{10}-2b_{9}+b_{8}}{1}$ & $\dfrac{3b_{8}-4b_{6}}{6}$ & $\dfrac{2b_{9}-3b_{8}+b_{6}}{2}$  \\
 & $I(P_{6,3},P_{10,3})$, $I(P_{4,3},P_{6,3})$ & {}& {} & {}\\ 
 \hline
 \end{tabular}
\end{center}

The five different filtration types correspond to the different configurations of $\ell(e_{1})$, $\ell(e_{2})$ and $\ell(e_{3})$. Here the labeling is as in Figure \ref{TreeN5}. As an exercise, the reader is invited to draw the five filtration types and to find the connection with the above table.  

\section{The semistable reduction type of superelliptic curves}\label{Section3}

In this section, we use the invariants introduced in the previous section to determine the semistable reduction type of a superelliptic curve $X_{n}$ defined by $y^{n}=f(x)$, where $f(x)$ is a separable polynomial and $n$ is coprime to the residue characteristic. By Theorem \ref{MainThm1}, the tropical invariants give us the tree type corresponding to the roots of $f(x)$ and we show here that this tree type is enough to determine the minimal skeleton of $X_{n}$. The latter was known in the discretely valued case by the results in \cite{supertrop}, we extend these results to the non-discrete case here using the machinery of Berkovich spaces. The main theoretical result we will use is the tame simultaneous semistable reduction theorem: 
\begin{theorem}\label{SimultaneousSemSta} {\bf{[Tame simultaneous semistable reduction theorem]}}
Let $\phi:X'\rightarrow{X}$ be a residually tame covering of smooth proper connected algebraic curves over $K$ and let $\mathcal{X}$ be a (strongly) semistable model for $(X,B)$, where $B$ is the branch locus of $\phi$. 
Let $\mathcal{X}'$ be the normalization of $\mathcal{X}$ in the function field $K(X')$. Then $\mathcal{X}'$ is (strongly) semistable and $\mathcal{X}'\rightarrow{\mathcal{X}}$ is a finite morphism of (strongly) semistable models over $R$.
\end{theorem} 

\begin{proof}
See \cite[Theorem 3.1]{TropFund1} for the result over a nondiscrete valuation ring and \cite[Theorem 1.1]{NewtonPuiseuxSemSta} for its counterpart over a discrete valuation ring. Weaker versions of this theorem can be found in \cite[Proposition 4.30]{liu2} and \cite[Theorem 2.3]{liu1}, where the covers are assumed to be Galois with Galois group $G$ satisfying $\mathrm{char}(k)\nmid{|G|}$. 
\end{proof}

In terms of the terminology of 
\cite{TropFund1} and \cite{ABBR1}, this theorem says that for a residually tame covering $\phi: X'\to{X}$, any 
semistable vertex set $V$ of $(X,B)$ lifts to a semistable vertex set $V'$ of $(X',\phi^{-1}(B))$. Here a semistable vertex set of a marked curve $(X,B)$ is a finite set of type-$2$ points of the Berkovich analytification $X^{\mathrm{an}}$ such that $X^{\mathrm{an}}\backslash{V}$ is a disjoint union of open disks and finitely many open annuli. Here the points in $B$ are in distinct open disks. A similar result also holds for strongly semistable vertex sets.  

For a superelliptic curve $X_{n}$ with covering map $X_{n}\to\mathbf{P}^{1}$, we impose the following conditions so that we can use 
Theorem \ref{SimultaneousSemSta}. We suppose that either $\mathrm{char}(k)=0$ or $\mathrm{char}(k)$ is coprime to $n$ if $\mathrm{char}(k)\neq{0}$. The morphism $\phi$ is then automatically residually tame. For the branch locus $B$, we easily find that $B\subset{Z(f(x))}\cup\{\infty\}$. 
We now consider the metric tree in $\mathbf{P}^{1,\mathrm{an}}$ associated to the set $Z(f(x))\cup\{\infty\}\subset{\mathbf{P}^{1}(K)}$ by the construction in Remark \ref{TreeBerkovich2}. 
By taking the points in this tree of valence not equal to two, we obtain a strongly semistable vertex $V(\Sigma)$ of $(\mathbf{P}^{1},B)$ with skeleton $\Sigma$. 
For the remainder of this section, we use this skeleton $\Sigma$, which lifts to a skeleton $\Sigma(X_{n})$ of the superelliptic curve given by $y^{n}=f(x)$ by Theorem \ref{SimultaneousSemSta}. We can enhance the map of skeleta $\Sigma(X_{n})\to\Sigma$ to a tame covering of metrized complexes using \cite[Lemma 4.33]{ABBR1}. This in particular implies that the local Riemann-Hurwitz formulas hold, see \cite[Section 2.12]{ABBR1}.  

To reconstruct $\Sigma(X_{n})$ from $\Sigma$, 
we determine the following: 

\begin{enumerate}
\item ({\it{Edges}}) The number of open annuli lying over an open annulus. See \ref{FormulaEdges}. 
\item ({\it{Vertices}}) The number of points lying over a type-$2$ point of $\Sigma$.  See \ref{FormulaVertices}. 
\item ({\it{Gluing}}) A prescription for gluing the edges and vertices. See \ref{GluingData}.
\end{enumerate} 

The first two follow from the non-archimedean slope formula and local considerations as in \cite{supertrop}. These then also determine the genera of the type-$2$ points in $\Sigma(X_{n})$ by the local Riemann-Hurwitz equations. For the third, we show that the transitive $\mathbf{Z}/n \mathbf{Z}$-action completely determines the underlying graph of $\Sigma(X_{n})$.

After the proof of Theorem \ref{MainThm2}, we prove a criterion for potential good reduction in Proposition \ref{CriterionPotGoodRed}. This proposition says that a superelliptic curve has potential good reduction if and only if the tree of $f(x)$ is trivial. If we replace the zero locus of $f(x)$ by the branch locus of a map $X\to\mathbf{P}^{1}$, then the corresponding statement is false, 
see Remark \ref{CounterexampleRemark}.  We end this section by determining the minimal skeleta of curves $y^{n}=f(x)$ with $\mathrm{deg}(f(x))\leq{5}$.


\subsection{Reconstructing the skeleton}

We start by reviewing the set-up given in the beginning of the proof of \cite[Theorem 1.3]{TropFund1}. The skeleton $\Sigma$ of $\mathbf{P}^{1,\mathrm{an}}$ corresponds to a strongly semistable model $\mathcal{Y}$ over the valuation ring $R$. We write $K(x)\subset{K(x)(\alpha)}$ for the inclusion of function fields corresponding to $X\to\mathbf{P}^{1}$, where $\alpha^{n}=f(x)$. By Theorem \ref{SimultaneousSemSta}, the normalization $\mathcal{X}$ of $\mathcal{Y}$ in $K(X)$ 
is again a strongly semistable model and we have a finite morphism $\mathcal{X}\to\mathcal{Y}$.  
We can calculate the normalization near the generic point $\eta$ corresponding to a type-$2$ point $y$ as follows. 
We first write $f=f_{\eta}\cdot{}\omega$, where $v_{\eta}(f_{\eta})=0$ and $\omega\in{K}$. The element $\alpha_{\eta}=\dfrac{\alpha}{\omega^{1/n}}$ 
is then integral over $\mathcal{O}_{\mathcal{Y},\eta}$, satisfying $\alpha_{\eta}^{n}=f_{\eta}$. By our assumption on the residue characteristic, we find that $\mathcal{O}_{\mathcal{Y},\eta}\subset\mathcal{O}_{\mathcal{Y},\eta}[\alpha_{\eta}]$ is \'{e}tale and thus normal. 
The points lying over $\eta$ are thus described by the scheme
\begin{equation}\label{SchemeLocalPoints}
Z=\mathrm{Spec}(\mathcal{O}_{\mathcal{Y},\eta}[\alpha_{\eta}]),
\end{equation}
or equivalently by its base change over $\mathrm{Spec}(k(\eta))$.

\begin{lemma}\label{CalculationKummer}
Let $r$ be the largest divisor of $n$ such that $\overline{f}_{y}\in{k(\eta)}$ can be written as $\overline{h}^{r}=\overline{f}_{y}$. Then the number of points in $\mathcal{X}$ lying over $\eta\in\mathcal{Y}$ is equal $r$. Let $k=n/r$. Then the residue field extension $k(\eta)\subset{k(\eta')}$ for any point $\eta'$ lying over $\eta$ is $k(\eta)$-isomorphic to $k(\eta)[z]/(z^{k}-\overline{h})$. 
\end{lemma}

\begin{proof}
See \cite[Lemma 4.2(2)]{supertrop}. 
\end{proof}

We will give a more explicit formula for the vertices in Lemma \ref{FormulaVertices}. First, we determine the number of edges lying above an edge. To do this, we switch to the theory of Berkovich spaces and use \cite[Theorem 4.23(2)]{ABBR1}, which says that we can determine the expansion factor using the morphism of residue curves.

In terms of Berkovich spaces, we have the following. The morphism $\phi:X\to\mathbf{P}^{1}$ is Galois with Galois group $G=\mathbf{Z}/n\mathbf{Z}$. By \cite[Theorem 4.23(1)]{ABBR1}, the morphism of Berkovich analytifications $\phi^{\mathrm{an}}:X^{\mathrm{an}}\to\mathbf{P}^{1,\mathrm{an}}$ is piecewise linear. For any bounded line segment $e'$ in $X^{\mathrm{an}}$ where $\phi^{\mathrm{an}}$ is linear, we then have a well-defined expansion factor $d_{e'}$, which is the slope of $\phi^{\mathrm{an}}$ along this segment. Using the results in \cite[Section 2]{BPRa1}, one then sees that this expansion factor is the degree of $\phi^{\mathrm{an}}$ on $e'$. 

We now use this on a specific set of line segments in $X^{\mathrm{an}}$. We will write $\phi:=\phi^{\mathrm{an}}$ to ease notation. The semistable model $\mathcal{Y}$ for $\mathbf{P}^{1}$ gives a set of open annuli $V\simeq{}\mathbf{S}_{+}(a)$ such that the induced covering $\phi^{-1}(V)\to{V}$ is {\it{Kummer}}. That is, for every connected component $U$ of $\phi^{-1}(V)$ there is an $a\in{K^{*}}$ with $\mathrm{val}(a)\geq{0}$ and an isomorphism $U\simeq{\mathbf{S}_{+}(a^{1/r})}$ such that the induced map 
\begin{equation}
\mathbf{S}_{+}(a^{1/r})\to\mathbf{S}_{+}(a)
\end{equation}  
is given by $t\mapsto{t^{r}}$. This integer $r$ is the expansion factor of $\phi$ along the skeleton of $U$ and it is the same for all connected components in $\phi^{-1}(V)$. Indeed, this easily follows from the fact that $\phi^{-1}(V)$ is a $G$-torsor. 
We denote this integer by $d_{e}$, where $e$ is the skeleton of $V$. Using the transitivity of $G$ on the connected components, we then obtain the following formula for the number $n_{e}$ of edges lying over $e$: 
\begin{equation}
n_{e}=\dfrac{n}{d_{e}}.
\end{equation}
We thus see that it suffices to know the expansion factor to know the number of edges.

We now determine this expansion factor. Let $\mathbf{P}^{1,\mathrm{an}}$ be the Berkovich analytification of the projective line and let $\mathrm{Supp}(f)\subset{\mathbf{P}^{1,\mathrm{an}}}$ be the subset of zeros and poles of $f$, considered as type-$1$ points. Consider the function 
\begin{equation}
F=-\mathrm{log}|f|
\end{equation}
on the complement $\mathbf{P}^{1,\mathrm{an}}\backslash\,{\mathrm{Supp}(f)}$. This is piecewise linear by \cite[Theorem 5.15]{ABBR1}.  
Let $y$ be a type-$2$ point in the above complement. 
For any tangent direction $w$ starting at $y$, we then have a well-defined slope $d_{w}(F)$. If $F$ is linear on a bounded line segment $e$ corresponding to $w$, then we write  
\begin{equation}
\delta_{e}(F)=|d_{w}(F)|
\end{equation}
for the absolute value of this slope. For instance, for any open annulus in the decomposition $\mathbf{P}^{1,\mathrm{an}}\backslash{V}$ with skeleton $e$, the function $F$ is automatically linear. 
\begin{lemma}\label{FormulaEdges} Let $e$ be the skeleton of an open annulus in the semistable decomposition determined by $\mathcal{Y}$. Then 
$n_{e}=\mathrm{gcd}(\delta_{e}(F),n)$.
\end{lemma}
\begin{proof}
Let $y$ be a type-$2$ point corresponding to an endpoint of $e$ and let $x\in{X^{\mathrm{an}}_{n}}$ be any point lying over $y$. In terms of the above semistable models, $y$ corresponds to the generic point $\eta_{y}$ of an irreducible component of $\mathcal{Y}_{s}$ and $x$ corresponds to a point $\eta_{x}$ of the scheme 
$Z$ in Equation \ref{SchemeLocalPoints} lying over $\eta_{y}$. By Lemma \ref{CalculationKummer} the extension of residue fields $k(\eta_{y})\subset{k(\eta_{x})}$ is described by the polynomial $z^{k}-\overline{h}$, where $\overline{h}^{r}=\overline{f}_{y}$.  
We denote their residue curves by $C_{x}$ and $C_{y}$. These are the smooth proper curves corresponding to the residue fields of $\eta_{x}$ and $\eta_{y}$ respectively. 
Note that the edge $e$ corresponds to a closed point on the curve $C_{y}$. We denote this point by $z$. By \cite[Theorem 4.23(2)]{ABBR1}, the expansion factor along $e$ is equal to the ramification index of the covering $C_{x}\to{C_{y}}$ at any point lying over $z$. This point $z$ corresponds to a discrete valuation $w_{z}(\cdot{})$ of the function field $k(C_{y})=k(\eta_{y})$. By \cite[Lemma 4.2]{supertrop}, the ramification index is given by the formula
\begin{equation}\label{RamificationDegrees}
e_{z}=\dfrac{k}{\mathrm{gcd}(w_{z}(\overline{h}),k)}.
\end{equation}
By the non-archimedean slope formula \cite[Theorem 5.15(3)]{BPRa1}, we then have $rw_{z}(\overline{h})=w_{z}(\overline{f}_{y})=\delta_{e}(F)$. Using the identity $$\dfrac{1}{r}\cdot{}\mathrm{gcd}(w_{z}(\overline{f}_{y}),n)=\mathrm{gcd}(\dfrac{w_{z}(\overline{f}_{y})}{r},\dfrac{n}{r})$$ from elementary number theory and the expression for $e_{z}$ in Equation \ref{RamificationDegrees}, we then directly obtain the statement of the Lemma. 

\end{proof}

\begin{lemma}\label{FormulaVertices}
Let $y$ be a type-$2$ point, let $m$ be the greatest common divisor of all the $\delta_{e}(F)$ for the outgoing edges $e$ at $y$ and let $n_{y}$ be the number of points lying over $y$. Then 
\begin{equation}
n_{y}=m.
\end{equation}
\end{lemma}

\begin{proof}
By Lemma \ref{CalculationKummer}, the number of points lying over $y$ is equal to the highest divisor $r$ of $n$ such that $\overline{f}_{y}=\overline{h}^{r}$. Using the non-archimedean slope formula \cite[Theorem 5.15]{BPRa1} and the fact that the class group of $\mathbf{P}^{1}_{k}$ is trivial, we then easily find the desired formula. See \cite[Proposition 4.3(3)]{supertrop} for more details. 

\end{proof}

We now consider the third problem of gluing these edges and vertices. 
We start with a vertex $v'\in{V(\Sigma(X_{n}))}$ with image $v\in{V(\Sigma)}$. We then choose an adjacent edge $e$ of $v$ and an edge $e'$ such that $v'$ is connected to $e'$. We then continue with the other endpoint $w$ of $e$ and we choose a vertex $w'$ such that $e'$ is connected to $w'$. Continuing in this way, we obtain a subgraph $\Sigma'$ of $\Sigma(X_{n})$ isomorphic to the tree $\Sigma$. 

We now write $D_{v}$ and $D_{e}$ for the stabilizer of any vertex $v'$ or edge $e'$ lying over $v$ or $e$ respectively. These groups are independent of the vertex or edge chosen since the Galois group is abelian. We also refer to these as the decomposition groups of the vertices and edges. Since $G$ is moreover cyclic, we  find that the order of $D_{v}$ or $D_{e}$ completely determines the subgroup. These orders in turn are determined by the formulae in Lemmas \ref{FormulaEdges} and \ref{FormulaVertices}. Using this, we can reconstruct $\Sigma(X_{n})$ from $\Sigma$: 

\begin{lemma}\label{GluingData}
The underlying graph of the skeleton $\Sigma(X_{n})$ of the superelliptic curve $X_{n}$ is completely determined by the orders of the stabilizers $D_{v}$ and $D_{e}$ for the vertices and edges of $\Sigma$.
\end{lemma}
\begin{proof}
We reconstruct $\Sigma(X_{n})$ from $\Sigma'$ as follows. 
We take $n$ copies of $\Sigma'$, indexed by $G=\mathbf{Z}/n\mathbf{Z}$. We impose an equivalence relation on these $n$ copies $\Sigma'_{i}$ as follows. Two vertices $v'_{i}$ and $v'_{j}$ lying over $v$ are equivalent if and only if the images of $i$ and $j$ in the quotient $G/D_{v}$ are the same. We similarly define the equivalence relation for the edges. The quotient $\bigsqcup{\Sigma'_{i}}/\sim$ then has a natural $G$-action and it is isomorphic to $\Sigma(X_{n})$. The isomorphism is as follows. We send a vertex $v'_{i}$ to the vertex $\sigma_{i}(v')$, where $\sigma_{i}$ is the automorphism corresponding to $i\in\mathbf{Z}/n\mathbf{Z}$. We similarly define the map for the edges. That this map is well-defined follows from the fact that $X^{\mathrm{an}}$ is a $G$-torsor outside the branch locus. 
The transitivity of this action then implies that this map is surjective and the definition using the decomposition groups shows that it is injective. We leave it to the reader to fill in the set-theoretic details.      

\end{proof}

Using Lemma \ref{GluingData}, we can now prove the second Main Theorem of this paper:

\begin{reptheorem}{MainThm2}{\bf{[Main Theorem]}}
For any $n\geq{2}$, let $X_{n}$ be the superelliptic curve over a complete algebraically closed non-archimedean field $K$ defined by $y^{n}=f(x)$, where $f(x)$ is a separable polynomial of degree $d$. We suppose that $\mathrm{gcd}(n,\mathrm{char}(k))=1$. Then the weighted metric graph $\Sigma(X_{n})$ of $X_{n}$ is completely determined by the tropical invariants of $f(x)$.  
\end{reptheorem}   
\begin{proof}
By Theorem \ref{MainThm1}, the tropical invariants completely determine the marked tree filtration 
associated to $f(x)$. 
The absolute slopes $\delta_{e}(F)$ of the piecewise-linear function $F=-\mathrm{log}|f|$ on $\mathbf{P}^{1,\mathrm{an}}$ can then be recovered by a recursive procedure, see \cite[Lemma 2.6]{supertrop}. This determines the orders of the decomposition groups by Lemmas \ref{FormulaEdges} and \ref{FormulaVertices}. By Lemma \ref{GluingData}, we can then recover the underlying graph of the skeleton $\Sigma(X_{n})$.  
Furthermore, the edge lengths are given by 
\begin{equation}\label{EdgeLengthsFormula}
\ell(e')=\dfrac{\ell(e)}{|D_{e}|}
\end{equation}  
and the genera of the vertices $v'$ are given by the local Riemann-Hurwitz formulas given in \cite[Section 2.12]{ABBR1} for generically \'{e}tale morpisms. Here the fact that $\Sigma(X_{n})\to\Sigma$ is generically \'{e}tale is contained in \cite[Lemma 4.33]{ABBR1}.  
This finishes the proof. 
\end{proof}

\begin{rem}
Suppose we are given a superelliptic curve $X_{n}$ over a complete discretely valued non-archimedean field $K$. One can then consider the group of connected components $\Phi(\mathcal{J})$ of the special fiber of the N\'{e}ron model $\mathcal{J}$ of the Jacobian $J$ of $X_{n}$. For a finite extension $K'$ of $K$ over which $X_{n}$ attains semistable reduction, we then have that the group $\Phi(\mathcal{J}_{R'})${\footnote{This is the corresponding group for the N\'{e}ron model of the base change $J\times_{K}{K'}$; note that $\mathcal{J}_{R'}$ cannot simply be recovered as the base of change $\mathcal{J}$ over $R'$.}} 
is completely determined by the metric graph $\Sigma(X_{n})$. Indeed, this follows from the results in \cite{Baker2008}. We thus see that Theorem \ref{MainThm2} can be used to calculate the component groups for the Jacobians of superelliptic curves. In this case the extension $K'$ can be made explicit: we can take the composite of the splitting field of $f(x)$ and $K(\varpi^{1/n})$, where $\varpi$ is a uniformizer of $K$.  
\end{rem}

\subsection{A criterion for potential good reduction and explicit examples}\label{SectionPotentialReduction}

In this section we prove a criterion for potential good reduction for superelliptic curves. Here we say that a curve $X$ has potential good reduction if the minimal skeleton of $X$ consists of a single vertex of genus $g(X)$. After this, we determine the reduction types of the curves $X_{n}: y^{n}=f(x)$ 
for all $f(x)$ of degree $d\leq{5}$ using the results in Section \ref{SectionExamplesTrees}.  

\begin{prop}\label{CriterionPotGoodRed}
Let $X_{n}$ be a superelliptic curve. Then $X_{n}$ has potential good reduction if and only if the phylogenetic type of the tree associated to the branch locus is trivial. 
\end{prop}
\begin{proof}
If the phylogenetic type is trivial, then the lifted skeleton $\Sigma(X_{n})$ consists of only one vertex. Indeed, 
we then have a semistable vertex set of $\mathbf{P}^{1,\mathrm{an}}$ consisting of a single point, which lifts by Theorem \ref{SimultaneousSemSta}, so we see that $X$ has potential good reduction. Note that we did not use that $X_{n}\to\mathbf{P}^{1}$ is superelliptic here. 
Conversely, suppose that $X_{n}$ has potential good reduction. 
Consider the point $v$ in $\mathbf{P}^{1,\mathrm{an}}$ that is the image of the point with good reduction. For any point $P$ in the branch locus, there is a unique path from $P$ to $v$. We claim that these paths do not meet before $v$. This then shows that the tree associated to the branch locus is trivial.

Suppose on the other hand that at least two of these paths meet. We can consider a maximal point $w$ where at least two of these meet, in the sense that all paths from points in the branch locus that go through $w$ have unique tangent directions. We consider a point $w'$ lying over $w$ and  
 claim that $|D_{w'}|=|D_{e'}|$, where $e'$ is any edge lying over the $e$ in the direction of $v$.  Granted this, we can now use the following:
\begin{lemma}\label{RiemannHurwitz}
Let $k$ be a field and let $\mathbf{P}_{k}^{1}\to\mathbf{P}_{k}^{1}$ be a finite Galois morphism with Galois group $\mathbf{Z}/r\mathbf{Z}$ suppose that $(r,\mathrm{char}(k))=1$. If there 
is a branch point with ramification index $r$, then the branch locus is of order two and both points have ramification index $r$. 
\end{lemma}
\begin{proof}
Let $Q$ be the given branch point with $e(Q)=r$ and let $B_{Q}=B\backslash\{Q\}$. Using the Riemann-Hurwitz formula, we obtain 
\begin{equation}
-2=-2r+(r-1)+R,
\end{equation}
where $R=\sum_{P\in{B_{Q}}}{(r-\dfrac{r}{e(P)})}$. 
Let $s=|B_{Q}|$. 
We then have $s\geq{1}$. Rewriting the above equation, we obtain
\begin{equation}\label{ComputationEquationRH}
\sum_{P\in{B_{Q}}}{\dfrac{r}{e(P)}}=(s-1)r+1.
\end{equation} 
Then $e(P)\geq{2}$ and thus $\dfrac{r}{e(P)}\leq{\dfrac{r}{2}}$ for every $P\in{B_{Q}}$. 
If we combine this with \ref{ComputationEquationRH}, then we obtain $(s-1)r+1\leq{\dfrac{sr}{2}}$. But for $s\geq{2}$, this is not true as one easily sees using induction. We conclude that $s=1$, which forces $e(P)=r$, as desired. 

\end{proof}
We now continue the proof. Since $X_{n}$ has potential good reduction, we find that the genus of the curve corresponding to $w'$ is $0$. We are thus in the scenario of Lemma \ref{RiemannHurwitz} with $r=|D_{w'}|$. The assumption that at least two points in the branch locus meet at $w$ now directly gives a contradiction. We thus see that we only have to show that $|D_{w'}|=|D_{e'}|$. The following lemma gives us one inequality: 
\begin{lemma}
For every $e'$ adjacent to $w'$, we have $D_{e'}\subset{D_{w'}}$. 
\end{lemma} 
\begin{proof}
Indeed, suppose that there exists a $\sigma\in{D_{e'}}$ that does not fix $w'$. Continuity of $\sigma$ then implies that $w'$ is sent to the other endpoint $v'$ of $e'$. But then $v'$ and $w'$ are in the same orbit and are thus mapped to the same point in $\mathbf{P}^{1,\mathrm{an}}$, a contradiction. 
\end{proof}
 Suppose now for a contradiction that $|D_{w'}|>|D_{e'}|$. We can then find an automorphism $\sigma\in\mathbf{Z}/n\mathbf{Z}$ that fixes $w'$, but not $e'$. We thus have two distinct edges $e'$ and $\sigma(e')$ starting at $w'$. These two edges can be extended to give two paths from $w'$ to the vertex of good reduction. 
 A moment of reflection shows that the composition of the first with the inverse of the second 
 gives a non-trivial homology class in $\Sigma(X_{n})$, a contradiction. We conclude that none of the paths from points in the branch locus to the vertex $v$ meet before the vertex $v$. This implies that the phylogenetic type of the branch locus is trivial, as desired.

\end{proof}

Using the results in Section \ref{SectionExamplesTrees}, we now immediately obtain the following practical criterion for potential good reduction: 

\begin{cor}
Consider a superelliptic curve $X_{n}:y^{n}=f(x)$ and let $G_{2}$ be a trivially weighted subgraph of order two with generating polynomial $F_{G_{2}}$. 
Then $X_{n}$ has potential good reduction if and only if $\pi_{2}(\mathrm{trop}(f))\in{I(P_{0},P_{d(2)})}$. 
\end{cor}

\begin{rem}\label{CounterexampleRemark}
Given the result in Proposition \ref{CriterionPotGoodRed}, one might be led to hypothesize the following criterion for general coverings: a curve $X$ admitting a residually tame morphism $\psi:X\to\mathbf{P}^{1}$ has potential good reduction if and only if the tree corresponding to the branch locus of $\psi$ is phylogenetically trivial. As we saw in the proof of Proposition \ref{CriterionPotGoodRed}, this is true in one direction. The other direction does not hold however  by the following counterexample. Let $E$ be the elliptic curve defined by $y^2=x^3+Ax+B$, where $v(A)>0$, $v(B)=0$, $v(\Delta)=0$ and $\mathrm{char}(k)\neq{2,3}$. We also assume that $A\neq{0}$.  
Then $E$ has potential good reduction, but the branch locus of the degree three covering $(x,y)\mapsto{y}$ gives a nontrivial tree. See \cite[Theorem 10.6.1(3)]{tropabelian} for the details in the discretely valued case. 



\end{rem}

We now work out the semistable reduction types for superelliptic curves $X_{n}$ for $d\leq{5}$ using the tree data in Section \ref{SectionExamplesTrees}.  


\begin{exa}\label{ExampleD3}
Let $X_{n}$ be as in Theorem \ref{MainThm2} with $d=3$, so that $X_{n}$ is described by $y^{n}=f(x)=x^3-a_{1}x^2+a_{2}x-a_{3}$. The tree types of $f(x)$ including the edge lengths are given by the half-spaces in Section \ref{SectionDegreeThree}. For type II, the slope of the piecewise linear function $F=\mathrm{log}|f|$ along the nontrivial segment is two. We use this together with Lemmas \ref{FormulaEdges} and \ref{FormulaVertices} to determine the reduction types.  

Suppose first that $n\not\equiv{0}\bmod{3}$. Then $g(X_{n})=n-1$. 
  \begin{figure}
\centering
\includegraphics[width=11cm]{{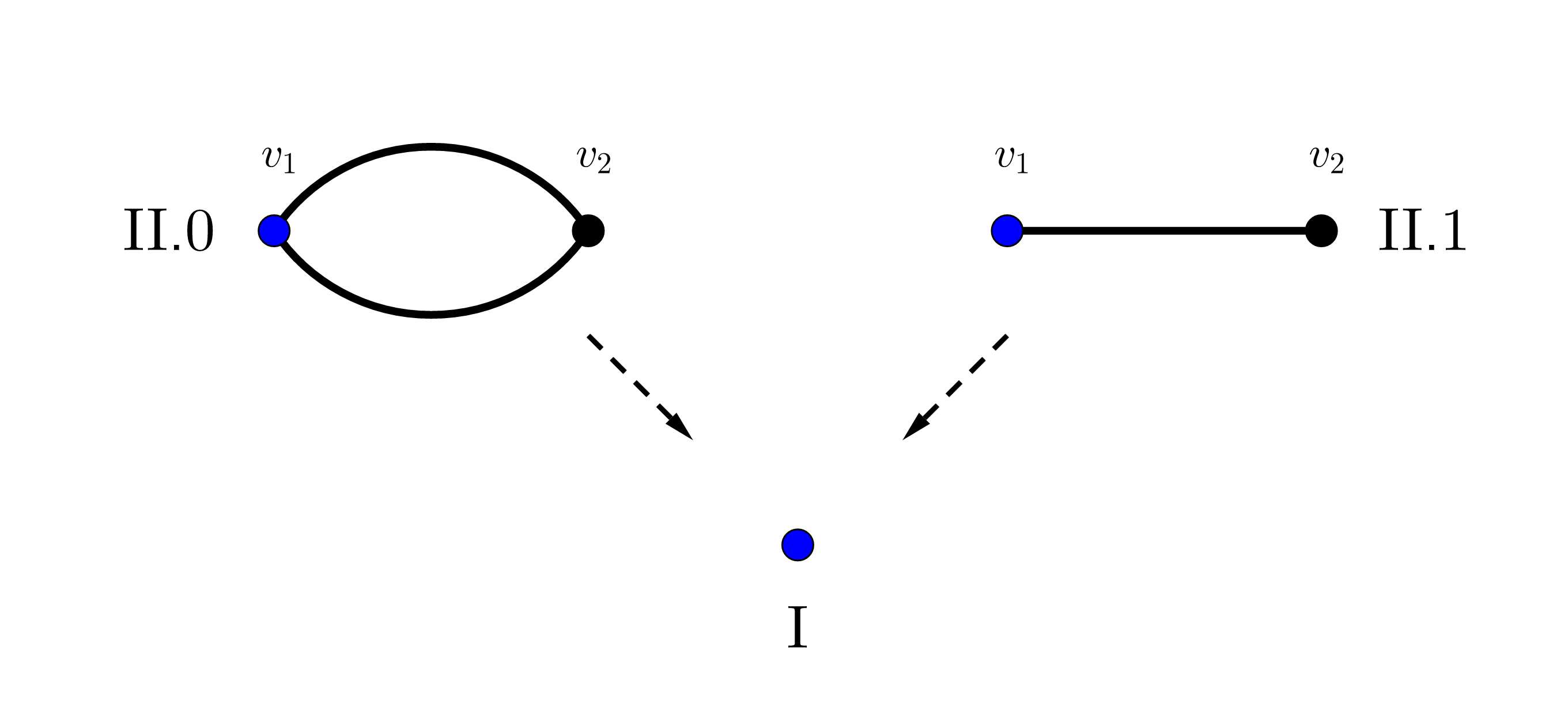}}
\caption{\label{GenusD3}
The reduction types for superelliptic curves $y^{n}=x^3-a_{1}x^2+a_{2}x-a_{3}$ as in Example \ref{ExampleD3}. The blue vertex denotes the vertex that a point over $\infty$ reduces to. The arrows correspond to contractions as in Figure \ref{TreeN3}. }
\end{figure}
The possible reduction types are as in Figure \ref{GenusD3} and the genera of the vertices and the edge lengths are given by the following table: 
 \begin{center}
 \begin{tabular}{|c| c|c|c|c|c| }
 \hline
 Reduction type & Additional congruence & $2g(v_{1})$ & $2g(v_{2})$ & $|D_{e}|$ & $\beta(\Sigma_{n})$  \\
 \hline
 I & -- & $2n-2$ & -- & -- & 0 \\ 
 \hline
 II.0 & $n\equiv{0}\bmod{2}$ & $n-2$ & $n-2$ & $n/2$ & 1 \\ 
 \hline
 II.1 & $n\equiv{1}\bmod{2}$ & $n-1$ & $n-1$ & $n$ & 0\\
 \hline
 \end{tabular}
 \end{center}

If $n\equiv{0}\bmod{3}$, then $g(X_{n})=n-2$. As before, there are three reduction types which can be found in Figure \ref{GenusD3}. The genera of the vertices and the edge lengths are given by the following table: 

 \begin{center}
 \begin{tabular}{|c| c|c|c|c|c| }
 \hline
 Reduction type & Additional congruence & $2g(v_{1})$ & $2g(v_{2})$ & $|D_{e}|$ & $\beta(\Sigma_{n})$  \\
 \hline
 I & -- & $2n-4$ & -- & -- & 0 \\ 
 \hline
 II.0 & $n\equiv{0}\bmod{2}$ & $n-4$ & $n-2$ & $n/2$ & 1 \\ 
 \hline
 II.1 & $n\equiv{1}\bmod{2}$ & $n-1$ & $n-3$ & $n$ & 0 \\
 \hline
 \end{tabular}
 \end{center}

To distinguish between I and II, we use the valuation of $j_{\mathrm{trop}}$ as in Section \ref{SimultaneousSemSta}. Note that if we use the classical $j$-invariant here, then we obtain the wrong result for fields of residue characteristic two. For instance, if we consider the curves $X_{n}$ 
for $n=1\bmod{2}$ and $f(x)\in\mathbf{C}_{2}[x]$, then $X_{n}$ has potential good reduction if and only if $v_{2}(j_{\mathrm{trop}})\geq{0}$. This criterion is not the same as having $v_{2}(j)\geq{0}$, since the 
$j$-invariant has additional factors of $2$, see Equation \ref{JInvariantModified}. 
For the bad characteristics (i.e., for primes dividing $n$), it might be that the invariants we use here can be modified in a similar way to obtain the right criteria. 


\end{exa}

For the upcoming cases we will only restrict to certain congruence subclass of $n$ to illustrate Theorem \ref{MainThm2}. We invite the reader to work out the remaining cases.

\begin{exa}\label{ExampleD4}
Let $X_{n}$ be as in Theorem \ref{MainThm2} with $d=4$, so $X_{n}$ is given by $y^{n}=f(x)=x^4-a_{1}x^3+a_{2}x^2-a_{3}x+a_{4}$. The slopes of $F=-\mathrm{log}|f|$ along the non-trivial line segments are as in the following table. Here the edges are as in Section \ref{SectionDegreeFour}. 

\begin{center}
\begin{tabular}{|c|c|c|}
\hline
Tree type & $\delta_{e_{1}}(F)$ & $\delta_{e_{2}}(F)$\\
\hline
I & - & - \\
\hline
II.1 & 3 & -\\
\hline
II.2 & 2 & - \\
\hline
III.1 & 3  & 2\\
\hline
III.2 & 2 & 2 \\
\hline
\end{tabular}
\end{center} 

For simplicity, we suppose that $n\not\equiv{0}\bmod{2}$. Then $g(X_{n})=\dfrac{3n-3}{2}$ by the Riemann-Hurwitz formula (note that $\infty$ is a branch point with $e_{\infty}=n$). The reduction types are given in Figure \ref{GenusD4}. 
\begin{figure}
\centering
\includegraphics[width=12cm]{{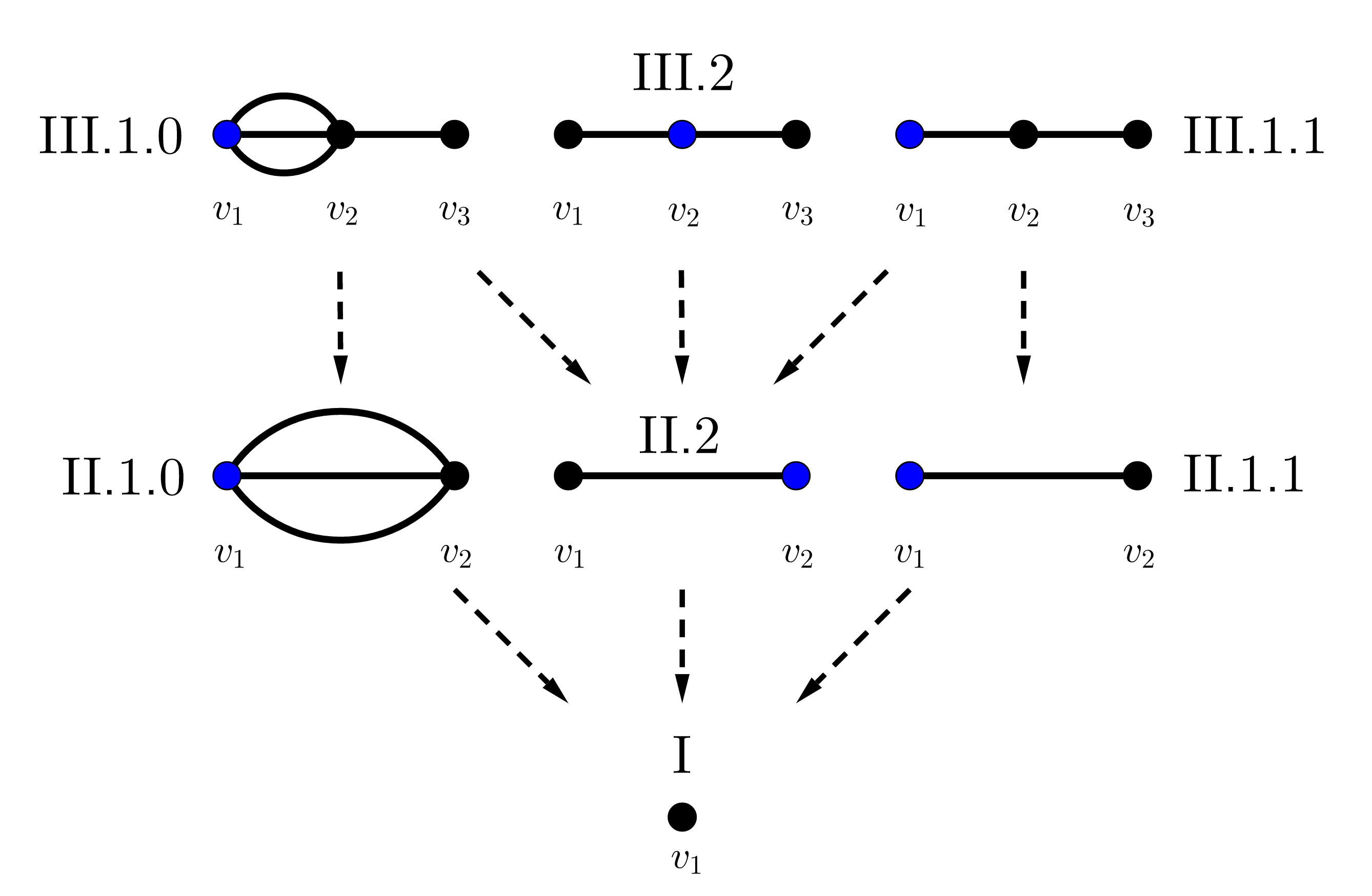}}
\caption{\label{GenusD4}
The reduction types for superelliptic curves $y^{n}=x^4-a_{1}x^3+a_{2}x^2-a_{3}x+a_{4}$ as in Example \ref{ExampleD4}. The blue vertex denotes the vertex that a point over $\infty$ reduces to.} 
\end{figure}
The genera of the vertices and the orders $|D_{e}|$ (which determine the lengths by Equation \ref{EdgeLengthsFormula}) are given in Table \ref{TableD4}. 
\begin{table}
 \begin{center}
 \begin{tabular}{|c|c||c|c|c||c|c||c|}
 \hline
 Reduction type & Congruence & $2g(v_{1})$ & $2g(v_{2})$ & $2g(v_{3})$ & $|D_{e_{1}}|$ & $|D_{e_{2}}|$ & $\beta(\Sigma_{n})$  \\
 \hline
 I & $n\equiv{1}\bmod{2}$ & ${3n-3}{}$ & -- & -- &  -- & -- & $0$\\ 
 \hline
 $\mathrm{II}.1.{0}$ & $n\equiv{3}\bmod{6}$ & ${n-3}{}$ & ${2n-4}{}$ & -- & {$n/3$} & -- & $2$ \\ 
 \hline
 $\mathrm{II}.1.{1}$ & $n\equiv{1,5}\bmod{6}$ & ${n-1}{}$ & ${2n-2}{}$ & -- & $n$& -- & $0$ \\
 \hline
 II.2 & $n\equiv{1}\bmod{2}$ & ${n-1}{}$ & ${2n-2}{}$ & -- & $n$ & --  & $0$ \\
 \hline
 III.1.0 & $n\equiv{3}\bmod{6}$ &${n-3}{}$ & ${n-3}{}$ & ${n-1}{}$ & {$n/3$} & $n$ & $2$ \\
 \hline
 III.1.1 & $n\equiv{1,5}\bmod{6}$ & ${n-1}{}$ & ${n-1}{}$ & ${n-1}{}$ & {$n$} & $n$ & $0$\\
 \hline
 III.2 & $n\equiv{1}\bmod{2}$ & ${n-1}{}$ & ${n-1}{}$ & ${n-1}{}$ &  $n$ & $n$ & $0$ \\
 \hline
 \end{tabular}
 \end{center}
 \caption{ \label{TableD4}The genera of the vertices and the expansion factors for the edges in Example \ref{ExampleD4}. The conditions for the tree types can be found in Example \ref{SectionDegreeFour}.   }
 \end{table}
If we plug in $n=3$, then we obtain a decomposition for {\it{Picard curves}}. The corresponding reduction types are on the lefthand side of Figure \ref{GenusD4}.  
\end{exa}

\newpage

\begin{exa}\label{ExampleD5}

Let $X_{n}$ be as in Theorem \ref{MainThm2} with $d=5$. In this case the reduction type depends on the image of $n$ in $\mathbf{Z}/60\mathbf{Z}$. We restrict ourselves to a subset of residue classes that contains hyperelliptic genus $2$ curves. That is, we suppose that $n\equiv{2}\bmod{4}$, $n\not\equiv{0}\bmod{3}$ and $n\not\equiv{0}\bmod{5}$. The latter condition implies that the covering $(x,y)\mapsto{x}$ ramifies completely over $\infty$.  
In Section \ref{SectionDegreeFive}, we saw that there are $7$ unmarked phylogenetic types, $12$ marked phylogenetic types and $18$ filtration types for these polynomials. It turns out that the reduction type does not depend on the marked point in this case. Our formulas for the edges however do depend on the marked point. If we use the leftmost vertex for our marked point, then the slopes of $F=-\mathrm{log}|f|$ are as follows:
\begin{center}
\begin{tabular}{|c|c|c|c|}
\hline
Tree type & $\delta_{e_{1}}(F)$ & $\delta_{e_{2}}(F)$ & $\delta_{e_{3}}(F)$\\
\hline
I & - & - & - \\
\hline
II & 3 & - & -\\
\hline
III & 2 & - & - \\
\hline
IV & 3  & 2 & -\\
\hline
V & 2 & 2 & -\\
\hline
VI & 4 & 3 & 2\\
\hline
VII & 4  & 2 & 2\\
\hline
\end{tabular}
\end{center} 
%


We have $g(X_{n})=2n-2$ by the Riemann-Hurwitz formula and the reduction types of the $X_{n}$ are as in Figure \ref{GenusD5}. 
\begin{figure}
\centering
\includegraphics[width=12cm]{{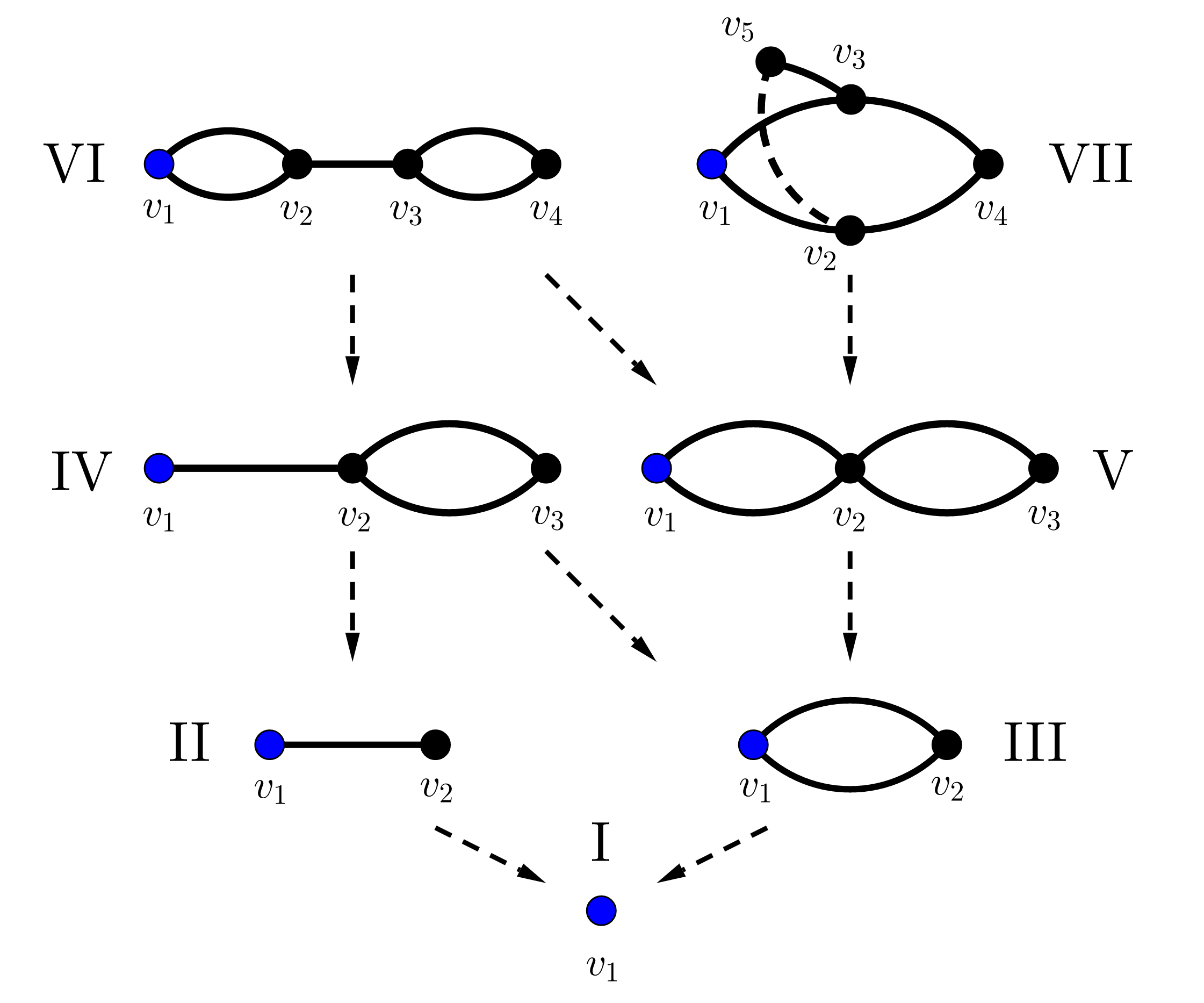}}
\caption{\label{GenusD5}
The reduction types for superelliptic curves $y^{n}=x^5-a_{1}x^4+a_{2}x^3-a_{3}x^2+a_{4}x-a_{5}$ as in Example \ref{ExampleD5}. The blue vertex corresponds to the marked point.  }
\end{figure}
The genera of the vertices, the orders $|D_{e}|$ and the first Betti number of $\Sigma(X_{n})$ are given by Table \ref{TableD5}. 
\begin{table}
 \begin{center}
 \begin{tabular}{|c||c|c|c|c|c||c|c|c||c| }
 \hline
 Reduction type & $2g(v_{1})$ & $2g(v_{2})$ & $2g(v_{3})$ & $2g(v_{4})$ & $2g(v_{5})$ & $|D_{e_{1}}|$ & $|D_{e_{2}}|$ & $|D_{e_{3}}|$ & $\beta(\Sigma_{n})$   \\
 \hline
 I & $4n-4$ & -- & -- & -- & -- & -- & -- & --  & $0$ \\
 \hline
   II& $2n-2$ & $2n-2$ & -- & --& -- & $n$ & -- & -- & $0$ \\
    \hline
    III & ${3n-4}$ & ${n-2}$ & -- & -- & -- & $n/2$ & -- & -- & $1$  \\
     \hline
     IV & $2n-2$  & $n-2$ & $n-2$ & -- & -- &  $n$ & $n/2$ & --& $1$ \\
      \hline
      V & $n-2$ & $2n-4$ & $n-2$ & -- & -- & $n/2$ & $n/2$ & --& $2$ \\
       \hline
       VI & $n-2$ & $n-2$ & $n-2$ & $n-2$ & -- & $n/2$ & $n$ & $n/2$ & $2$ \\
        \hline
        VII & $n-2$ & $n/2-1$ & $n/2-1$ & $n-2$ & $n-2$ & $n/2$ & $n/2$ & $n/2$ & $2$ \\
         \hline
 \end{tabular}

 \end{center}
   \caption{ \label{TableD5}The genera of the vertices and the expansion factors for the edge in Example \ref{ExampleD5}. The conditions for the tree types can be found in Example \ref{SectionDegreeFive}.   }
 \end{table}
If we use these formulas with $n=2$, then we a decomposition of the reduction types for genus $2$ curves. We invite the reader to compare this to \cite[Th\'{e}or\`{e}me 1]{liu}, where a similar decomposition using Igusa invariants is given. The edge lengths can be found in Proposition 2 of \cite{liu}. 
It is not clear to the author whether there is a direct connection between the invariants there and the ones in this paper. 
One thing that fails immediately is that the invariants given here are not all $\mathrm{SL}_{2}$-invariants, which implies that they are not in the algebra $\mathbf{Z}[J_{2},J_{4},J_{6},J_{8},J_{10}]$ generated by the Igusa invariants $J_{2i}$. 
The criteria given with the tropical invariants have two advantages over the Igusa invariants: first, 
they work for more general curves of the form $y^{n}=f(x)=x^{5}-a_{1}x^{4}+a_{2}x^{3}-a_{3}x^{2}+a_{4}x-a_{5}$. Furthermore, their definition makes the connection to the underlying tree of $f(x)$ clear. 
On the other hand, the formulas given in \cite{liu} also work in residue characteristic two, whereas the ones here do not.   
\end{exa}

\newpage

\section{Appendix}\label{Appendix}

We explicitly give the invariants that determine the tree type of polynomials of degree five here. To calculate these, one only needs an algorithm that can write symmetric polynomials in terms of elementary symmetric polynomials. We used the computer-algebra package Singular to do this.

\subsection{Polynomials of degree five}

 We write $f(x)=x^5-a_{1}x^4+a_{2}x^3-a_{3}x^2+a_{4}x-a_{5}$. In Section \ref{PolynomialsDegreeFive}, we used the two generating polynomials $F_{G_{2}}$ and $F_{G_{3}}$ corresponding to the complete graphs on two and three vertices respectively. The relevant coefficients of $F_{G_{2}}=\prod_{i=1}^{10}(x-[ij])$ are given by
\begin{align*}
c_{10,2}=&a_{1}^{2}a_{2}^{2}a_{3}^{2}a_{4}^{2}-4a_{2}^{3}a_{3}^{2}a_{4}^{2}-4a_{1}^{3}a
_{3}^{3}a_{4}^{2}+18a_{1}^{1}a_{2}^{1}a_{3}^{3}a_{4}^{2}-27a_{3}^{4}a_{4}^{2}-4
a_{1}^{2}a_{2}^{3}a_{4}^{3}+\\
& 16a_{2}^{4}a_{4}^{3}+18a_{1}^{3}a_{2}^{1}a_{3}^{1}a
_{4}^{3}-80a_{1}^{1}a_{2}^{2}a_{3}^{1}a_{4}^{3}-6a_{1}^{2}a_{3}^{2}a_{4}^{3}+14
4a_{2}^{1}a_{3}^{2}a_{4}^{3}-27a_{1}^{4}a_{4}^{4}+\\
&144a_{1}^{2}a_{2}^{1}a_{4}^{4
}-128a_{2}^{2}a_{4}^{4}-192a_{1}^{1}a_{3}^{1}a_{4}^{4}+256a_{4}^{5}-4a_{1}^{2}a
_{2}^{2}a_{3}^{3}a_{5}^{1}+\\
&16a_{2}^{3}a_{3}^{3}a_{5}^{1}+16a_{1}^{3}a_{3}^{4}a_
{5}^{1}-72a_{1}^{1}a_{2}^{1}a_{3}^{4}a_{5}^{1}+108a_{3}^{5}a_{5}^{1}+18a_{1}^{2
}a_{2}^{3}a_{3}^{1}a_{4}^{1}a_{5}^{1}-72a_{2}^{4}a_{3}^{1}a_{4}^{1}a_{5}^{1}-\\
&80
a_{1}^{3}a_{2}^{1}a_{3}^{2}a_{4}^{1}a_{5}^{1}+356a_{1}^{1}a_{2}^{2}a_{3}^{2}a_{
4}^{1}a_{5}^{1}+24a_{1}^{2}a_{3}^{3}a_{4}^{1}a_{5}^{1}-630a_{2}^{1}a_{3}^{3}a_{
4}^{1}a_{5}^{1}-6a_{1}^{3}a_{2}^{2}a_{4}^{2}a_{5}^{1}+\\
&24a_{1}^{1}a_{2}^{3}a_{4}
^{2}a_{5}^{1}+144a_{1}^{4}a_{3}^{1}a_{4}^{2}a_{5}^{1}-746a_{1}^{2}a_{2}^{1}a_{3
}^{1}a_{4}^{2}a_{5}^{1}+560a_{2}^{2}a_{3}^{1}a_{4}^{2}a_{5}^{1}+1020a_{1}^{1}a_
{3}^{2}a_{4}^{2}a_{5}^{1}-36a_{1}^{3}a_{4}^{3}a_{5}^{1}+\\
&160a_{1}^{1}a_{2}^{1}a_
{4}^{3}a_{5}^{1}-1600a_{3}^{1}a_{4}^{3}a_{5}^{1}-27a_{1}^{2}a_{2}^{4}a_{5}^{2}+
108a_{2}^{5}a_{5}^{2}+144a_{1}^{3}a_{2}^{2}a_{3}^{1}a_{5}^{2}-630a_{1}^{1}a_{2}
^{3}a_{3}^{1}a_{5}^{2}-128a_{1}^{4}a_{3}^{2}a_{5}^{2}+\\
&560a_{1}^{2}a_{2}^{1}a_{3
}^{2}a_{5}^{2}+825a_{2}^{2}a_{3}^{2}a_{5}^{2}-900a_{1}^{1}a_{3}^{3}a_{5}^{2}-19
2a_{1}^{4}a_{2}^{1}a_{4}^{1}a_{5}^{2}+1020a_{1}^{2}a_{2}^{2}a_{4}^{1}a_{5}^{2}-
900a_{2}^{3}a_{4}^{1}a_{5}^{2}+\\
&160a_{1}^{3}a_{3}^{1}a_{4}^{1}a_{5}^{2}-2050a_{1
}^{1}a_{2}^{1}a_{3}^{1}a_{4}^{1}a_{5}^{2}+2250a_{3}^{2}a_{4}^{1}a_{5}^{2}-50a_{
1}^{2}a_{4}^{2}a_{5}^{2}+\\
&2000a_{2}^{1}a_{4}^{2}a_{5}^{2}+256a_{1}^{5}a_{5}^{3}-
1600a_{1}^{3}a_{2}^{1}a_{5}^{3}+2250a_{1}^{1}a_{2}^{2}a_{5}^{3}+\\
&2000a_{1}^{2}a_
{3}^{1}a_{5}^{3}-3750a_{2}^{1}a_{3}^{1}a_{5}^{3}-2500a_{1}^{1}a_{4}^{1}a_{5}^{3
}+3125a_{5}^{4},\\
c_{9,2}=&-a_{1}^{2}a_{2}^{2}a_{3}^{4}+4a_{2}^{3}a_{3}^{4}+4a_{1}^{3}a_{3}^{5}-18a_{1}^
{1}a_{2}^{1}a_{3}^{5}+27a_{3}^{6}+6a_{1}^{2}a_{2}^{3}a_{3}^{2}a_{4}^{1}-\\
&24a_{2}
^{4}a_{3}^{2}a_{4}^{1}-26a_{1}^{3}a_{2}^{1}a_{3}^{3}a_{4}^{1}+116a_{1}^{1}a_{2}
^{2}a_{3}^{3}a_{4}^{1}+6a_{1}^{2}a_{3}^{4}a_{4}^{1}-198a_{2}^{1}a_{3}^{4}a_{4}^
{1}-\\
&9a_{1}^{2}a_{2}^{4}a_{4}^{2}+36a_{2}^{5}a_{4}^{2}+42a_{1}^{3}a_{2}^{2}a_{3}
^{1}a_{4}^{2}-186a_{1}^{1}a_{2}^{3}a_{3}^{1}a_{4}^{2}+18a_{1}^{4}a_{3}^{2}a_{4}
^{2}-\\
&114a_{1}^{2}a_{2}^{1}a_{3}^{2}a_{4}^{2}+434a_{2}^{2}a_{3}^{2}a_{4}^{2}+138
a_{1}^{1}a_{3}^{3}a_{4}^{2}-54a_{1}^{4}a_{2}^{1}a_{4}^{3}+282a_{1}^{2}a_{2}^{2}
a_{4}^{3}-224a_{2}^{3}a_{4}^{3}+\\
&18a_{1}^{3}a_{3}^{1}a_{4}^{3}-504a_{1}^{1}a_{2}
^{1}a_{3}^{1}a_{4}^{3}-40a_{3}^{2}a_{4}^{3}+72a_{1}^{2}a_{4}^{4}+320a_{2}^{1}a_
{4}^{4}-6a_{1}^{3}a_{2}^{2}a_{3}^{2}a_{5}^{1}+\\
&24a_{1}^{1}a_{2}^{3}a_{3}^{2}a_{5
}^{1}+32a_{1}^{4}a_{3}^{3}a_{5}^{1}-154a_{1}^{2}a_{2}^{1}a_{3}^{3}a_{5}^{1}+60a
_{2}^{2}a_{3}^{3}a_{5}^{1}+180a_{1}^{1}a_{3}^{4}a_{5}^{1}+\\
&18a_{1}^{3}a_{2}^{3}a
_{4}^{1}a_{5}^{1}-72a_{1}^{1}a_{2}^{4}a_{4}^{1}a_{5}^{1}-120a_{1}^{4}a_{2}^{1}a
_{3}^{1}a_{4}^{1}a_{5}^{1}+594a_{1}^{2}a_{2}^{2}a_{3}^{1}a_{4}^{1}a_{5}^{1}-330
a_{2}^{3}a_{3}^{1}a_{4}^{1}a_{5}^{1}-\\
&12a_{1}^{3}a_{3}^{2}a_{4}^{1}a_{5}^{1}-550
a_{1}^{1}a_{2}^{1}a_{3}^{2}a_{4}^{1}a_{5}^{1}-450a_{3}^{3}a_{4}^{1}a_{5}^{1}+21
6a_{1}^{5}a_{4}^{2}a_{5}^{1}-\\
&1230a_{1}^{3}a_{2}^{1}a_{4}^{2}a_{5}^{1}+1380a_{1}
^{1}a_{2}^{2}a_{4}^{2}a_{5}^{1}+1170a_{1}^{2}a_{3}^{1}a_{4}^{2}a_{5}^{1}+200a_{
2}^{1}a_{3}^{1}a_{4}^{2}a_{5}^{1}-2000a_{1}^{1}a_{4}^{3}a_{5}^{1}+72a_{1}^{4}a_
{2}^{2}a_{5}^{2}-\\
&450a_{1}^{2}a_{2}^{3}a_{5}^{2}+675a_{2}^{4}a_{5}^{2}-192a_{1}^
{5}a_{3}^{1}a_{5}^{2}+1320a_{1}^{3}a_{2}^{1}a_{3}^{1}a_{5}^{2}-2100a_{1}^{1}a_{
2}^{2}a_{3}^{1}a_{5}^{2}-1150a_{1}^{2}a_{3}^{2}a_{5}^{2}+\\
&3000a_{2}^{1}a_{3}^{2}
a_{5}^{2}-240a_{1}^{4}a_{4}^{1}a_{5}^{2}+1950a_{1}^{2}a_{2}^{1}a_{4}^{1}a_{5}^{
2}-3750a_{2}^{2}a_{4}^{1}a_{5}^{2}-\\
& 250a_{1}^{1}a_{3}^{1}a_{4}^{1}a_{5}^{2}+5000
a_{4}^{2}a_{5}^{2}-1000a_{1}^{3}a_{5}^{3}+3750a_{1}^{1}a_{2}^{1}a_{5}^{3}-6250a
_{3}^{1}a_{5}^{3},\\
c_{8,2}=&a_{1}^{2}a_{2}^{4}a_{3}^{2}-8a_{2}^{5}a_{3}^{2}-12a_{1}^{3}a_{2}^{2}a_{3}^{
3}+52a_{1}^{1}a_{2}^{3}a_{3}^{3}+\\
& 17a_{1}^{4}a_{3}^{4}-78a_{1}^{2}a_{2}^{1}a_{3}
^{4}-45a_{2}^{2}a_{3}^{4}+108a_{1}^{1}a_{3}^{5}-6a_{1}^{2}a_{2}^{5}a_{4}^{1}+24
a_{2}^{6}a_{4}^{1}+38a_{1}^{3}a_{2}^{3}a_{3}^{1}a_{4}^{1}-164a_{1}^{1}a_{2}^{4}
a_{3}^{1}a_{4}^{1}-54a_{1}^{4}a_{2}^{1}a_{3}^{2}a_{4}^{1}+\\
&246a_{1}^{2}a_{2}^{2}
a_{3}^{2}a_{4}^{1}+140a_{2}^{3}a_{3}^{2}a_{4}^{1}-14a_{1}^{3}a_{3}^{3}a_{4}^{1}
-240a_{1}^{1}a_{2}^{1}a_{3}^{3}a_{4}^{1}-270a_{3}^{4}a_{4}^{1}-\\
&25a_{1}^{4}a_{2}
^{2}a_{4}^{2}+112a_{1}^{2}a_{2}^{3}a_{4}^{2}-15a_{2}^{4}a_{4}^{2}+48a_{1}^{5}a_
{3}^{1}a_{4}^{2}-178a_{1}^{3}a_{2}^{1}a_{3}^{1}a_{4}^{2}-390a_{1}^{1}a_{2}^{2}a
_{3}^{1}a_{4}^{2}+\\
& 275a_{1}^{2}a_{3}^{2}a_{4}^{2}+960a_{2}^{1}a_{3}^{2}a_{4}^{2}
-48a_{1}^{4}a_{4}^{3}+400a_{1}^{2}a_{2}^{1}a_{4}^{3}-360a_{2}^{2}a_{4}^{3}-840a
_{1}^{1}a_{3}^{1}a_{4}^{3}+400a_{4}^{4}+6a_{1}^{3}a_{2}^{4}a_{5}^{1}-24a_{1}^{1
}a_{2}^{5}a_{5}^{1}-\\
&32a_{1}^{4}a_{2}^{2}a_{3}^{1}a_{5}^{1}+130a_{1}^{2}a_{2}^{3
}a_{3}^{1}a_{5}^{1}+40a_{2}^{4}a_{3}^{1}a_{5}^{1}+24a_{1}^{5}a_{3}^{2}a_{5}^{1}
-50a_{1}^{3}a_{2}^{1}a_{3}^{2}a_{5}^{1}-420a_{1}^{1}a_{2}^{2}a_{3}^{2}a_{5}^{1}
+70a_{1}^{2}a_{3}^{3}a_{5}^{1}+\\
&600a_{2}^{1}a_{3}^{3}a_{5}^{1}+56a_{1}^{5}a_{2}^
{1}a_{4}^{1}a_{5}^{1}-310a_{1}^{3}a_{2}^{2}a_{4}^{1}a_{5}^{1}+290a_{1}^{1}a_{2}
^{3}a_{4}^{1}a_{5}^{1}-200a_{1}^{4}a_{3}^{1}a_{4}^{1}a_{5}^{1}+\\
&1740a_{1}^{2}a_{
2}^{1}a_{3}^{1}a_{4}^{1}a_{5}^{1}-1650a_{2}^{2}a_{3}^{1}a_{4}^{1}a_{5}^{1}-1550
a_{1}^{1}a_{3}^{2}a_{4}^{1}a_{5}^{1}-\\
&640a_{1}^{3}a_{4}^{2}a_{5}^{1}+1400a_{1}^{
1}a_{2}^{1}a_{4}^{2}a_{5}^{1}+\\
&1000a_{3}^{1}a_{4}^{2}a_{5}^{1}-112a_{1}^{6}a_{5}
^{2}+840a_{1}^{4}a_{2}^{1}a_{5}^{2}-2025a_{1}^{2}a_{2}^{2}a_{5}^{2}+\\
&1750a_{2}^{
3}a_{5}^{2}-200a_{1}^{3}a_{3}^{1}a_{5}^{2}-250a_{1}^{1}a_{2}^{1}a_{3}^{1}a_{5}^
{2}+1875a_{3}^{2}a_{5}^{2}+2000a_{1}^{2}a_{4}^{1}a_{5}^{2}-5000a_{2}^{1}a_{4}^{
1}a_{5}^{2},\\
\end{align*}
\begin{align*}
c_{7,2}=&-a_{1}^{2}a_{2}^{6}+4a_{2}^{7}+8a_{1}^{3}a_{2}^{4}a_{3}^{1}-34a_{1}^{1}a_{2}^
{5}a_{3}^{1}-24a_{1}^{4}a_{2}^{2}a_{3}^{2}+
 112a_{1}^{2}a_{2}^{3}a_{3}^{2}-7a_{2
}^{4}a_{3}^{2}+\\
&28a_{1}^{5}a_{3}^{3}-148a_{1}^{3}a_{2}^{1}a_{3}^{3}+160a_{1}^{2}
a_{3}^{4}-102a_{2}^{1}a_{3}^{4}+
8a_{1}^{4}a_{2}^{3}a_{4}^{1}-56a_{1}^{2}a_{2}^{
4}a_{4}^{1}+106a_{2}^{5}a_{4}^{1}-30a_{1}^{5}a_{2}^{1}a_{3}^{1}a_{4}^{1}+\\
& 234a_{
1}^{3}a_{2}^{2}a_{3}^{1}a_{4}^{1}-504a_{1}^{1}a_{2}^{3}a_{3}^{1}a_{4}^{1}-
 54a_{
1}^{4}a_{3}^{2}a_{4}^{1}+150a_{1}^{2}a_{2}^{1}a_{3}^{2}a_{4}^{1}+\\
& 612a_{2}^{2}a_
{3}^{2}a_{4}^{1}-596a_{1}^{1}a_{3}^{3}a_{4}^{1}+26a_{1}^{6}a_{4}^{2}-178a_{1}^{
4}a_{2}^{1}a_{4}^{2}+388a_{1}^{2}a_{2}^{2}a_{4}^{2}-\\
& 308a_{2}^{3}a_{4}^{2}+92a_{
1}^{3}a_{3}^{1}a_{4}^{2}-158a_{1}^{1}a_{2}^{1}a_{3}^{1}a_{4}^{2}+\\
& 570a_{3}^{2}a_
{4}^{2}-118a_{1}^{2}a_{4}^{3}-80a_{2}^{1}a_{4}^{3}-6a_{1}^{5}a_{2}^{2}a_{5}^{1}
+34a_{1}^{3}a_{2}^{3}a_{5}^{1}-48a_{1}^{1}a_{2}^{4}a_{5}^{1}+\\
&16a_{1}^{6}a_{3}^{
1}a_{5}^{1}-86a_{1}^{4}a_{2}^{1}a_{3}^{1}a_{5}^{1}+86a_{1}^{2}a_{2}^{2}a_{3}^{1
}a_{5}^{1}+80a_{2}^{3}a_{3}^{1}a_{5}^{1}+\\
&112a_{1}^{3}a_{3}^{2}a_{5}^{1}-400a_{1
}^{1}a_{2}^{1}a_{3}^{2}a_{5}^{1}+700a_{3}^{3}a_{5}^{1}-68a_{1}^{5}a_{4}^{1}a_{5
}^{1}+180a_{1}^{3}a_{2}^{1}a_{4}^{1}a_{5}^{1}+290a_{1}^{1}a_{2}^{2}a_{4}^{1}a_{
5}^{1}-40a_{1}^{2}a_{3}^{1}a_{4}^{1}a_{5}^{1}-\\
&2150a_{2}^{1}a_{3}^{1}a_{4}^{1}a_
{5}^{1}+1500a_{1}^{1}a_{4}^{2}a_{5}^{1}+490a_{1}^{4}a_{5}^{2}-2450a_{1}^{2}a_{2
}^{1}a_{5}^{2}+2500a_{2}^{2}a_{5}^{2}+1500a_{1}^{1}a_{3}^{1}a_{5}^{2}-3750a_{4}
^{1}a_{5}^{2},\\
c_{6,2}=&3a_{1}^{4}a_{2}^{4}-18a_{1}^{2}a_{2}^{5}+25a_{2}^{6}-16a_{1}^{5}a_{2}^{2}a_{3
}^{1}+100a_{1}^{3}a_{2}^{3}a_{3}^{1}-\\
& 140a_{1}^{1}a_{2}^{4}a_{3}^{1}+22a_{1}^{6}
a_{3}^{2}-144a_{1}^{4}a_{2}^{1}a_{3}^{2}+174a_{1}^{2}a_{2}^{2}a_{3}^{2}+\\
& 52a_{2}
^{3}a_{3}^{2}+92a_{1}^{3}a_{3}^{3}-156a_{1}^{1}a_{2}^{1}a_{3}^{3}-53a_{3}^{4}-2
a_{1}^{6}a_{2}^{1}a_{4}^{1}+\\
& 32a_{1}^{4}a_{2}^{2}a_{4}^{1}-152a_{1}^{2}a_{2}^{3}
a_{4}^{1}+194a_{2}^{4}a_{4}^{1}-42a_{1}^{5}a_{3}^{1}a_{4}^{1}+330a_{1}^{3}a_{2}
^{1}a_{3}^{1}a_{4}^{1}-480a_{1}^{1}a_{2}^{2}a_{3}^{1}a_{4}^{1}-\\
& 378a_{1}^{2}a_{3
}^{2}a_{4}^{1}+708a_{2}^{1}a_{3}^{2}a_{4}^{1}-84a_{1}^{4}a_{4}^{2}+388a_{1}^{2}
a_{2}^{1}a_{4}^{2}-522a_{2}^{2}a_{4}^{2}+142a_{1}^{1}a_{3}^{1}a_{4}^{2}+\\
& 40a_{4}
^{3}+8a_{1}^{7}a_{5}^{1}-70a_{1}^{5}a_{2}^{1}a_{5}^{1}+194a_{1}^{3}a_{2}^{2}a_{
5}^{1}-144a_{1}^{1}a_{2}^{3}a_{5}^{1}+66a_{1}^{4}a_{3}^{1}a_{5}^{1}-\\
& 394a_{1}^{2
}a_{2}^{1}a_{3}^{1}a_{5}^{1}+240a_{2}^{2}a_{3}^{1}a_{5}^{1}+380a_{1}^{1}a_{3}^{
2}a_{5}^{1}+128a_{1}^{3}a_{4}^{1}a_{5}^{1}-\\
& 130a_{1}^{1}a_{2}^{1}a_{4}^{1}a_{5}^
{1}-950a_{3}^{1}a_{4}^{1}a_{5}^{1}-700a_{1}^{2}a_{5}^{2}+1750a_{2}^{1}a_{5}^{2},\\
c_{5,2}=&-3a_{1}^{6}a_{2}^{2}+24a_{1}^{4}a_{2}^{3}-66a_{1}^{2}a_{2}^{4}+66a_{2}^{5}+8a
_{1}^{7}a_{3}^{1}-66a_{1}^{5}a_{2}^{1}a_{3}^{1}+\\
& 192a_{1}^{3}a_{2}^{2}a_{3}^{1}-
220a_{1}^{1}a_{2}^{3}a_{3}^{1}-3a_{1}^{4}a_{3}^{2}+\\
& 20a_{1}^{2}a_{2}^{1}a_{3}^{2
}+118a_{2}^{2}a_{3}^{2}-104a_{1}^{1}a_{3}^{3}-\\
& 8a_{1}^{6}a_{4}^{1}+66a_{1}^{4}a_
{2}^{1}a_{4}^{1}-180a_{1}^{2}a_{2}^{2}a_{4}^{1}+196a_{2}^{3}a_{4}^{1}-10a_{1}^{
3}a_{3}^{1}a_{4}^{1}-104a_{1}^{1}a_{2}^{1}a_{3}^{1}a_{4}^{1}+\\
& 260a_{3}^{2}a_{4}^
{1}+169a_{1}^{2}a_{4}^{2}-360a_{2}^{1}a_{4}^{2}-24a_{1}^{5}a_{5}^{1}+\\
& 150a_{1}^{
3}a_{2}^{1}a_{5}^{1}-240a_{1}^{1}a_{2}^{2}a_{5}^{1}-110a_{1}^{2}a_{3}^{1}a_{5}^
{1}+400a_{2}^{1}a_{3}^{1}a_{5}^{1}-250a_{1}^{1}a_{4}^{1}a_{5}^{1}+625a_{5}^{2},\\
c_{4,2}=&a_{1}^{8}-10a_{1}^{6}a_{2}^{1}+45a_{1}^{4}a_{2}^{2}-104a_{1}^{2}a_{2}^{3}+
  95a_{2}^{4}-20a_{1}^{5}a_{3}^{1}+\\
  & 116a_{1}^{3}a_{2}^{1}a_{3}^{1}-160a_{1}^{1}a_{2}^
{2}a_{3}^{1}-
  44a_{1}^{2}a_{3}^{2}+
   92a_{2}^{1}a_{3}^{2}+18a_{1}^{4}a_{4}^{1}-\\
   & 98a
_{1}^{2}a_{2}^{1}a_{4}^{1}+
  124a_{2}^{2}a_{4}^{1}+
  36a_{1}^{1}a_{3}^{1}a_{4}^{1}-
95a_{4}^{2}
  +32a_{1}^{3}a_{5}^{1}-\\
  & 120a_{1}^{1}a_{2}^{1}a_{5}^{1}+200a_{3}^{1}a_{
5}^{1}.
\end{align*}

For the last set of invariants, we write $F_{G_{3}}=\prod{(x-\sigma_{i}([12][13][23]))}=\sum_{i=0}^{10}c_{10-i,3}x^{i}$, where the product is over a set of representatives $\sigma_{i}$ of $S_{5}/H_{G_{3},1}$ in $S_{5}$. To distinguish between the various filtration types, we need the coefficients $c_{4,3}, c_{6,3}$ and $c_{7,3}$. These consist of $955$, $284$ and $123$ monomials respectively, they can be found by a Gr\"{o}bner basis calculation. The interested reader can find the corresponding Singular code with output on \href{https://paulhelminck.files.wordpress.com/2020/12/singularcodeinvariants3.pdf}{this website}.

\begin{center}
\bibliographystyle{alpha}
\bibliography{bibfiles}{}

\newcommand{\etalchar}[1]{$^{#1}$}
\begin{thebibliography}{BCK{\etalchar{+}}20}

\bibitem[ABBR15]{ABBR1}
Omid Amini, Matthew Baker, Erwan Brugall{\'{e}}, and Joseph Rabinoff.
\newblock Lifting harmonic morphisms {I}: {M}etrized complexes and {B}erkovich
  skeleta.
\newblock {\em Springer, Research in the Mathematical Sciences}, 2(1), June
  2015.

\bibitem[ACP15]{ACP2015}
Dan Abramovich, Lucia Caporaso, and Sam Payne.
\newblock The tropicalization of the moduli space of curves.
\newblock {\em Annales scientifiques de l'{\'{E}}cole normale
  sup{\'{e}}rieure}, 48(4):765--809, 2015.

\bibitem[Bak08]{Baker2008}
Matthew Baker.
\newblock Specialization of linear systems from curves to graphs.
\newblock {\em Algebra {\&} Number Theory}, 2(6):613--653, October 2008.

\bibitem[BCK{\etalchar{+}}20]{WomeninNTReductionGen3}
Irene Bouw, Nirvana Coppola, Pinar Kili\c{c}er, Sabrina Kunzweiler,
  Elisa~Lorenzo Garc\`{i}a, and Anna Somoza.
\newblock Reduction type of genus-3 curves in a special stratum of their moduli
  space.
\newblock \href{https://arxiv.org/pdf/2003.07633.pdf}{arXiv:2003.07633}, 2020.

\bibitem[Ber90]{berkovich2012}
Vladimir~G. Berkovich.
\newblock {\em Spectral Theory and Analytic Geometry over Non-Archimedean
  Fields}.
\newblock Mathematical Surveys and Monographs. American Mathematical Society,
  1990.

\bibitem[Ber93]{Berkovich1993}
Vladimir~G. Berkovich.
\newblock \'{E}tale cohomology for non-{A}rchimedean analytic spaces.
\newblock {\em Publications {M}ath\'{e}matiques de l'IH\'{E}S}, 78:5--161,
  1993.

\bibitem[BH20]{supertrop}
Madeline Brandt and Paul~Alexander Helminck.
\newblock Tropical superelliptic curves.
\newblock {\em Advances in Geometry}, 20(4):527--551, October 2020.

\bibitem[BPR14]{BPRa1}
Matthew Baker, Sam Payne, and Joseph Rabinoff.
\newblock On the structure of nonarchimedean analytic curves.
\newblock In {\em Tropical and Non-Archimedean Geometry}, volume 605, pages pp.
  93--121. American Mathematical Society, 2014.

\bibitem[BW17]{bouw_wewers_2017}
Irene~I. Bouw and Stefan Wewers.
\newblock Computing {L}-functions and semistable reduction of superelliptic
  curves.
\newblock {\em Glasgow Mathematical Journal}, 59(1):77--108, 2017.

\bibitem[DK02]{derksen2002computational}
Harm Derksen and Gregor Kemper.
\newblock {\em Computational invariant theory}.
\newblock Springer, Berlin New York, 2002.

\bibitem[GMR71]{SGA1}
Alexander {G}rothendieck and {M}ich{\`{e}}le {R}aynaud.
\newblock {\em Rev{\^{e}}tements {E}tales et {G}roupe {F}ondamental}.
\newblock Springer {B}erlin {H}eidelberg, 1971.

\bibitem[Hel16]{Igusa}
Paul~Alexander Helminck.
\newblock Tropical {I}gusa invariants and torsion embeddings.
\newblock {\em \href{https://arxiv.org/abs/1604.03987}{arXiv:1604.03987}},
  2016.

\bibitem[Hel18a]{tropabelian}
Paul~Alexander Helminck.
\newblock Tropicalizing abelian covers of algebraic curves.
\newblock {\em \href{https://arxiv.org/abs/1703.03067}{arXiv:1703.03067v2}, PhD
  Thesis, Universit{\"a}t Bremen}, 11-07-2018.

\bibitem[Hel18b]{TropFund1}
Paul~Alexander Helminck.
\newblock Tropical decompositions of the fundamental group of a non-archimedean
  curve.
\newblock \href{https://arxiv.org/abs/1808.03541}{arXiv:1808.03541}, 2018.

\bibitem[Hel19]{FaithTropEll}
Paul~Alexander Helminck.
\newblock Faithful tropicalizations of elliptic curves using minimal models and
  inflection points.
\newblock {\em Arnold Mathematical Journal}, 5(4):401--434, 2019.

\bibitem[Hel20]{NewtonPuiseuxSemSta}
Paul~Alexander Helminck.
\newblock A generalization of the {N}ewton-{P}uiseux algorithm for semistable
  models.
\newblock \href{https://arxiv.org/abs/2007.09449}{arXiv:2007.09449}, 2020.

\bibitem[Kap93]{Kapranov1993}
Mikhail Kapranov.
\newblock Chow {Q}uotients of {G}rassmannians {I}.
\newblock In {\em I. M. {G}elfand {S}eminar}. American {M}athematical
  {S}ociety, aug 1993.

\bibitem[KS91]{KapranovSturmfels1991}
Mikhail Kapranov and Bernd Sturmfels.
\newblock Quotients of toric varieties.
\newblock {\em Mathematische Annalen}, 290(4):643--656, 1991.

\bibitem[Liu93]{liu}
Qing Liu.
\newblock Courbes stables de genre 2 et leur sch\'{e}ma de modules.
\newblock {\em Mathematische Annalen}, 295(2):201--222, 1993.

\bibitem[Liu06]{liu2}
Qing Liu.
\newblock {\em Algebraic Geometry and Arithmetic Curves}.
\newblock Oxford Graduate Texts in Mathematics (Book 6). Oxford University
  Press, 2006.

\bibitem[LL99]{liu1}
Qing Liu and Dino Lorenzini.
\newblock Models of curves and finite covers.
\newblock {\em Compositio Mathematica}, 118(1):61--102, 1999.

\bibitem[LLGR18]{LiuSmoothPlaneQuartics}
Reynald Lercier, Qing Liu, Elisa~Lorenzo Garc\`{i}a, and Christophe
  Ritzenthaler.
\newblock Reduction type of smooth quartics.
\newblock \href{https://arxiv.org/pdf/1803.05816.pdf}{arXiv:1803.05816}, 2018.

\bibitem[MS15]{tropicalbook}
Diane Maclagan and Bernd Sturmfels.
\newblock {\em Introduction to Tropical Geometry}.
\newblock Graduate Studies in Mathematics. American Mathematical Society, 2015.

\bibitem[N{\'e}r64]{Neron2}
Andr{\'e} N{\'e}ron.
\newblock Mod{\`e}les minimaux des vari{\'e}t{\'e}s ab{\'e}liennes sur les
  corps locaux et globaux.
\newblock {\em Publications Math{\'e}matiques de l'Institut des Hautes
  {\'E}tudes Scientifiques}, 21(1):5--125, 1964.

\bibitem[Neu99]{neu}
J\"{u}rgen Neukirch.
\newblock {\em Algebraic Number Theory}.
\newblock Grundlehren der mathematischen Wissenschaften. Springer Berlin
  Heidelberg, 1999.

\bibitem[Ses77]{SESHADRI1977}
Conjeevaram~S. Seshadri.
\newblock Geometric reductivity over arbitrary base.
\newblock {\em Advances in Mathematics}, 26(3):225 -- 274, 1977.

\bibitem[Stu08]{sturmfels2008algorithms}
Bernd Sturmfels.
\newblock {\em Algorithms in invariant theory}.
\newblock Springer-Verlag, Wien New York, 2008.

\bibitem[Swa69]{SwanToric1969}
Richard~G. Swan.
\newblock Invariant rational functions and a problem of {S}teenrod.
\newblock {\em Inventiones mathematicae}, 7(2):148--158, 1969.

\end{thebibliography}
\end{center}

\end{document}